\documentclass[hidelinks,onefignum,onetabnum]{siamart220329}

\usepackage[utf8]{inputenc}
\usepackage{url}
\usepackage{amsmath,amsfonts, amssymb,MnSymbol,multirow, graphicx}
\usepackage[caption=false]{subfig} 
\usepackage{algorithm,algpseudocode}
\usepackage{booktabs}
\newtheorem{remark}{Remark}
\usepackage{xcolor}
\usepackage[mathscr]{eucal}
\usepackage{amsbsy,bm}
\usepackage{amssymb}
\usepackage{hyperref}
\usepackage{tikz}
\usepackage{pgfplots}
\pgfplotsset{compat=1.18}
\usepackage{pgfplotstable}
\usepackage{calc}
\usetikzlibrary{calc,arrows}
\usetikzlibrary{decorations.pathreplacing}
\usepackage[T1]{fontenc}

\newcommand{\T}[2][]{\boldsymbol{#1\mathscr{\MakeUppercase{#2}}}}

\newcommand{\TT}[2]{\T{T}_{\T{#1},#2}} 
\newcommand{\slice}[2]{} 
\newcommand{\TTslice}[2]{ 
\ifx1#2  
	\renewcommand{\slice}[2]{\TT{#1}{1}(i_1,:)}
\else \ifx#2N 
	\renewcommand{\slice}[2]{\TT{#1}{N}(:,i_N)}
\else 
	\renewcommand{\slice}[2]{\TT{#1}{#2}(:,i_{#2},:)}
\fi \fi
\slice{#1}{#2} 
}

\newcommand{\R}{\mathbb{R}}

\newtheorem{thm}{Theorem}

\usepackage{ifthen}
\usepackage{float}
\floatstyle{ruled}
\newfloat{algorithm}{htb}{alg}
\floatname{algorithm}{Algorithm}
\newcounter{algo@row}
\newcounter{algo@rowindent}
\newcommand{\algofont}[1]{\textbf{#1}}
\newcommand{\algonumbersize}[1]{\scriptsize{#1}}
\newcommand{\algopreitem}[1][\arabic{algo@row}]{\texttt{\algonumbersize{#1}}}
\newcommand{\algoitemskip}{\hspace{\value{algo@rowindent}cc}}
\newcommand{\algonewnestedopen}[2]{
	\newcommand{#1}[1][]{%
		\ifthenelse{\equal{##1}{}}{\item}{\item[{\algopreitem[##1]}]}
		\algoitemskip\algofont{#2}%
		\addtocounter{algo@rowindent}{1}%
		\ignorespaces
	}
}
\newcommand{\algonewnestedaux}[2]{
	\newcommand{#1}[1][]{
		\addtocounter{algo@rowindent}{-1}
		\ifthenelse{\equal{##1}{}}{\item}{\item[{\algopreitem[##1]}]}
		\algoitemskip\algofont{#2}%
		\addtocounter{algo@rowindent}{+1}%
		\ignorespaces
	}
}
\newcommand{\algonewnestedclose}[2]{
	\newcommand{#1}[1][]{
		\addtocounter{algo@rowindent}{-1}
		\ifthenelse{\equal{##1}{}}{\item}{\item[{\algopreitem[##1]}]}
		\algoitemskip\algofont{#2}%
		\ignorespaces
	}
}
\newcommand{\algonewcommand}[2]{
	\newcommand{#1}[1][default]{
		\ifthenelse{\equal{##1}{default}}{\item}{\item[{\algopreitem[##1]}]}%
		\algoitemskip\algofont{#2}%
		\ignorespaces
	}%
}
\newcommand{\algonewkeyword}[2]{\newcommand{#1}{\algofont{#2}}}
\algonewcommand{\STATE}{\ignorespaces}
\algonewcommand{\INPUT}{Input: }
\algonewcommand{\pINPUT}{\phantom{Input: }}
\algonewcommand{\COMPUTE}{Compute: }
\algonewcommand{\OUTPUT}{Output: }
\algonewcommand{\pOUTPUT}{\phantom{Output: }}

\algonewnestedopen{\IF}{if }
\algonewnestedaux{\ELSEIF}{else if }
\algonewnestedaux{\ELSE}{else }
\algonewnestedclose{\ENDIF}{end if }
\algonewnestedopen{\FOR}{for }
\algonewnestedclose{\ENDFOR}{end for }
\algonewnestedopen{\WHILE}{while }
\algonewnestedclose{\ENDWHILE}{end while }
\algonewcommand{\BREAK}{break}%

\algonewkeyword{\To}{to }%
\algonewkeyword{\Do}{do }%
\algonewkeyword{\Then}{then }%
\algonewkeyword{\End}{end }%
\algonewkeyword{\AND}{and }%
\algonewkeyword{\True}{true }%
\algonewkeyword{\False}{false }%
\algonewkeyword{\irbleigs}{irbleigs }%
\algonewkeyword{\tridiag}{tridiag}%
\algonewkeyword{\reorth}{reorth}%

\newcommand{\bA}{{\bf A}}
\newcommand{\bB}{{\bf B}}

\newcommand{\bH}{{\bf H}}
\newcommand{\bI}{{\bf I}}

\newcommand{\bL}{{\bf L}}

\newcommand{\bP}{{\bf P}}
\newcommand{\bQ}{{\bf Q}}
\newcommand{\bR}{{\bf R}}
\newcommand{\bS}{{\bf S}}
\newcommand{\bT}{{\bf T}}
\newcommand{\bU}{{\bf U}}
\newcommand{\bV}{{\bf V}}
\newcommand{\bW}{{\bf W}}

\newcommand{\bb}{{\bf b}}

\newcommand{\bd}{{\bf d}}
\newcommand{\be}{{\bf e}}

\newcommand{\br}{{\bf r}}

\newcommand{\bu}{{\bf u}}
\newcommand{\bv}{{\bf v}}
\newcommand{\bw}{{\bf w}}
\newcommand{\bx}{{\bf x}}
\newcommand{\by}{{\bf y}}
\newcommand{\bz}{{\bf z}}

\usepackage{gensymb}

\newcommand{\bPsi}{{\boldsymbol{\Psi}}}

\newcommand{\bepsilon}{{\boldsymbol{\epsilon}}}

\usepackage{amsmath}
\usepackage{amssymb}
\usepackage[utf8]{inputenc}

\newcommand{\calJ}{\mathcal{J}}

\newcommand{\calQ}{\mathcal{Q}}

\newcommand{\eqs}{\begin{eqnarray}}
\newcommand{\eqe}{\end{eqnarray}}
\newcommand{\eqsn}{\begin{eqnarray*}}
\newcommand{\eqen}{\end{eqnarray*}}
\newcommand{\Ra}[1]{\text{range}(#1)}

\usepackage{optidef}
\usepackage{subfig}
\title{A provably convergent MM-GKS variant for large-scale inverse problems}
\author{Mirjeta Pasha
\thanks{Department of Mathematics, Virginia Tech, Blacksburg, VA, USA.
(\email{mpasha@vt.edu}), (\email{sturler@vt.edu})}
\and Eric de Sturler 
\footnotemark[1]
\and Misha E. Kilmer 
\thanks{Department of Mathematics, Tufts University, Medford, MA, USA. (\email{misha.kilmer@tufts.edu})}}

\ifpdf
\hypersetup{
  pdftitle={Limited Memory MM-GKS for Large-Scale Inverse Problems},
  pdfauthor={}
}
\fi

\begin{document}

\maketitle

\begin{abstract}
For high-quality images with sharp edges, a popular choice for
edge-preserving regularization 
is using a general(ized) $\ell_q$-norm of the gradient of the image. This can be implemented efficiently using the $\ell_2$-norm 
and a sequence of weighted gradients, with weights derived from the current solution estimate. 
We can solve the resulting sequence of regularized least squares problems using
hybrid Krylov subspace methods,
which efficiently compute the 
regularization parameter using the
problem projected on the Krylov
subspace. However, each update of the regularization operator
requires
a new Krylov subspace. 
The majorization-minimization generalized Krylov subspace method (MM-GKS)
addresses this problem by 
using a single, generalized, Krylov subspace (GKS).  
Unfortunately, for large-scale problems, if convergence is not fast, 
MM-GKS has overwhelming memory
requirements 
and computational costs. 

We propose a variant of MM-GKS that 
alternately compresses and expands the search space while maintaining 
strict monotonic convergence. We show that our method
provably converges
to the minimum of the selected 
functional,
even if the search space dimension is kept very small.
This substantially improves on previous theoretical results for MM-GKS, 
where the convergence proof relies on the basis for the solution space (eventually) spanning the full space.
We show that our method can solve large-scale problems efficiently both in terms of memory requirements
and computational complexity.
We further generalize our proposed method 
to handle {\em streaming problems}, where the data is either not all available simultaneously or needs to be treated as such because of the extreme memory requirements.  

We use numerical examples from image deblurring, dynamic photoacoustic tomography, and streaming X-ray computed tomography (CT) to illustrate the effectiveness of our proposed methods.
\end{abstract}



\section{Introduction}
In this paper, we develop memory-aware and provably convergent algorithms for edge-preserving image reconstruction. 
In particular, we develop a new and efficient variant of the MM-GKS method that 
explores an effective search space while keeping the dimension of the search space small, and still provably converges to the desired regularized functional.  
We consider the linear forward model  
\begin{equation}\label{eq:LinSys}
\bA \bx_{\rm true}
+ \be = \bd ,
\end{equation}
where 
$\bA\in \R^{m\times n}$ represents the discretization of an ill-posed operator, 
the components of $\bx_{\rm true} \in \mathbb{R}^n$ are the pixel or voxel values of 
the medium of interest, $\be \in \R^{m}$ represents additive
noise from discretization, measurement, and rounding errors, and 
$\bd$ represents the
measured data.  

If Gaussian white noise is assumed, one standard regularization technique is generalized Tikhonov regularization 
\begin{equation} \label{eq:ellone} 
  \min_{\bx} \frac{1}{2}\| \bA \bx - \bd\|_2^2 + 
  \frac{\lambda}{q} \|\Psi \bx \|_q^q \, ,
\end{equation}
where $0 < q \leq 2$, $\Psi \in \R^{r \times n}$ is a regularization matrix, and the regularization parameter $\lambda$
determines
the balance of the 
regularization term (second term) versus the data misfit term (first term).  
The regularization operator can be used either to promote edge representation in images or to promote sparsity in the data.  In this paper, we focus on the former, in which case $\Psi$ will be a discrete gradient operator in space and/or time. 
Frequently, a choice of $q = 1$ is made to ensure adequate representation of edges and that the regularization term defines a valid norm \cite{buccini2021linearized, chung2019flexible, chung2022,lindbloom2025priorconditioned}. We therefore use $q = 1$ throughout the paper, but we note that {\it the theory and algorithms we develop here can easily be adapted to other choices of $q$ (in particular, $q<1$)}.
Note also that since $\Psi$ is not the identity, algorithms suited to $\ell_1$ regularization of the model coefficients, such as for LASSO, are not applicable here. 

In this paper, we are concerned with the fast and accurate solution of (\ref{eq:ellone}) when memory is at a premium. In particular, dynamic inverse problems are exceedingly large because they recover multiple/many images simultaneously and the regularization may couple the
individual reconstructions
(with each a large system)
\cite{pasha2021efficient}.  In addition, we consider solution of (\ref{eq:ellone}) in the streaming data case.  In the context of our work, the streaming data setting arises when either (a) the data may not be available all at once, due to the measurement system setup, or (b) all the data may be available but the volume
is so large or the rate so 
high that memory constraints demand that we process only chunks of data at a time.

Since \eqref{eq:ellone} is not differentiable for $q=1$ (or $q < 1$), it is common to replace the regularization term with a differentiable approximation, and instead minimize
\begin{equation} \label{eq:Je}
  \min_{\bx} \mathcal{J}_{\epsilon,\lambda}(\bx) = 
  \min_{\bx} \frac{1}{2} 
  \| \bA \bx - \bd\|_2^2 +
  \lambda \sum_{j=1}^{r} \phi_{\epsilon}( (\Psi \bx)_{j} ) ,
\end{equation}
where $\phi_{\epsilon}(t)$ gives a smoothed approximation to the $\ell_1$ norm for small $\epsilon$.
Here, we use 
\eqs\label{eq:phi_eps}
  \phi_{\epsilon}(t) 
  & = & 
  \sqrt{t^2 + \epsilon^2}.
\eqe
We can use the majorization-minimization (MM) approach \cite{hunter2004tutorial, lange2016mm, rodriguez2008efficient} to solve (\ref{eq:Je}), 
which leads to the solution of a sequence of regularized linear least squares problems,
\begin{equation} \label{eq:reg2}
  \bx^{(k+1)} := \arg \min_{\bx} 
  \frac{1}{2} \|\bA \bx - \bd \|_2^2 + \frac{\lambda}{2} 
  \| \bP_{\epsilon}^{(k)}\Psi \bx \|_2^2 \, , 
\end{equation}
where $\bP_{\epsilon}^{(k)}$ denotes a diagonal weighting matrix, whose entries are determined from the current solution estimate $\bx^{(k)}$; 
see Section \ref{sec:background}. 
For a {\it fixed $k$} and known $\lambda$, (\ref{eq:reg2}) can be solved by a Krylov subspace method.  If $\lambda$ is not known,  both $\lambda$ and 
the solution can be computed simultaneously using hybrid Krylov subspace methods \cite{gazzola2020inner, chung2019flexible, gazzola2018ir}.
However, 
every optimization step, i.e., for each $k$,
a new Krylov subspace must 
be computed to accommodate 
the change in 
$\bP_{\epsilon}^{(k)}$, 
leading to a nested iteration
that constitutes a
substantial computational bottleneck.  
To resolve this problem, a majorization-minimization 
generalized Krylov subspace method (MM-GKS) has 
been proposed \cite{lanza2015generalized, huang2017majorization} that combines norm reweighting 
with solution approximations from 
a single generalized Krylov subspaces (GKS) \cite{lampe2012large}.
Unfortunately, 
computational experiments \cite{huang2017majorization, pasha2021efficient} illustrate that, when MM-GKS is used for large-scale  problems and convergence
is not fast, storage requirements can easily exceed memory capacity, and the 
orthogonalization costs of MM-GKS
can become prohibitive.  
Hence, there is a need to develop alternative approaches that limit the dimension
of the generalized
Krylov space without sacrificing
convergence rate and  reconstruction quality. Recently, a restarted version of MM-GKS \cite{buccini2023limited} has been proposed to address the potentially high memory requirements of MM-GKS. However, 
restarting  
discards important information, necessitating additional work,  
 potentially reducing 
solution quality (see section \ref{subsec:DynPAT}), and convergence to the solution of \eqref{eq:Je} cannot be proved.

Therefore, we develop a more effective method that alternately enlarges and judiciously compresses the solution subspace, keeping the 
dimension bounded independent of the number of iterations, while still preserving the relevant information throughout the process for fast convergence, similar to the approach in \cite{jiang2021hybrid}.
This is particularly relevant 
for handling streaming data
and (large) dynamic inverse
problems.  

However, keeping the solution space dimension bounded leads to a theoretical 
issue. The convergence proof of MM-GKS relies on eventually spanning the full solution space, $\R^n$, \cite{lanza2015generalized,
huang2017majorization,
chan2014}.
Since our proposed method keeps the search space dimension bounded, just as for the restarted method, the
previous approaches to a convergence proof for MM-GKS cannot be used\footnote{Of course, in practice, relying on spanning the full solution space, $\R^n$, is
not useful.}. 
Therefore, we consider two changes to the MM-GKS algorithm. First, the new algorithm expands the search space with the gradient of the updated quadratic majorant 
rather than the gradient of the current
quadratic majorant (see Section \ref{ssec:modMMGKS}). Second, we propose a number of methods 
that select appropriate subspaces of the search 
space to keep when compressing 
the search space. This allows us
to develop an algorithm that (1) provably converges to the minimum of 
\eqref{eq:Je} for a fixed
regularization parameter (this is the context in which previous convergence proofs have been derived), (2) can use hybrid Krylov method techniques to estimate the regularization parameter
dynamically (avoiding an 
expensive outer iteration to determine a good regularization parameter),
and (3) maintains fast convergence.
Extending our method to provably converge to the minimum
of \eqref{eq:Je} while dynamically determining an optimal regularization parameter is future work.

The main innovations
in this paper are as follows:
\begin{enumerate}
  \item We propose a variant of MM-GKS that extends the search space with a different search direction that allows much stronger theoretical results than have been proved previously for MM-GKS itself.
  
  \item We prove convergence of the new MM-GKS variant to the minimum of the functional \eqref{eq:Je}, 
  using only a fixed, small maximum search space dimension. 
  This approach would also allow stronger convergence results for a correspondingly adapted version of the recently proposed restarted MM-GKS \cite{buccini2023limited} (it would no longer be just a restarted method, though).
  This substantially improves the previous convergence results in \cite{huang2017majorization,buccini2023limited}.  
  
  \item For fast convergence with limited memory requirements, we combine the new version of MM-GKS with a variant that alternates between  
  \begin{itemize}
    \item expanding the solution space by a modest number of basis vectors up to
    a chosen maximum dimension of the solution space;
    \item compressing the solution space to a chosen dimension using techniques that retain important information;
  \end{itemize}
  
  \item We give a new procedure for the initialization of the solution basis so that edge information is encoded earlier into the process;

  \item We develop a new LM-MM-GKS approach for
  streaming data, which can handle 
  \begin{itemize}
    \item chunks of data (all pertaining to the same solution) that are available only sequentially, and
    \item storage limitations that  
    allow storing only a modest 
    number of rows of the system; 
  \end{itemize}
  
  \item We provide rigorous numerical comparisons against competing methods, on several applications, illustrating the superiority of our method with regard to both storage and quality of reconstruction. 
\end{enumerate}

The paper is organized as follows.  In Section \ref{sec:background}, we discuss
the MM-GKS method.  
In Section \ref{sec: recycle}, we make the case for more memory and computationally efficient variants
of MM-GKS, propose our limited-memory MM-GKS algorithm (LM-MM-GKS), 
and discuss its memory and computational efficiency.
We show how to adapt LM-MM-GKS to handle streaming data in Section \ref{sec: streaming}.
In Section \ref{sec: compression}, we 
discuss several methods for
compressing the solution space, and, in Section \ref{sec:Convergence}, we give the convergence proof of the new
LM-MM-GKS algorithm.
Extensive numerical experiments are presented in
Section \ref{sec:NumericalExperiments}. We  
provide conclusions and suggestions for 
future work in Section \ref{sec: conclusion}.  
\section{Background} \label{sec:background}
To make the paper self-contained, we first provide some background on the MM-GKS method.
First proposed in \cite{lanza2015generalized}, the MM-GKS method uses a majorization-minimization approach, detailed below, to construct a sequence of iterates 
$\bx^{(k)}$ that converges to a stationary point of $\mathcal{J}_{\epsilon,\lambda}(\bx)$. 
Note that throughout this section, the value of $\lambda$ is assumed to be fixed and known, although in practice, a suitable value will need to be computed.  
\subsection{The majorization formulation}
Huang et al. \cite{huang2017majorization} describe two approaches to construct a quadratic 
tangent majorant to $\mathcal{J}_{\epsilon,\lambda}(\bx)$ in (\ref{eq:Je}) at an approximate solution $\bx^{(k)}$,
proposing adaptive and fixed 
quadratic majorants. The latter are cheaper to compute but may give slower convergence. 
In this paper, we focus on the adaptive quadratic majorants, but all results also 
hold for fixed quadratic majorants. 
\begin{definition}[\cite{huang2017majorization}]\label{def: 1}
The functional $\bx\mapsto\mathcal{Q}(\bx,\bv):\R^{n}\rightarrow \R$ is a 
quadratic tangent majorant for $\mathcal{J}_{\epsilon,\lambda}(\bx)$ 
at $\bv\in\R^n$, if for all $\bv \in \R^n$,
\begin{enumerate}
\item $\bx\mapsto\mathcal{Q}(\bx,\bv)$ is quadratic,
\item $\mathcal{Q}(\bv,\bv)=\mathcal{J}_{\epsilon,\lambda}(\bv)$,
\item $\bigtriangledown_{\bx}\mathcal{Q}(\bv,\bv)=\bigtriangledown_{\bx}\mathcal{J}_{\epsilon,\lambda}(\bv)$,
\item $\mathcal{Q}(\bx,\bv)\geq \mathcal{J}_{\epsilon,\lambda}(\bx)\quad\forall \bx\in \R^{n}$.
\end{enumerate}
\end{definition}
Let $\bx^{(k)}$ be an approximate solution of \eqref{eq:Je}, define
$\bu^{(k)}$ and the weights
$\bw_{\epsilon}^{(k)}$ as 
\begin{eqnarray}
\label{eq:uk}
\bu^{(k)} & = & \Psi\bx^{(k)} , \\
\label{eq: weightslplq}
\bw_{\epsilon}^{(k)} & = & \left((\bu^{(k)})^2+\epsilon^2\right)^{-1/2},
\end{eqnarray}
where all operations 
on the right-hand side
of \eqref{eq: weightslplq} 
are done element-wise.
Also define 
\begin{equation}
\label{eq: weightingMatrix}
    \bP_{\epsilon}^{(k)} = ( \text{diag}(\bw^{(k)}_{\epsilon}))^{1/2} ,
\end{equation}
and consider the quadratic tangent majorant for 
$\mathcal{J}_{\epsilon,\lambda}(\bx)$,  
with weighting matrix 
$\bP_{\epsilon}^{(k)}$, 
%
\begin{equation}\label{eq: QuadraticMajorantQ}
\begin{array}{rcl}
\mathcal{Q}(\bx, \bx^{(k)})  &=&
\frac{1}{2}\|\bA\bx-\bd\|^{2}_{2}
+\frac{\lambda}{2} \|\bP_{\epsilon}^{(k)}\Psi\bx\|^{2}_{2} +
c_k,
\end{array}
\end{equation}
where $c_k$ is a suitable constant\footnote{A non-negative value of $c_k$ is technically required for the second condition in the definition to hold. However, since the minimizer is independent of the value of $c_k$, we will not discuss it further.} that is independent of
$\bx$.
We refer the reader to \cite{lanza2015generalized} for the derivation of the weights $\bw_{\epsilon}^{(k)}$ 
and the value 
of $c_k$.

\subsection{The minimization step}\label{sec: MM_GKS}
In the MM approach, 
the next iterate, 
$\bx^{(k+1)}$, 
is 
the minimizer of 
\eqref{eq: QuadraticMajorantQ}.
Setting the gradient to zero
leads to the normal equations
\begin{equation}\label{eq: normaleqQuadMajorant}
(\bA^T\bA + \lambda  \Psi^{T} (\bP_{\epsilon}^{(k)})^2\Psi)\bx = \bA^T\bd.
\end{equation}
The system \eqref{eq: normaleqQuadMajorant} has a unique solution if
\begin{equation}\label{eq: NANL}
\mathcal{N}(\bA^T\bA)\cap \mathcal{N}(\Psi^{T} (\bP_{\epsilon}^{(k)})^2\Psi)=\{0\},
\end{equation}
which typically holds in practice.
In that case, for given $\lambda$,
the solution of \eqref{eq: normaleqQuadMajorant} is the unique minimizer
of the quadratic tangent majorant $\mathcal{Q}(\bx, \bx^{(k)})$.
In practice, $\lambda$ is usually not known in advance, and we discuss its computation in Section \ref{sec: RegParm}.

Unfortunately, solving \eqref{eq: normaleqQuadMajorant} for large $\bA$, $\Psi$, and changing
$\bP_\epsilon^{(k)}$, may be computationally prohibitive.  Therefore, 
the 
MM-GKS method, proposed
in \cite{huang2017majorization}, approximates
the MM iterates 
by projection onto low dimensional subspaces. 
The method starts with a few steps of Golub-Kahan
bidiagonalization (GKB) 
with $\bA$ and $\bu_1 = \bd/\| \bd \|_2$ to determine an initial subspace $\bV_\ell$ such that 
\begin{equation}\label{bdiag}
\bA\bV_{\ell}=\bU_{\ell+1}\bB_{\ell},
\end{equation}
and $\bV_\ell$ and
$\bU_{\ell+1}$ have orthonormal columns.
Given $\bx^{(k)}$ and 
$\bP_{\epsilon}^{(k)}$ 
obtained from 
$\bx^{(k)}$ (with $k=0$ in the first step),
MM-GKS computes the thin QR factorizations 
\begin{align}\label{eq: QR}
\bA\bV_{\ell+k} = \bQ_{\bA}\bR_{\bA}, \quad
\bP_{\epsilon}^{(k)}\Psi\bV_{\ell+k} = \bQ_{\Psi}\bR_{\Psi} .
\end{align}
Restricting 
$\bx^{(k+1)}$ to 
${\rm range}
(\bV_{\ell + k})$, 
$\bx^{(k+1)} = 
\bV_{\ell + k} \bz^{(k+1)}$, 
and adjusting
\eqref{eq: normaleqQuadMajorant}
accordingly leads to the
following small system
of equations
for $\bz^{(k+1)}$, 
\begin{equation} \label{eq: minKryov2}
  (\bR_\bA^T\bR_\bA + 
  \lambda \bR_{\Psi}^T\bR_{\Psi})\bz^{(k+1)} =
  \bR_\bA^T\bQ_\bA^T\bd .
\end{equation} 
Its solution gives for
the residual vector of the (full) regularized normal equations:
\begin{equation}\label{eq:residual}
\br^{(k+1)}=\bA^T(\bA\bV_{\ell+k}\bz^{(k+1)} -\bd)+\lambda  \Psi^T(\bP_{\epsilon}^{(k)})^2\Psi
\bV_{\ell+k}\bz^{(k+1)}.
\end{equation}
We expand the solution subspace with the normalized residual $\bv_{\rm new}=\br^{(k+1)}/\|\br^{(k+1)}\|_2$,
\begin{equation}\label{eq: enlargeMMGKS}
\bV_{\ell+k+1}=[\bV_{\ell+k},\bv_{\rm new}]\in\R^{n\times(\ell+k+1)}.\end{equation} 
In exact arithmetic,  $\bv_{\rm new}$ is orthogonal to the columns of 
$\bV_{\ell+k}$, 
but in computer arithmetic 
reorthogonalization of 
$\bV_{\ell+k+1}$ is typically needed.
Next, $\bP_{\epsilon}^{(k)}$ is updated for the new solution estimate, and the process in steps \eqref{eq: QR} - \eqref{eq: enlargeMMGKS} is repeated, expanding the solution space until a 
sufficiently accurate solution is reached.  
\subsection{Choice of the regularization parameter} \label{sec: RegParm}

For known $\lambda$ and using (eventually) the full solution space $\R^n$,
the MM-GKS method produces
a sequence of iterates $\bx^{(k)}$ that converges to the minimizer of $\mathcal{J}_{\epsilon,\lambda}$. However, 
in practice, the optimal 
$\lambda$ 
is not known a priori. 
On the other hand, a suitable regularization parameter $\lambda = \lambda^{(k)}$ for each 
projected problem (\ref{eq: minKryov2}) can
be determined by various methods, 
such as generalized cross-validation (GCV), since the dimension of the projected problem is small. 

We summarize the MM-GKS method in Algorithm \ref{Alg: MM-GKS}.

\begin{algorithm}[!ht]
\caption{MM-GKS \cite{lanza2015generalized}} 
\label{Alg: MM-GKS}
\begin{algorithmic}[1]
\INPUT{$\bA, \Psi, \bd, \bx^{(0)}, \epsilon$} 
\OUTPUT{An approximate solution $\bx^{(k+1)}$}
\Function{ $\bx^{(k+1)} = $ MM-GKS}{$\bA, \Psi, \bd, \bx^{(0)}, \epsilon$}
  \State Generate initial solution space basis $\bV_{\ell}\in \R^{n\times \ell}$
  (see \ref{bdiag})
  \For {$k = 0, 1, 2, \ldots$ 
    until convergence}
    \State $\bu^{(k)}=\Psi \bx^{(k)}$
    \State $\bw^{(k)}_{\epsilon}=
      \left( (\bu^{(k)})^2+
      \epsilon^2 \right)^{-1/2}$	
    \State $\bP_{\epsilon}^{(k)} = {\rm
      diag} (\bw_{\epsilon}^{(k)})^{1/2}$
    \State Update QR-decomposition $\bA\bV_{\ell+k} = \bQ_{\bA}\bR_{\bA}$
    \State Compute QR-decomposition $\bP_{\epsilon}^{(k)} \Psi \bV_{\ell+k}=\bQ_{\Psi}\bR_{\Psi}$
    \State Select $\lambda^{(k)}$ using heuristic method for (\ref{eq: minKryov2}) (e.g., GCV). 
    \State Solve for $\bz^{(k+1)}$ from 
      (\ref{eq: minKryov2}) with selected $\lambda^{(k)}$
    \State $\bx^{(k+1)}=\bV_{\ell+k}\bz^{(k+1)}$
    \If { $\| \bx^{(k+1)} - \bx^{(k)}\|_2 / \|\bx^{(k)}\|_2 \leq tol_1$ }
      \State \text{break;}
    \EndIf
    \State $\br^{(k+1)} = \bA^T(\bA\bV_{\ell+k}
      \bz^{(k+1)} -\bd) + \lambda^{(k)} \Psi^T(\bP_{\epsilon}^{(k)})^2  \Psi\bV_{\ell+k}\bz^{(k+1)}$
    \State $\br^{(k+1)}  = \br^{(k+1)} -
      \bV_{\ell+k} \bV_{\ell+k}^{T} \br^{(k+1)}$ \Comment{Reorthogonalize, if needed}
    \State $\bv_{\rm new} = 
      \frac{ \br^{(k+1)} }
      { \|\br^{(k+1)}\|_{2} }$;  $\bV_{\ell+k+1}=[\bV_{\ell+k}, \bv_{\rm new}]$  \Comment{Enlarge the solution subspace} 
  \EndFor
\EndFunction
\end{algorithmic}
\end{algorithm}

\section{A Limited Memory MM-GKS variant 
with guaranteed convergence}
\label{sec: recycle}

In this section, we describe a limited memory variant of MM-GKS that provably converges to 
the minimum of the smoothed functional $\calJ_{\epsilon,\lambda}$
(or a stationary point for a non-convex functional), 
even when the dimension of the search space remains bounded. The key ingredients are (i)~alternating expansion and compression of the search space, and (ii)~a subtle but critical change in the  
subspace extension using updated weights, as detailed below. 
The convergence is proved in Section~\ref{sec:Convergence}.

MM-GKS methods have been widely used for large-scale inverse problems \cite{buccini2020modulus, pasha2021efficient}. However, for large-scale problems
where convergence is not
fast,
storing the 
search space basis vectors can easily exceed the memory capacity, and 
the algorithm may have excessive computational cost. 
In particular, step~10 of Algorithm 2.1 leads to a computational cost of $O(r k(\ell+k)^2)$ flops for $k$
iterations beyond the initialization, where generally $r \geq n$. 
Moreover, if we constrain the solution space dimension below $n$,
the conditions for the convergence proof in \cite{huang2017majorization}
do not hold. 
Therefore, we propose a variant of MM-GKS, for which convergence to the minimum of the smoothed functional \eqref{eq:Je}
can be proved, even if the solution space dimension is
kept very small using
some truncation scheme.

The method we propose, 
LM-MM-GKS, modifies MM-GKS by 
(a) initializing with a better low-dimensional search space,
(b)
keeping the memory requirements constant,  without sacrificing reconstruction quality, 
and (c) extending the search space each step with a search vector that guarantees convergence to the
minimum of \eqref{eq:Je}.  
Indeed, our numerical results in 
Section \ref{sec:NumericalExperiments} show that we can improve the reconstruction quality over typical MM-GKS. 
LM-MM-GKS alternates between two main steps, enlarging and compressing the solution 
space, until a desired 
solution is found. 
The LM-MM-GKS method is presented in Algorithm \ref{Alg: RMMGKS}.
\subsection{Initialization} \label{ssec:init}
The basis vectors of the initial solution space in MM-GKS,  
$\bV_{\ell}$,  
are known to be fairly smooth for small $\ell$
\cite[Section 4.3]{hansen2008noise}, which delays convergence to solutions with edge information.  
Therefore, after computing
$\bV_\ell$, $\lambda^{(0)}$, $\bx^{(1)}$, and $\bP_\epsilon^{(1)}$,
our method generates {\it an improved 
seed basis that encodes the edge information we have obtained so far}, 
$\bV_{k_{\rm min}}$, whose columns are an orthonormal basis for the Krylov subspace
(with $k_{\rm min}$ the chosen minimal basis dimension)
\[ \mathcal{K}_{k_{\rm min}}(\bA^T \bA + \lambda^{(0)} \Psi^T (\bP_\epsilon^{(1)})^2 \Psi, \bA^T \bd).\]
\begin{remark}
In comparison with 
MM-GKS, the improved initialization space $\bV_{k_{\rm min}}$ has an additional cost of one MM-GKS expansion step (see below) and $k_{\rm min}$ matvecs with $\bA$, $\bP_{\epsilon}^{(1)} \Psi$ and their transposes. We do not assume that $k_{\rm min} = \ell$.  
\end{remark}

\subsection{Modifying the MM-GKS Update} \label{ssec:modMMGKS}
Standard MM-GKS, 
after computing $\bx^{(k+1)}$, 
extends the solution space with
the residual for $\bx^{(k+1)}$ 
in the regularized normal equations (\ref{eq:residual}) with
the weight matrix $\bP_\epsilon^{(k)}$
computed from $\bx^{(k)}$.  
However, 
having computed $\bx^{(k+1)}$, we can first compute 
$\bP_\epsilon^{(k+1)}$ and then
compute the {\it residual for the regularized normal equations 
with the new weight matrix}. 
{\em This alternative residual} corresponds to the gradient 
$\nabla_{\bx}{\cal Q}(\bx^{(k+1)},\bx^{(k+1)})$, which can be easily verified from \eqref{eq: QuadraticMajorantQ}. 
Making this subtle change in our code and extending the solution space with $\nabla_{\bx}Q(\bx^{(k+1)},\bx^{(k+1)})$, ensures our limited memory variant of MM-GKS converges
to the minimum of $\calJ_{\epsilon,\lambda}$, as proved in Section \ref{sec:Convergence}.
This difference can be seen by comparing line 17 of Algorithm 2.1
(the original MM-GKS)
with line 11 of our Algorithm 3.1, which describes how we enlarge our search space.

\subsection{Limited Memory Details}
To limit the search space dimension for an arbitrary number 
of iterations, the main loop of 
LM-MM-GKS
consists of two phases: a basis expansion phase and a basis compression phase.  
In the compression phase, 
we
determine a subspace of small dimension that
is deemed important for convergence
to recycle as the new search space (discussed below). 

\subsubsection{Basis Expansion}
\begin{algorithm}[!th]
\caption{Enlarge }%
\label{Alg: Enlarge}
\begin{algorithmic}[1]
\Function{$[\bx^{({\rm out})},
  \lambda^{({\rm out})}, 
  \bV_{k_{\rm out}}, 
  \bQ_\bA, 
  \bR_{\bA}, 
  \bR_{\Psi}, 
  \bu^{({\rm out})},
  \br^{({\rm out})}, 
  \bz^{({\rm out})}, 
  k_{\rm out}] = $ 
\newline \mbox{\;} Enlarge}
{$\bA, \bPsi, \bV_{k_{\min}},
\bA\bV_{k_{\min}}, \bPsi\bV_{k_{\min}},
\bd, \bx^{(0)}, \bu^{(0)}, 
s, \epsilon, tol_1, tol_2$} \vspace{.5em}

\State $\bw^{(0)}_{\epsilon} = \left((\bu^{(0)})^2+\epsilon^2\right)^{-1/2}$; $\bP_{\epsilon}^{(0)} = (\rm diag(\bw_{\epsilon}^{(0)}))^{1/2}$;

\For {$k=0,1,2,\ldots, s-1$}
    
  \State $[\bQ_{\bA},\bR_{\bA}] = 
  \textrm{updateQR}( \bA\bV_{k+k_{\min}} )$ 
    
  \State $[\bQ_{\bPsi},\bR_{\bPsi}] = \textrm{QR}( \bP_{\epsilon}^{(k)}\bPsi \bV_{k+k_{\min}})$
  
  \State Compute $\lambda$ using heuristic method for (\ref{eq: minKryov2}) (e.g., GCV).
  \vspace{0.5em}

  \State $\bz^{(k+1)} = \arg \min_{\bz} 
  \left\| \left[ 
  \begin{array}{c} 
  \bR_{\bA} \\ \sqrt{\lambda}\bR_{\Psi} 
  \end{array} \right] \bz - 
  \left[ \begin{array}{c}
    \bQ_{\bA}^T \bd \\ 0
  \end{array} \right] \right\|_2$ \vspace{0.5em}
  \Comment{using QR}
  
  \State $\bx^{(k+1)} = \bV_{k+k_{\min}}\bz^{(k+1)}$\;
  
  \State $\bu^{(k+1)} = \bPsi\bx^{(k+1)}$\;

  \State $\bw^{(k+1)}_{\epsilon} = \left((\bu^{(k+1)})^2+\epsilon^2\right)^{-1/2}$; $\bP_{\epsilon}^{(k+1)} = (\rm diag(\bw_{\epsilon}^{(k+1)}))^{1/2}$;
  
  \State $\br^{(k+1)}=\bA^T(\bA \bx^{(k+1)} -\bd)+\lambda^{(k)} \Psi^T(\bP_{\epsilon}^{(k+1)})^2\bu^{(k+1)}$

  \State $\bv = \br^{(k+1)} -
    \bV_{k+k_{\rm min}}
    \bV_{k+k_{\rm min}}^{T} \br^{(k+1)}$\;
    \Comment{Reorthogonalize if needed}
    
  \State $\bv_{\rm new} =
    \frac{\bv}{\|\bv\|_{2}}$
    
  \State $\bV_{k+k_{\rm min}+1} =
    [\bV_{k+k_{\rm min}} \;\; \bv_{\rm new}]$ \;\Comment{Enlarge the solution subspace}
    
  \State $\bA\bV_{k+k_{\min}+1} =
    [\bA\bV_{k+k_{\min}} \; \bA\bv_{\rm new}]$

  \State $\bPsi\bV_{k+k_{\min}+1} =
    [\bPsi\bV_{k+k_{\min}} \; 
    \bPsi \bv_{\rm new}]$    
  
  \If{\; 
    $\|\bx^{(k+1)} -
    \bx^{(k)}\|_2 / \|\bx^{(k)}\|_2 \leq tol_1$
    \; OR \; 
    $\|\br^{(k+1)}\|_2 \leq tol_2$ \;}
    
    \State $k_\textrm{out} =
      k_{\min} + k + 1$
    
    \State \text{break}
  \EndIf

\EndFor
\State $k_{\rm out} = s+k_{\min} $
\State \{{\em the subsequent assignments are for clarity; no copy is done}\}
\State $\bx^{(\textrm{out})} =
  \bx^{(k+1)}$; \;
  $\br^{(\textrm{out})} = \br^{(k+1)}$; \;
  $\bz^{(\textrm{out})} = \bz^{(k+1)}$; \;
  $\lambda^{(\textrm{out})} = \lambda$; \;
  $\bu^{({\rm out})} = 
  \bu^{(k+1)}$
  
\State $\bV_{k_{out}} = \bV_{k_{\rm min}+k}$\;

\EndFunction
\end{algorithmic}
\end{algorithm}

Apart from the change discussed in subsection \ref{ssec:modMMGKS}, this phase 
mimics the basis expansion in MM-GKS, but allows at most
$k_{\rm max}$ basis vectors.  
The first expansion step begins with the basis 
$\bV_{k_{\rm min}}$ 
described in subsection \ref{ssec:init} above
and adds at most $s = k_{\rm max} - k_{\rm min}$ vectors. 
Once LM-MM-GKS reaches $k_{\max}$ basis vectors, it switches to the compression phase outlined in the next subsection. 
If it has not converged, another expansion phase will follow the basis compression.
The compression phase
computes a `compressed' search space,
of lower dimension,
that contains the best solution so far as well as the corresponding gradient of the updated quadratic tangent functional. 
Then, our algorithm roughly repeats steps 6--19 of Algorithm 2.1 with the  modified basis update discussed above for
$j = 0,\ldots,s-1$ or until the convergence criterion is met.  
The expansion phase is presented in Algorithm \ref{Alg: Enlarge}.
We note that the QR decomposition of $\bA\bV$ on line 4 of Algorithm~\ref{Alg: Enlarge} can be efficiently updated from the previous iteration, since only one new column is appended. The QR decomposition of $\bP_{\epsilon}^{(k)}\bPsi\bV$ on line 4 must be recomputed, as $\bP_{\epsilon}^{(k)}$ changes at each iteration. However, no additional matvecs with $\bPsi$ are required. Computational costs are discussed in detail in Section~\ref{sec: costs}.

\subsubsection{Basis Compression}\label{subsec: basiscompression}
\begin{algorithm}[!b]
\caption{Compress }
\label{Alg: Compress}
\begin{algorithmic}[1]

\Function{$[\bV_{k_{\min}}] =$ 
\newline \mbox{\;} Compress}
    {$\bV, \bQ_{\bA}, \bR_{\bA}, \bR_{\Psi}, \bd, \bx, \br,
    \bz, k_{\rm min}, \lambda$}

  \State $ \bW = \chi\left( \bR_{\bA}, \bR_{\Psi},  \bQ_{\bA}, \bd, 
  \bz, \lambda, k_{\min} \right)$ \; (with $\bW$ isometric) \Comment{see Section \ref{subsec: basiscompression}}

  \State $\widetilde{\bV} = \bV \bW \qquad\quad
    (\; \in \mathbb{R}^{n \times (k_{\min}-2)} \;)$ 
  \State $[\widetilde{\bV}, \verb.~.] = 
    \textrm{updateQR}([\widetilde{\bV} \;\; \bx \;\; \br])$
  \Comment{add solution and residual
    to search space}

  \State $\bV_{k_{\rm min}} =
    \widetilde{\bV}$
\EndFunction
\end{algorithmic}
\end{algorithm}
If the solution obtained at the end of the expansion phase is not sufficiently accurate, 
the dimension of the solution space
must be reduced to $k_{\rm min}$ 
in a way that
{\em keeps important solution information
in the compressed solution space}. 
In addition, we ensure that the retained solution space contains the current 
best solution and the gradient of the
corresponding quadratic majorant. This guarantees strict
monotonic decrease of 
${\cal J}_{\epsilon,\lambda}$
for a 
fixed regularization parameter
$\lambda$, which  follows directly from 
Definition \ref{def: 1} and 
the minimization of each 
iteration. For a similar argument, 
but only for monotonic decrease, see \cite{buccini2023limited}. 
Here, we describe the compression in a generic fashion, allowing various choices, 
in particular, those from \cite{jiang2021hybrid}.
In Section \ref{sec: compression}, we give details on specific compression routines
with certain desired properties.

Using the matrices
$\bR_{\bA}$, $\bR_{\Psi}$, $\bQ_\bA$,
the data vector $\bd$, the solution coefficient vector $\bz$,
and the regularization parameter, 
$\lambda$, 
we 
compute an isometric
matrix $\bW^{k_{\rm max} \times (k_{\rm min}-2)}$, and we set 
$\widetilde{\bV} := \bV_{k_{\rm max}} \bW$.
Next, we ensure that 
the compressed space contains both the current 
best solution $\bx$ and the residual $\br = \nabla_{\bx}\calQ(\bx,\bx)$, as required by the convergence proof in Section~\ref{sec:Convergence}; see Algorithm~\ref{Alg: Compress}, line 4.
The compression procedure is summarized in Algorithm \ref{Alg: Compress}.
  
\begin{remark}
In this paper and our algorithms, we focus on the $\ell_1$ norm for the regularization term;   
however, it is straightforward to extend this to more general $\ell_q$ norms with $0<q\leq2$. In particular, to obtain a general $\ell_q$ constrained LM-MM-GKS method, we 

change the assignments $w_{\epsilon}^{(k)} = (\, (u^{(k)})^2+\epsilon^2 \,)^{-1/2}$
in Algorithms \ref{Alg: Enlarge} and \ref{Alg: RMMGKS} 
to 
$w_{q,\epsilon}^{(k)} = (\, (u^{(k)})^2+\epsilon^2 \,)^{q/2 - 1}$;
see \cite{huang2017majorization}.
\end{remark}

\begin{algorithm}[!ht]
\caption{LM-MM-GKS}%
\label{Alg: RMMGKS}
\begin{algorithmic}[1]
\Function{$[\bx, \bV_{k_{\rm min}}]= $ LM-MM-GKS}{$\bA, \bPsi, \bV, \bd, \bx^{(0)}, \epsilon$, $k_{\rm min}$, $k_{\rm max}$, $tol_1$, $tol_2$} 
\STATE
\hspace{5em} $\rhd$ 
If $\bV \neq [\;]$, then $\bx^{(0)},\br^{(0)} \in \Ra{\bV}$ and 
$\lambda^{(0)}$ is given

  \State  $s = k_{\rm max}-k_{\rm min}$\;
  \If{$\bV$ provided} 
    \State $\bV_{\ell} = \bV$
    \State $\bu^{(0)}=\Psi \bx^{(0)}$\;
    \State	 $\bw^{(0)}_{\epsilon}=\left((\bu^{(0)})^2+\epsilon^2\right)^{-1/2}$\;	
    \State $\bP_{\epsilon}^{(0)} = (\rm diag(\bw_{\epsilon}^{(0)}))^{1/2}$ \Comment{Compute the weights}
  \Else
    \State Generate initial subspace basis
    using GKB: $\bV_{\ell}\in \R^{n\times \ell}$ 
    \State $\bu^{(0)}=\Psi \bx^{(0)}$\;
    \State	 $\bw^{(0)}_{\epsilon} = 
    ( (\bu^{(0)})^2+\epsilon^2 )^{-1/2}$	
    \State $\bP_{\epsilon}^{(0)} = (\rm diag(\bw_{\epsilon}^{(0)}))^{1/2}$  \Comment{Compute the weights}
  \EndIf
  \State Compute QR-decomposition $\bA\bV_{\ell} = \bQ_{\bA}\bR_{\bA}$
  \State Compute QR-decomposition $\bP_{\epsilon}^{(0)}\Psi \bV_{\ell}=\bQ_{\Psi}\bR_{\Psi}$\; 
  \If{$\lambda^{(0)}$ not provided}
    \State Select $\lambda^{(0)}$ using heuristic method for (\ref{eq: minKryov2}) (e.g., GCV).
  \EndIf
  \State $\bz^{(1)}=( \bR_{\bA}^T\bR_{\bA} + \lambda^{(0)} \bR_{\Psi}^T\bR_{\Psi})^{-1}\bR_{\bA}^T\bQ_{\bA}^T\bd $; \; $\bx^{(1)}=\bV_{\ell} \bz^{(1)}$
  \State $\bu^{(1)}=\Psi \bx^{(1)}$; \; 
  $\bw^{(1)}_{\epsilon} = ( (\bu^{(1)})^2+\epsilon^2 )^{-1/2}$; \; $\bP_{\epsilon}^{(1)} = (\rm diag(\bw_{\epsilon}^{(1)}))^{1/2}$
  \State $\br^{(1)} =
  \bA^T(\bA \bx^{(1)} - \bd) + \lambda^{(0)}\bPsi^T
  \left(\bP_{\epsilon}^{(1)}\right)^2
  \bPsi \bx^{(1)}$
  \State Do  $k_{\min}-2$ GKB steps to compute isometric $\bV$ such that
  \newline \mbox{\qquad \qquad}
  $\Ra{\bV} = \mathcal{K}_{ k_{\min}-2 } (\bA^T\bA + \lambda^{(0)} \Psi^T\left(\bP^{(1)}_{\epsilon} \right)^2 \Psi, \bA^T\bd)$ 

  \State  [$\bV_{k_{\min}}$, \texttildelow ] = 
  \textrm{updateQR}$([\bV \;\; \bx^{(1)} \;\; 
    \br^{(1)} ])$ 

  
  \For {$i= 1,2,\ldots, i_{\rm max}$}{ 
    \State $[\bx^{(i+1)}, \lambda^{(i)}, \bV_{k_{\rm out}},   
    \bQ_{\bA},
    \bR_{\bA}, 
    \bR_{\Psi}, \bu^{(i+1)},
    \br^{(i+1)},
    \bz^{(i+1)}, k_{\rm out}
    ] =$ 
    \newline \mbox{\qquad \qquad} 
    \Call{Enlarge}{$\bA, 
    \bPsi, 
    \bV_{k_{\rm min}}, 
    \bA\bV_{k_{\rm min}},
    \bPsi\bV_{k_{\rm min}},
    \bd, \bx^{(i)}, \bu^{(i)}, 
    s, \epsilon, tol_1, tol_2$}
    \vspace{0.5em}
    
    \State $\bV_{k_{\rm min}} =$ \Call{Compress}{$\bV_{k_{\rm out}}, \bQ_{\bA},
    \bR_{\bA}, \bR_{\Psi}, 
    \bd, \bx^{(i+1)},
    \br^{(i+1)}, \bz^{(i+1)},
    k_{\rm min}, \lambda^{(i)}$}  
    
    \State \Comment{Compress the subspace  $\bV_{k_{\rm out}}$ to $\bV_{k_{\rm min}}$ \footnotemark }
    
    \If{$\|\bx^{(i+1)} -
      \bx^{(i)}\|_2/
      \|\bx^{(i)}\|_2 \leq tol_1$ 
     \; OR \;
     $\|\br^{(i+1)}\|_2 \leq tol_2$ \;}
      \State \text{break;}
    \EndIf
}
\EndFor
\EndFunction
\end{algorithmic}
\end{algorithm}

\subsection{Computational and Storage Costs}\label{sec: costs}
We briefly discuss the main costs of MM-GKS and compare these with those of LM-MM-GKS.

For clarity, in this paper, we present straightforward implementations of LM-MM-GKS, but more efficient implementations are possible and will be future work. Note that minimizing storage or flops or data movement will typically lead to different implementation choices.

The memory bottleneck for MM-GKS is that the storage, mainly for the vectors of length $n$ or $r$, grows linearly with the number of iterations. The computational bottleneck for MM-GKS is the repeated (thin) QR-decomposition of the matrix 
$\bP^{(k)}_{\epsilon}\Psi\bV_{\ell+k}$,
which must be computed each iteration, as $\bP^{(k)}_{\epsilon}$ changes every iteration. This alone leads to a computational cost of 
$O(\,r(k^3 + k^2\ell + k\ell^2)\,)$, where $\ell$ is the dimension of the initial search space,  $\bV_{\ell}$, $k$ is the number of iterations, and $r$ is the row dimension 
of the regularization matrix $\boldsymbol{\Psi}$. Typically, $r \geq n$; for example, in 
subsection 
\ref{Sec: ImgDebl},
$r \approx 2n$.

In contrast, the storage cost of LM-MM-GKS is bounded by a constant, 
independent of the number of iterations, as the method uses at most 
$k_{\max}$ basis vectors. 
The cost of repeated QR-decompositions of 
$\bP^{(k)}_{\epsilon}\boldsymbol{\Psi} \bV_{j}$, 
with $j$ the search space dimension, is now {\em linear} in the number of iterations, as $j \leq k_{\max}$, the maximum search space dimension. Of course, it is (still) roughly cubic in 
$k_{\max}$, but this is a modest constant, independent of the number of iterations. 
In Section \ref{sec:NumericalExperiments}, we demonstrate that $k_{\max}$ can
indeed be quite small with LM-MM-GKS still computing high quality 
reconstructions, comparable in quality to those computed by MM-GKS without constraints on the solution space dimension.
The second step in the initialization of $\bV_{k_{\min}}$ creates some modest overhead for LM-MM-GKS, but leads to faster convergence. 
Next, we outline the main computational costs of LM-MM-GKS,
considering only $O(n)$ or $O(r)$
components.   
We start with the input for 
Algorithm \ref{Alg: Enlarge} (Enlarge) and the for-loop, which repeats 
$s = k_{\rm max} - k_{\rm min}$ times (or until convergence).
After $s$ iterations 
the algorithm has incurred the following major costs: 

\medskip
\noindent{\underline{Matrix-vector products}}:  
\begin{itemize} 
   \item $k_{\max}$ matrix-vector products with $\bA$.
   \item $k_{\max}$ matrix-vector products with $\boldsymbol{\Psi}$ and with $\bP_{\epsilon}^{(k)}$ (diagonal matrix).
\end{itemize}
\medskip
\noindent{\underline{QR decompositions and Orthogonalizations}}:  
\begin{itemize}
  \item 1 QR decomp. of $\bA\bV_{k_{\max}} = \bQ_{\bA}\bR_{\bA}$.

  \item 
  $s$ QR decomp. of 
  $\bP_{\epsilon}^{(k)} \Psi\bV_{k_{\min}+k} = \bQ_\Psi\bR_\Psi$ for 
  $k = 0, \ldots, s-1$.

  \item $s$ orthogonalizations of $\br^{(k+1)}$ to $\bV_{k_{\rm min}+k}$
  of average size $k_{\min}+s/2$.
\end{itemize}
\medskip
Algorithm \ref{Alg: Compress} has the following major costs:
\medskip
\begin{itemize}
  \item Compute the product $\bV_{k_{\rm max}} \bW$. 
  \item Orthogonalize the current solution and residual/gradient against $\widetilde{\bV}$.  
\end{itemize}
\footnotetext{typically $k_{\rm out} = k_{\rm max}$}

\section{LM-MM-GKS for Streaming Data (s-LM-MM-GKS)}
\label{sec: streaming}
In addition to the linear forward model 
in~\eqref{eq:LinSys}, we assume that 
one of the following conditions
holds: 
\begin{itemize}
  \item 
  Only a portion of the data is available at any given time (the data is streamed), 
  \item The problem is massive, so, at any time, we can only handle a subset of rows. 
\end{itemize}
If one of these conditions holds, we refer to the problem as the {\it streaming data case.}

Ideally, we would like to solve 
simultaneously the whole problem,
\begin{equation}
    \label{eq:rtomoproblemall}
    \min_\bx \left\|{\begin{bmatrix}\bA_1\\ \vdots \\ \bA_{n_t}\end{bmatrix} \bx - \begin{bmatrix} \bd_1 \\ \vdots \\ \bd_{n_t} \end{bmatrix}}\right\|_2^2 + \lambda \|\Psi{\bx}\|_1^1 .  
\end{equation}
However, in the streaming data case, 
this is not feasible. 
To adapt LM-MM-GKS for the streaming data case, we partition $\bA$ and $\bd$ and 
solve the $n_t$ 
subproblems 
\begin{align}
 \label{eq:rtomo1}
  \min_{\bx} \|\bA_1 \bx - \bd_1\|_2^2 & +  \lambda_1 \|\Psi\bx\|_1^1  , \\
  \label{eq:rtomoi}
\min_{\bx} \|\bA_2 \bx - \bd_2\|_2^2 & +  \lambda_2 \|\Psi\bx\|_1^1 , \\ 
 & \vdots & \nonumber \\
\label{eq:rtomont}
\min_{\bx} \|\bA_{n_t} \bx - \bd_{n_t}\|_2^2 & + \lambda_{n_t} \|\Psi\bx\|_1^1 , 
\end{align}
in succession, while recycling selected
data from one subproblem to the next,
in such a way that $\bx_{(n_t)} \approx \bx,$ where $\bx$ refers to the solution of 
(\ref{eq:rtomoproblemall}). This is similar to the approach in \cite{jiang2021hybrid} for simpler, linear, problems.
We solve \eqref{eq:rtomo1} with LM-MM-GKS and obtain a solution subspace 
$\bV^{(1)}_{k_{\rm max}}$ of (typically) dimension $k_{\max}$ 
and an approximate solution $\bx_{(1)}$. 
For this first system, $\bV^{(0)}_{k_{\rm min}}$ (in line 2 of Algorithm \ref{Alg: sRMMGKS}) is typically not given, but computed as in line 23 of Algorithm \ref{Alg: RMMGKS}. 
To avoid keeping redundant basis information and/or because of storage
limitations, we compress 
$\bV^{(1)}_{k_{\rm max}}$ 
using Algorithm \ref{Alg: Compress} 
to obtain 
$\bV^{(1)}_{k_{\rm min}}$, 
ensuring that, 
similar to the non-streamed case, 
$\bx_{(1)} \in {\rm range}
(\bV^{(1)}_{k_{\rm min}})$.
Next, we use LM-MM-GKS to solve the second 
system with $\bx_{(1)}$ and $\bV_{k_{\rm min}}^{(1)}$ as initial approximate solution and solution subspace, respectively.
As the matrix changes, the gradient needs to be computed separately.
The process is repeated for the remaining systems. 
Algorithm~\ref{Alg: sRMMGKS} 
summarizes the solution of 
streaming problems. We illustrate the performance of s-LM-MM-GKS in the numerical examples section.

\begin{algorithm}[!t]
\caption{s-LM-MM-GKS}
\label{Alg: sRMMGKS}
\begin{algorithmic}[1]
\Function{$\bx = $ s-LM-MM-GKS}{${\bA_1}, \cdots, \bA_{n_t}, \Psi, \bd_1, \cdots, \bd_{n_t}, \bx^{(0)}, \bepsilon, k_{\rm min}, k_{\rm max}, tol_1$}
\State $[\bx_{(1)}, \bV^{(1)}_{k_{\rm min}}] =$ LM-MM-GKS($\bA_1, \bPsi, \bV^{(0)}_{k_{\rm min}}, \bd_1, \bx^{(0)}, \epsilon, k_{\rm min}, k_{\rm max}, tol_1, tol_2$)
  \For{$j=2$ to $n_t$}
    \State $[\bx_{(j)}, \bV^{(j)}_{k_{\rm min}}] =$ LM-MM-GKS($\bA_{j}, \bPsi, \bV^{(j-1)}_{k_{\rm min}}, \bd_{j}, \bx_{(j-1)}, \epsilon,  k_{\rm min}, k_{\rm max}, tol_1, tol_2$) 
  \EndFor
  \State $\bx = \bx_{(n_t)}$
\EndFunction
\end{algorithmic}
\end{algorithm}

\section{Compression approaches}\label{sec: compression}
We describe four strategies for computing the matrix $\bW$
in Algorithm \ref{Alg: Compress},
derived
from \cite{jiang2021hybrid},
that are appropriate for solving regularized projected ill-posed problems.
The first two strategies, truncated singular value decomposition (tSVD) and reduced basis decomposition (RBD), use 
the matrix 
\begin{equation}\label{eq: SVDStacked}
\bar{\bH}_{k_{\rm max}} = \begin{bmatrix} \bR_{\bA} \\ \sqrt{\lambda_{\rm curr}} \bR_{\Psi} \end{bmatrix},
\end{equation}
where $\lambda_{\rm curr}$ denotes the current regularization parameter.
The other two strategies,
solution-oriented compression and sparsity-enforcing compression,
are based on regularized solutions to the (small) projected problem. 
We give a numerical comparison of the compression approaches in section
\ref{Sec: ImgDebl}. 
\subsection{Truncated SVD (tSVD)}
We compute the tSVD of 
$\bar{\bH}_{k_{\rm max}} \approx \bU \bS \bW^{T}$, retaining the largest $k_{\rm min}-2$ singular values. 
Hence, $\bW \in \R^{k_{\rm max}\times (k_{\rm min}-2)}$ consists of
the corresponding right singular vectors.

\subsection{Reduced basis decomposition (RBD)}
The reduced basis decomposition 
\cite[Algorithm 1]{chen2015reduced} to obtain a compressed representation of a data matrix
uses a greedy strategy to determine an approximate factorization 
\begin{equation}\label{eq: RBD}
\bar{\bH}_{k_{\rm max}}^T \approx \bW_{j}  \bT_{j},
\end{equation}
where $\bW_{j} \in \R^{k_{\rm max}\times j}$ has orthonormal columns and $\bT \in \R^{j \times 2k_{\rm max}}$ is the transformation matrix. The RBD algorithm randomly selects a column of $\bar{\bH}_{k_{\rm max}}^T$, and sets $\bW_1$ to be the normalized column. Then $\bT_1$ is chosen to minimize the residual of \eqref{eq: RBD}. The maximum norm column of the residual matrix is selected as the next column of $\bW$, and so on. The RBD algorithm ends if a specified residual tolerance or
the maximum number of columns, here $k_{\min}-2$,
is reached. If the algorithm ends with 
fewer columns than $k_{\min}-2$, more iterations are allowed in the subsequent
Enlarge or LM-MM-GKS call. The final $\bW_j$ is returned as the selected $\bW$.

\subsection{Solution-oriented compression (SOC)}
In this method, $\bW$ is determined
by the solution of the regularized projected problem \eqref{eq: minKryov2}.  

Let $\bz = [z_1, z_2, \dots, z_{k_{\rm max}}]^T$ be the solution of \eqref{eq: minKryov2}. 
Define the index sets $I$ and $J$ for a given tolerance $tol$:
\begin{eqnarray}
\label{eq: Im}
I & = & \{i: |z_i|>tol\} \\
\label{eq: Jm}
J & = & \{i: |z_i| \text{ is one of the largest } k_{\min}-1 \text{ components}\}
\end{eqnarray}
Choose the index set $K = 
\left(I \cap J\right)$, and let $\bW = {\bf I}_{:,K}$ (with $\bf I$ the identity matrix). 
\subsection{Sparsity-enforcing compression (SEC)}
This method differs from the previous one only in the way the projected problem is regularized.  We solve
\begin{equation}\label{eq: Sparsity}
\bz^{**} = \arg\min_{\bz \in \R^{k_{\rm max}}}\|\bR_{\bA}\bz - \bQ_{\bA}^T \bd\|^2_2 + \rho\|\bR_{\Psi}\bz\|^1_1.
\end{equation}
which enforces sparsity of 
the solution $\bz^{**}$, with $\rho$ denoting a regularization parameter that can be computed using the GCV, but typically $\rho = \lambda_{\textrm{current}}$.  We solve this small regularized problem by a Majorization-Minimization algorithm \cite{lange2016mm}.
The solution $\bz^{**}$ is used only inside Algorithm \ref{Alg: Compress} to select the index sets in (\ref{eq: Im}) and (\ref{eq: Jm}).

\section{Convergence Proof}
\label{sec:Convergence}
For the convergence proof, 
we assume that $\lambda$ is 
known and fixed, as has been done
in previous papers
\cite{huang2017majorization, buccini2023limited}.
We prove the convergence of our algorithm
(for fixed $\lambda$) to 
(1) the minimizer of
$\calJ_{\epsilon,\lambda}$, cf. \eqref{eq:Je}, or more generally to the minimum of a convex functional, and (2) 
a stationary point of
a non-convex functional.
This is a substantial improvement
on previous proofs, which
either require (eventually) minimization of the tangent quadratic functional $Q(\bx,\bx^{(k)})$
over the full space $\mathbb{R}^n$ 
\cite{huang2017majorization}, which makes the proof unsuitable for limited memory implementations,
or, using only monotonicity,
prove convergence, but 
generally not to the 
minimum or stationary point of the
smoothed functional 
$\calJ_{\epsilon,\lambda}$  \cite{buccini2023limited}.

For each iteration $k$, we define 
\eqs\label{eq:Qk}
\bQ_k & = & 
\bA^T\bA + \lambda \bPsi^T(\bP_{\epsilon}^{(k)})^2 \bPsi ,
\eqe
$\bb = \bb^{(k)} := \bA^T \bd$,
which is a constant vector, and 
\eqs\label{eq:Qtxk}
  \widetilde{\cal Q}(\bx,\bx^{(k)}) & = & \frac{1}{2}\bx^T \bQ_k \bx - \bx^T\bb  
  \quad = 
  {\cal Q}(\bx,\bx^{(k)}) - c_k ,
\eqe
with ${\cal Q}(\bx,\bx^{(k)})$, a quadratic tangent majorant to 
${\cal J}_{\epsilon,\lambda}(\bx)$, defined in 
\eqref{eq: QuadraticMajorantQ}.
With these definitions, 
$\nabla_\bx {\cal Q}(\bx, \bx^{(k)}) = \nabla_\bx \widetilde{\cal Q}(\bx, \bx^{(k)})$ and 
${\cal Q}(\bx, \bx^{(k)}) - {\cal Q}(\by, \bx^{(k)}) = 
\widetilde{\cal Q}(\bx, \bx^{(k)}) - \widetilde{\cal Q}(\by, \bx^{(k)})$ for all $\bx, \by \in \R^{n}$. 
So, ${\cal Q}(\bx, \bx^{(k)})$ and 
$\widetilde{\calQ}(\bx, \bx^{(k)})$ have the same minimizer, and the improvement for an update $\bx^{(k+1)} - \bx^{(k)}$ is the same for both quadratic functions. Furthermore, it follows from Algorithms 
\ref{Alg: RMMGKS} (step 24), 
\ref{Alg: Enlarge} (step 11 -- 14), and \ref{Alg: Compress} (step 4)
that at each iteration 
of the for-loop in 
Algorithm \ref{Alg: Enlarge}, a new approximate solution is computed by minimizing over a search space that contains both the current best approximation and the corresponding gradient of the new tangent majorant. 

Now, consider a generic step of the algorithm when the search space is extended. 
Given $\bx_{k-1}$, a 
current search space $\bV_{m-1}$, and $\bP_{\epsilon}^{(k-1)}$ (derived from $\bx^{(k-1)}$
following \eqref{eq:uk}--\eqref{eq: weightingMatrix}), let
\begin{eqnarray}
\label{eq:xk}
  \bx^{(k)} & = & \arg \min_{\bx \, \in \, \textrm{range}(\bV_{m-1})} 
  \frac{1}{2}\| \bA\bx - \bd\|_2^2 + \frac{\lambda}{2} \|\bP_{\epsilon}^{(k-1)} \boldsymbol{\Psi} \bx\|_2^2 
\end{eqnarray}
(minimizing
$\widetilde{\calQ}(\bx,\bx^{(k-1)}) =
\frac{1}{2}\bx^T \bQ_{k-1} \bx - \bx^T\bb$), compute $\bP_{\epsilon}^{(k)}$
from $\bx^{(k)}$, and set 
\eqs
  \br^{(k)} & = & \nabla_{\bx}
  \widetilde{\calQ}(\bx^{(k)},\bx^{(k)}) = \bQ_k \bx^{(k)} - \bb .
\eqe
If $\br^{(k)} \neq 0$, then $\bx^{(k)}$ is not the minimizer of $\widetilde{\calQ}(\bx,\bx^{(k)})$, and the 
negative gradient, $-\br^{(k)}$, is a descent direction. The optimal step in this direction 
gives a reduction in $\widetilde{\calQ}(\bx,\bx^{(k)})$ that is at least proportional to 
$\| \nabla_\bx \widetilde{\calQ}(\bx^{(k)},\bx^{(k)})\|_2^2$ (see below).
Let $\bV_m$ be defined following steps 11--14 in Algorithm \ref{Alg: Enlarge}.
Then $\textrm{range}\;\bV_m = \textrm{range}[\bV_{m-1} \; \br^{(k)}]$ and
$\bx^{(k)} - \alpha \br^{(k)}
\in \textrm{range}\;\bV_m$,  so the minimizer for
$\widetilde{\calQ}(\bx,\bx^{(k)})$
over $\textrm{range}\;\bV_m$
gives a reduction in
$\widetilde{\calQ}(\bx,\bx^{(k)})$ that is at least as large
as that for $\bx^{(k)} - \alpha \br^{(k)}$. 

\begin{lemma}
\label{lem:min_progress}
Let $\calQ$, $\widetilde{\calQ}$, $\bQ_k$, $\bx^{(k)}$, and $\bb$ be defined as above, and let $\bar{\mu} = \sup_{k = 0, 1, \ldots}
\| \bQ_k \|_2$ exist and be finite. 
Furthermore, let
\eqsn
  \bx^{(k+1)} & = & 
  \arg \min_{\bx \,\in\, {\rm range}(\bV_{m} )}
  \widetilde{\calQ}(\bx,\bx^{(k)}).
\eqen
Then 
\eqsn  
  \calJ_{\epsilon,\lambda}(\bx^{(k+1)}) \leq
  \calQ(\bx^{(k+1)},\bx^{(k)}) & \leq & 
  \calQ(\bx^{(k)},\bx^{(k)}) - 
  \frac{\| \br^{(k)} \|_2^2}{2\bar{\mu}}
  = \calJ_{\epsilon,\lambda}(\bx_k) - \frac{\| \nabla_{\bx} \calJ_{\epsilon,\lambda}(\bx_k) \|_2^2}{2\bar{\mu}} .
\eqen
\end{lemma}
\begin{proof}

Let $\alpha$ be the minimizer of 
$\widetilde{\calQ}(\bx^{(k)} - \widetilde{\alpha}\br^{(k)},\bx^{(k)})$ as a function of $\widetilde{\alpha}$, then
$\alpha = \frac{\|\br^{(k)}\|_2^2} {{\br^{(k)}}^T \bQ_k \br^{(k)}}$ (see, e.g., \cite[section 3.3]{NocWri2e}). Substituting this into \eqref{eq:Qtxk}
gives
\eqsn
  \widetilde{\calQ}(\bx^{(k)} - \alpha \br^{(k)},\bx^{(k)})
  & = & 
  \widetilde{\calQ}(\bx^{(k)},\bx^{(k)}) +  
  \frac{1}{2}
  \alpha^2 {\br^{(k)}}^T \bQ_k \br^{(k)}
  - \alpha \|\br^{(k)}\|_2^2 \\
  & = &
  \widetilde{\calQ}(\bx^{(k)},\bx^{(k)}) -  
  \frac{1}{2} 
  \frac{\|\br^{(k)}\|_2^4}{{\br^{(k)}}^T\bQ_k\br^{(k)}} .
\eqen
Since $\| \bQ_k \|_2 \leq \bar{\mu}$, 
we have $(\br^{(k)})^T\bQ_k\br^{(k)} \leq
\|\bQ_k\|_2 \|\br^{(k)} \|_2^2  
\leq \bar{\mu} \|\br^{(k)}\|_2^2$,
and, therefore,
$\widetilde{\calQ}(\bx^{(k)} - \alpha\br^{(k)},\bx^{(k)}) 
\leq 
\widetilde{\calQ}(\bx^{(k)},\bx^{(k)}) 
- \| \br^{(k)} \|_2^2 /(2\bar{\mu})$.
As $\bx^{(k)} - \alpha\br^{(k)} \in 
\textrm{range}\;\bV_m$, 
$\widetilde{\calQ}(\bx^{(k+1)},\bx^{(k)}) 
\leq
\widetilde{\calQ}(\bx^{(k)} - \alpha \br^{(k)},\bx^{(k)})
\leq 
\widetilde{\calQ}(\bx^{(k)},\bx^{(k)}) 
- \| \br^{(k)} \|_2^2 / ( 2\bar{\mu} )$.
This also implies
\[
\calQ(\bx^{(k+1)},\bx^{(k)}) 
\leq
\calQ(\bx^{(k)},\bx^{(k)}) 
- \| \br^{(k)} \|_2^2 / ( 2\bar{\mu} )
\]
The remaining (in)equalities follow
directly from the fact that 
$\calQ(\bx,\bx^{(k)})$ is a quadratic 
majorant for $\calJ_{\epsilon,\lambda}(\bx)$
and tangent to $\calJ_{\epsilon,\lambda}(\bx)$
at $\bx^{(k)}$. See equations \eqref{eq:uk}--\eqref{eq: QuadraticMajorantQ}.
\end{proof}

\begin{lemma}
Let $\bQ_k = \bA^T\bA + \lambda \bPsi^T(\bP_{\epsilon}^{(k)})^2
  \bPsi$, 
cf. \eqref{eq:Qk}, then
\eqsn
  \bar{\mu} \leq 
  \|\bA\|_2^2 + \frac{1}{\epsilon}\|\bPsi\|_2^2 .
\eqen
\end{lemma}
\begin{proof}
For each $\bQ_k$, 
we have 
$\|\bQ_k\|_2 \leq
  \|\bA\|_2^2 +
  \|\bPsi\|_2^2 \;
  \|(\bP_{\epsilon}^{(k)})^2\|_2$, 
and the definition of 
$\bP_{\epsilon}^{(k)}$ in 
\eqref{eq: weightingMatrix}
and 
\eqref{eq: weightslplq}
give
$\|(\bP_{\epsilon}^{(k)})^2\|
\leq \epsilon^{-1}$,
which leads to the stated result.
\end{proof}

We can now give the proof that our algorithm converges to the minimum of $\calJ_{\epsilon,\lambda}(\bx)$ or to a stationary point for a non-convex functional.
\begin{thm}
Let 
$\calJ_{\epsilon,\lambda}(\bx)$ be defined as in \eqref{eq:Je} with $\phi_{\epsilon}(t)$ given by \eqref{eq:phi_eps}.
For $k = 1, 2, 3, \ldots$, let $\bx^{(k)}$, $\calQ(\bx,\bx^{(k)})$, and $\bx^{(k+1)}$ be defined as in 
\eqref{eq:xk},  
\eqref{eq: QuadraticMajorantQ}, and
Lemma \ref{lem:min_progress}.
Then
\eqsn
  \lim_{k \rightarrow \infty}
  \nabla_{\bx} \calJ_{\epsilon,\lambda}(\bx^{(k)}) = 0 .
\eqen
\end{thm}
\begin{proof}
Since $\calJ_{\epsilon,\lambda}(\bx)$ is bounded from below ($\calJ_{\epsilon,\lambda}(\bx) \geq 0$), 
$\calJ_{\epsilon,\lambda}^{*} = \inf_{\bx} \calJ_{\epsilon,\lambda}(\bx) \geq 0$
exists, and using Lemma 
\ref{lem:min_progress}
we have, for each $k$,
\eqsn
\calJ_{\epsilon,\lambda}^{*} - \calJ_{\epsilon,\lambda}(\bx_0) & \leq &
\calJ_{\epsilon,\lambda}(\bx^{(k+1)}) - \calJ_{\epsilon,\lambda}(\bx^{(0)}) 
\leq
\calJ_{\epsilon,\lambda}(\bx^{(k)}) - 
\frac{\| \nabla_{\bx} \calJ_{\epsilon,\lambda}(\bx^{(k)}) \|_2^2}{2\bar{\mu}}
- \calJ_{\epsilon,\lambda}(\bx^{(0)}) \\
& \leq &
\calJ_{\epsilon,\lambda}(\bx^{(k-1)}) - 
\frac{\| \nabla_{\bx} \calJ_{\epsilon,\lambda}(\bx^{(k-1)}) \|_2^2}{2\bar{\mu}} -
\frac{\| \nabla_{\bx} \calJ_{\epsilon,\lambda}(\bx^{(k)}) \|_2^2}{2\bar{\mu}}
- \calJ_{\epsilon,\lambda}(\bx^{(0)}) \\
& \leq & \cdots \leq
- \frac{1}{2\bar{\mu}}
\left( \sum_{j=0}^{k}
\| \nabla_{\bx} \calJ_{\epsilon,\lambda}(\bx^{(j)}) \|_2^2 \right) .
\eqen
If some $\eta > 0$ exists such that there is a (infinite) subsequence 
$\| \nabla_{\bx} \calJ_{\epsilon,\lambda}(\bx^{(j_i)}) \|_2^2 \geq \eta$
for all $j_i$ with $i = 1, 2, \ldots$, then 
\[
\lim_{k \rightarrow \infty}
- \frac{1}{2\bar{\mu}}
\left( \sum_{j=0}^{k}
\| \nabla_{\bx} \calJ_{\epsilon,\lambda}(\bx^{(j)}) \|_2^2 \right) = -\infty . 
\]
Hence, the inequality above shows that such an $\eta$ cannot exist, and therefore
\begin{equation} \nonumber
  \lim_{k \rightarrow \infty}
  \nabla_{\bx} \calJ_{\epsilon,\lambda}(\bx^{(k)}) = 0 .
\end{equation} \mbox{}
\end{proof}

\section{Numerical experiments} \label{sec:NumericalExperiments}

We illustrate the performance of LM-MM-GKS and s-LM-MM-GKS 
on applications in  
image deblurring, dynamic photoacoustic tomography, and computerized tomography, 
and we compare with the performance for existing methods. 
In every scenario, for each application, our approaches
are successful in providing high quality reconstructions with very limited memory requirements. 
Where 
we compare to full MM-GKS, s-LM-MM-GKS reconstructions are comparable in quality to the MM-GKS reconstructions but at greatly reduced memory requirements.  
Where we assume that the problem 
is too large to store all MM-GKS iterates, and therefore the quality of the MM-GKS reconstructions deteriorates, LM-MM-GKS, which can operate within the memory constraints, still provides high quality solutions. 

All computations were carried out in MATLAB\textsuperscript{\textregistered} R2023b with about 15 significant decimal digits running on a laptop computer with an 
Apple\textsuperscript{\textregistered} Core(TM)i7-8750H CPU @2.20GHz with 128GB of RAM.

For all testcases, we perturb the measurements with white Gaussian noise, i.e., the noise vector $\be$ has mean zero and a scaled identity covariance matrix; 
we refer to the ratio $\sigma=\|\be\|_{2}/\|\bA\bx\|_{2}$ as the noise level.  
For all MM-GKS and R-MMGKS variations in sections \ref{Sec: ImgDebl} and \ref{sec: CompTomo} we choose  \begin{equation}\label{eq: L}
    \Psi = \begin{bmatrix}
    \bI_{n_y} \otimes \bL_{x} \\ \bL_{y} \otimes \bI_{n_x}
\end{bmatrix},
\end{equation} where the matrix $\bL_x$ represents the discrete 
first derivative operator in the $x$-direction (vertical) and
$\bL_y$ in the $y$-direction (horizontal). 
Unless otherwise stated,
Algorithm ~\ref{Alg: Enlarge} (Enlarge) iterates until 
\begin{equation} \label{eq: TOL1} 
\frac{||\bx^{(k)} - \bx^{(k-1)}||_2}
{||\bx^{(k-1)}||_2} \leq
10^{-3} ,
\end{equation}
or the maximum number of iterations is reached, specified in each example. For all LM-MM-GKS experiments, the reported iteration counts refer to expansion steps of the Enlarge routine (Algorithm~\ref{Alg: Enlarge}). The initial approximation $\bx^{(1)}$ is computed from $\ell = 15$ GKB steps (line 10). The initial subspace $\bV_{k_{\min}}$ is then built from a $k_{\min}$-dimensional Krylov subspace associated with $\bA^T\bA + \lambda \bPsi^T \bP_{\epsilon}^2 \bPsi$, augmented with $\bx^{(1)}$ and $\br^{(1)}$ (lines 23-24); this one-time initialization cost is not included in the reported counts. For streaming experiments, subsequent subproblems reuse the recycled compressed subspace of dimension $k_{\min}$ directly. The values of $k_{\min}$ and $k_{\max}$ are specified in each experiment below.

To assess the reconstructed solutions, we consider the following quality 
measures. 
The Relative Reconstruction Error (RRE)
for reconstruction $\bx^{(k)}$ 
is defined as
\begin{equation} \label{eq: RRE} 
{\rm RRE}:={\rm RRE}(\bx^{(k)},\bx_{\mathrm{true}}) = \frac{||\bx^{(k)}-\bx_{\mathrm{true}}||_2}
{\|\bx_{\mathrm{true}}\|_2}. 
\end{equation}
For visual assessment, for all the numerical experiments reported below, we display error images $\bx_i^{(k)} - \bx_{\rm true}$
for each method $i$ in the inverted colormap. To enable direct comparison, all error images within each experiment are normalized by \begin{equation} e_{\max} = \max_{i}\|\bx_i^{(k)} - \bx_{\rm true}\|_{\infty},\end{equation} where the index $i$ runs over all methods compared in the experiment, and displayed on the unified scale $[-1, 0]$. Darker regions indicate larger reconstruction error. The Structural SIMilarity index (SSIM) measures 
how well the overall structure of a reference image is recovered by an approximation. The higher the index, the better the reconstruction; the highest achievable index is $1$.
The definition of the SSIM is involved, and we refer to \cite{SSIM} for details. 
Since this paper is concerned with reconstructing images that preserve edges in the medium of interest, we 
also use the recently proposed HaarPSI (HP) metric, which stands for Haar wavelet-based perceptual similarity index \cite{reisenhofer2018haar} with a maximum achievable value of 1. We also consider the PSNR that is a pixel-wise measure to quantify how close the reconstructed image is to the reference in terms of squared error \cite{huynh2008scope}.
\subsection{Image deblurring}
\label{Sec: ImgDebl}
\begin{figure}[ht!]
\begin{center}
\begin{minipage}{0.32\textwidth}
		\includegraphics[width=0.9\textwidth]{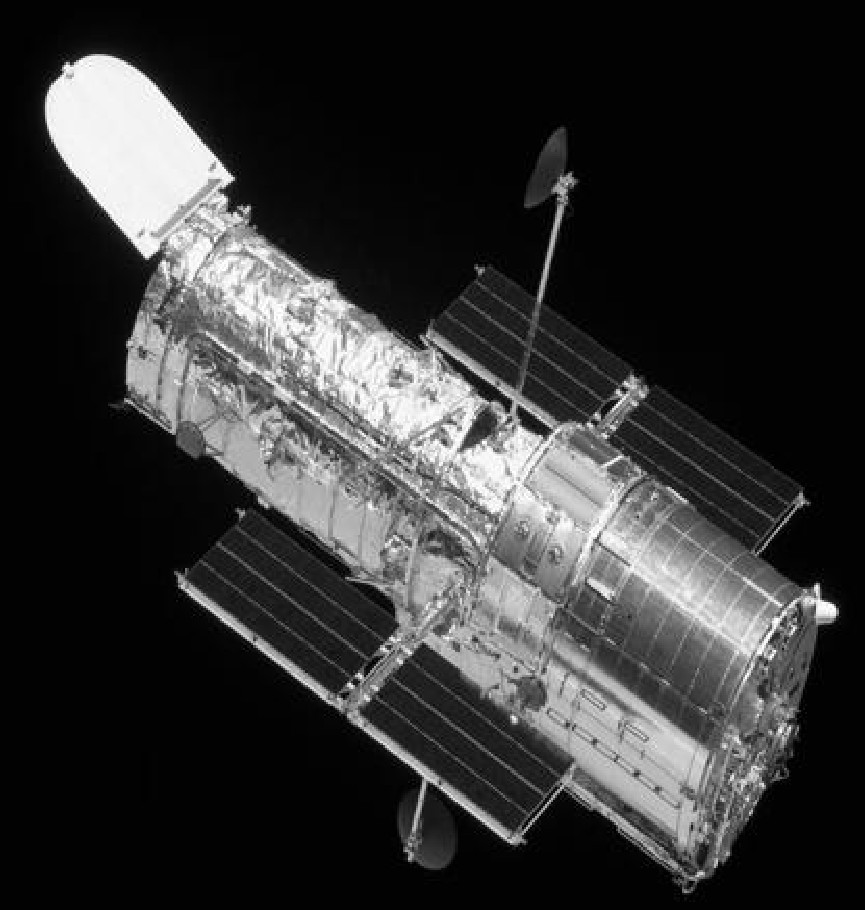}
	\end{minipage}
	\begin{minipage}{0.32\textwidth}
		\includegraphics[width=\textwidth]{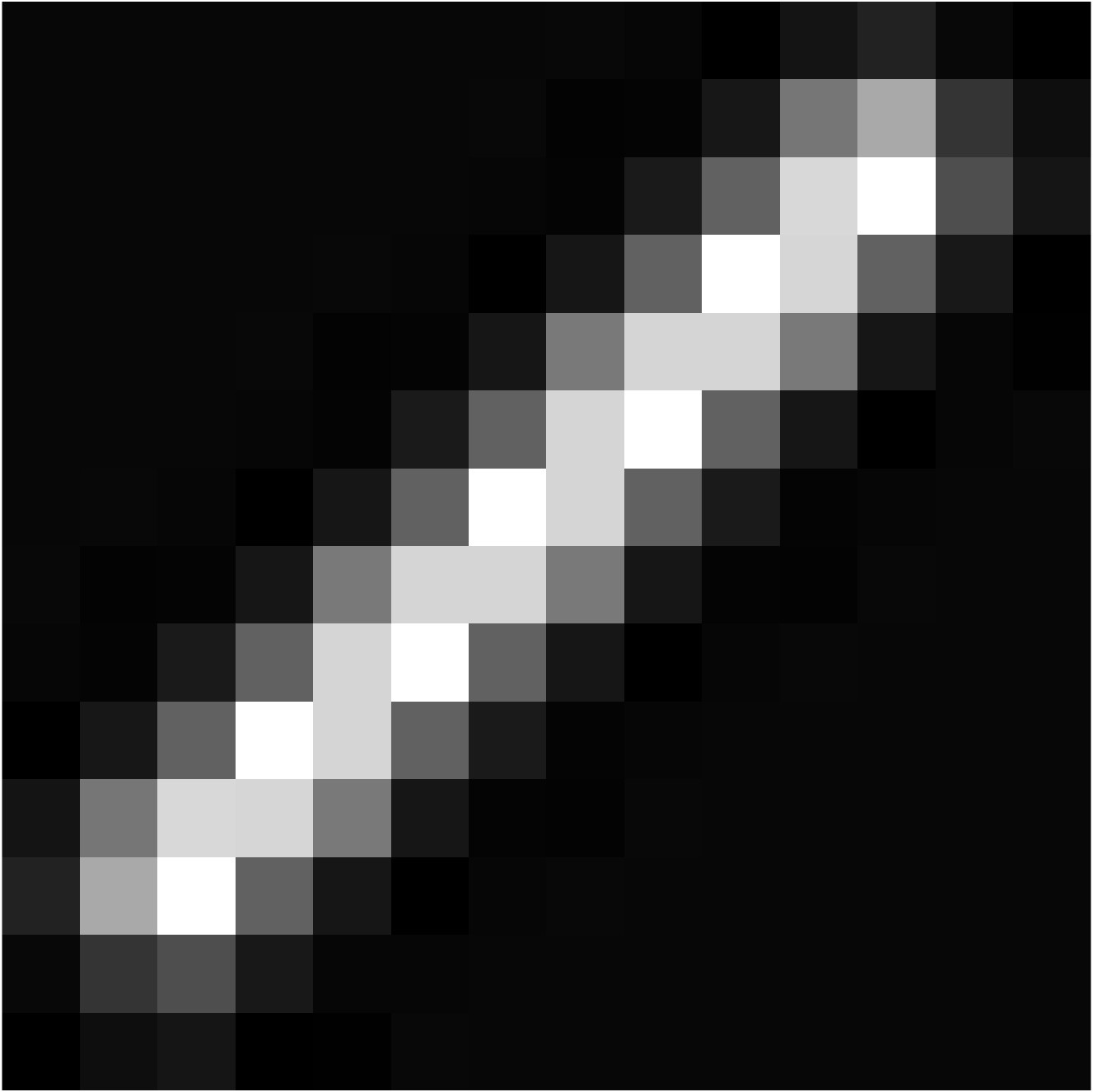}
	\end{minipage}
	\begin{minipage}{0.32\textwidth}
		\includegraphics[width=0.9\textwidth]{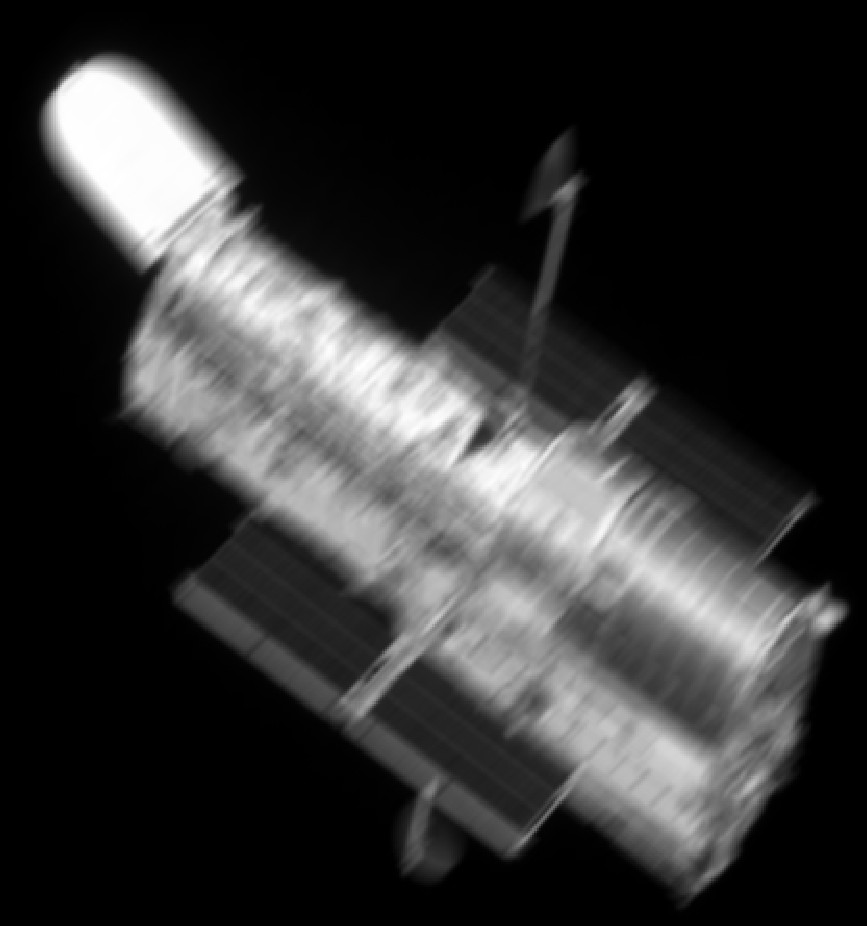}
	\end{minipage}
 \end{center}
\caption{Image deblurring. (a) True image of $500 \times 500$ pixels. (b) Motion blur PSF ($14 \times 14$ pixels), (c) Blurred and noisy image with $0.1\%$ Gaussian noise.}
\label{Fig: telescopeImages}
\end{figure}
\begin{figure}[ht!]
\centering
  \begin{tabular}{ccccc}
   MM-GKS (25) &  tSVD &  RBD & SEC & SOC\\
  \includegraphics[height = 0.13\textwidth, width = .15\textwidth]{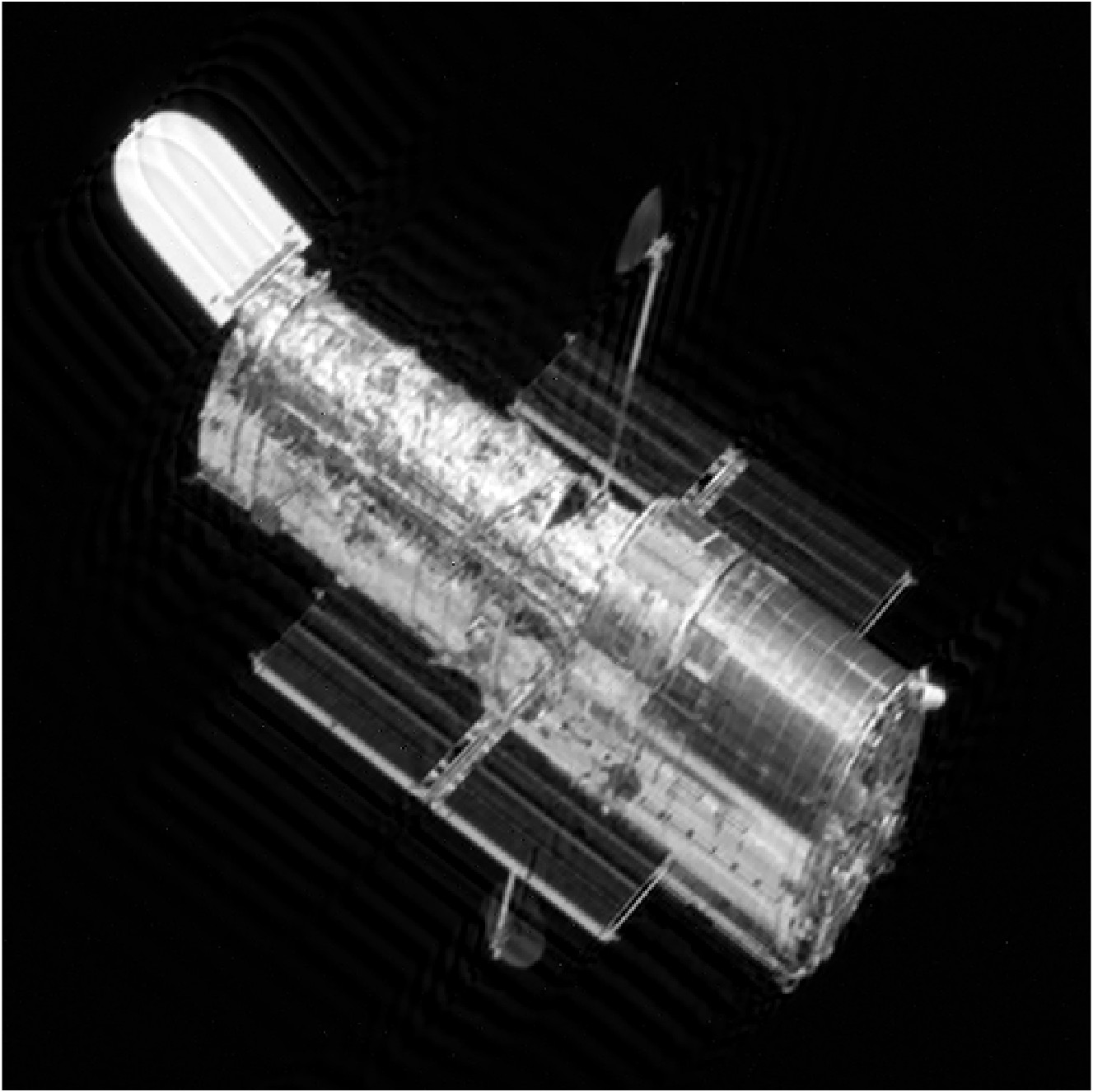} &
  \includegraphics[height = 0.13\textwidth, width = .15\textwidth]{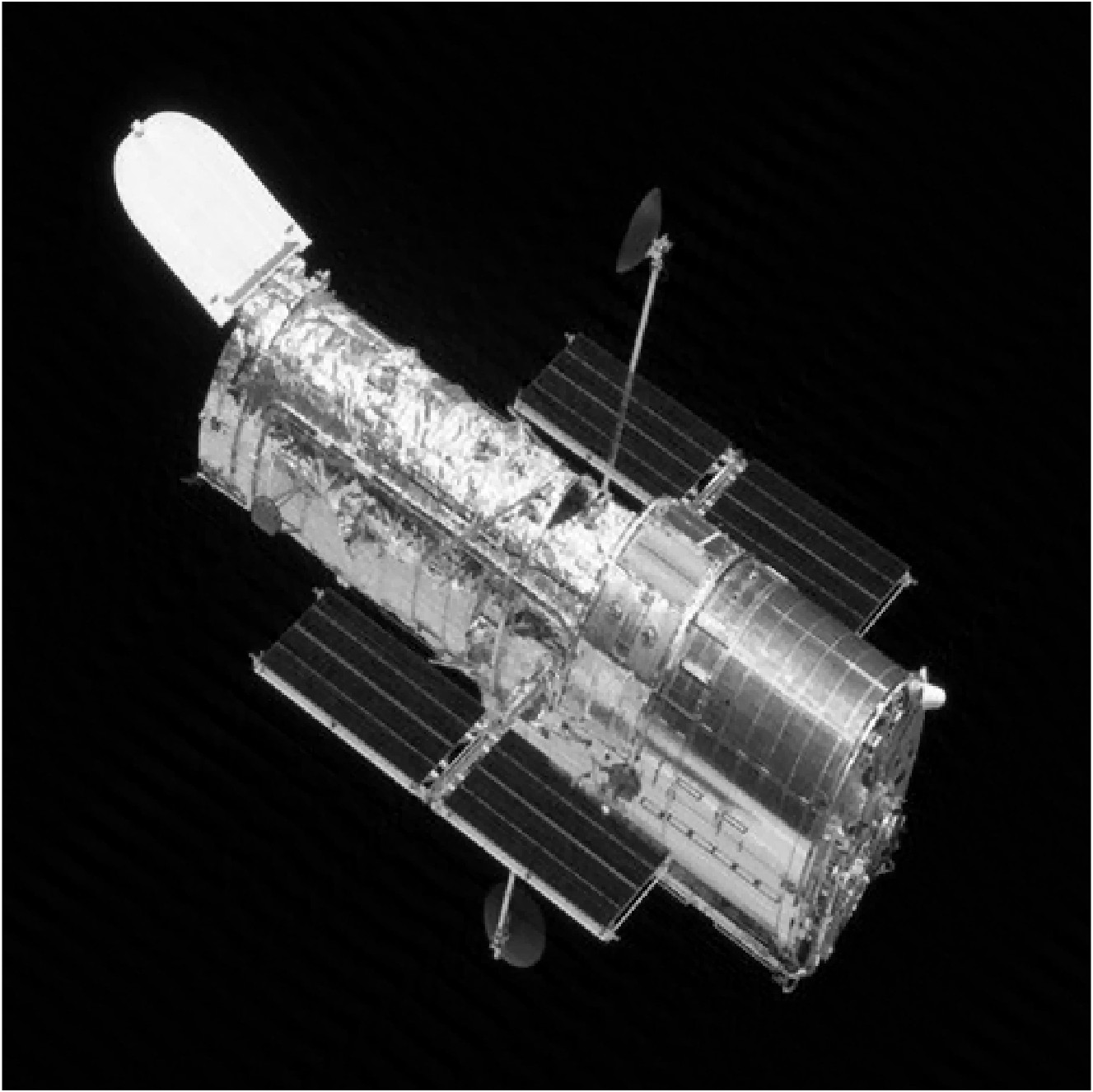} &
  \includegraphics[height = 0.13\textwidth, width = .15\textwidth]{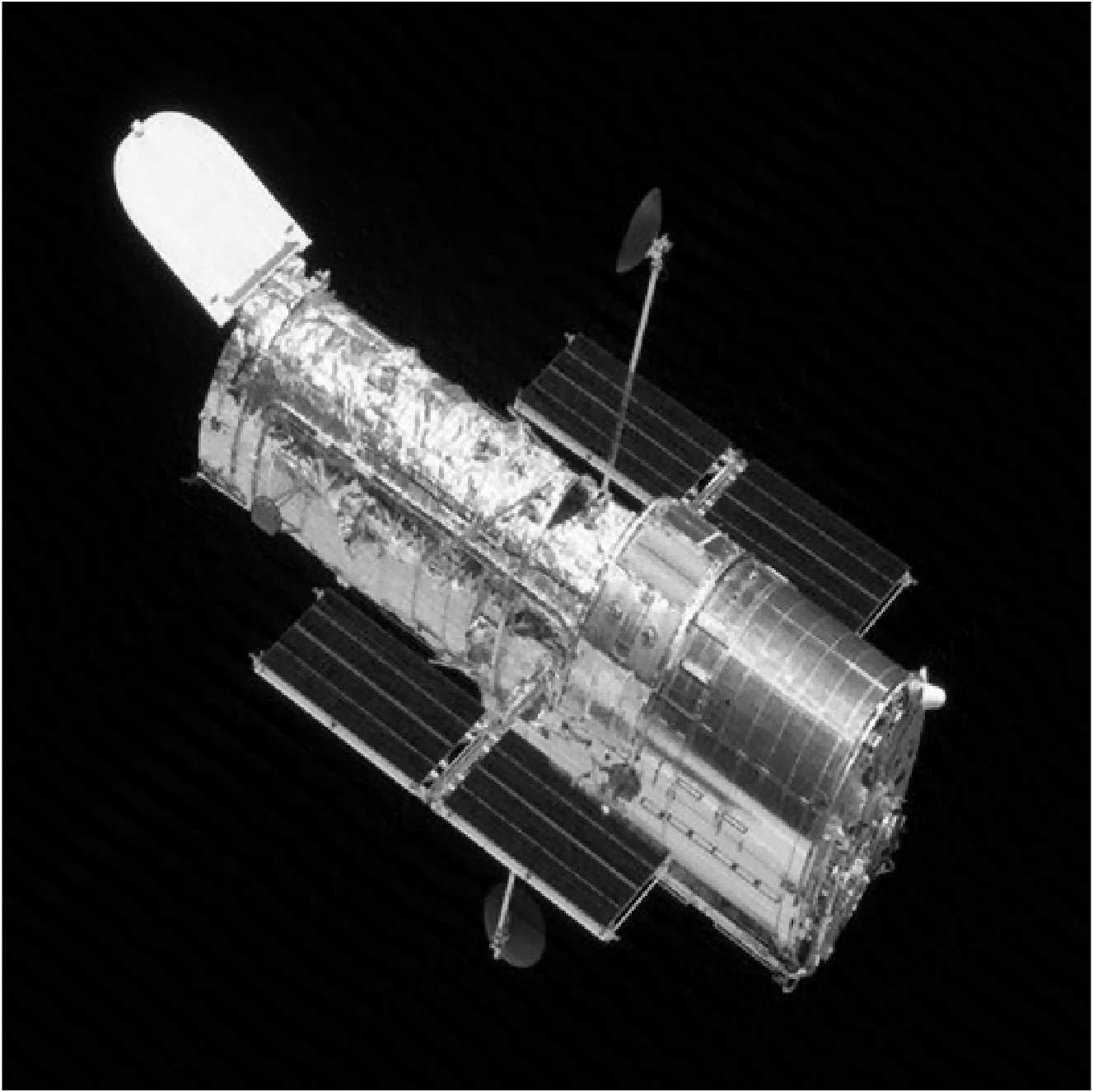}& 
    \includegraphics[height = 0.13\textwidth, width = .15\textwidth]{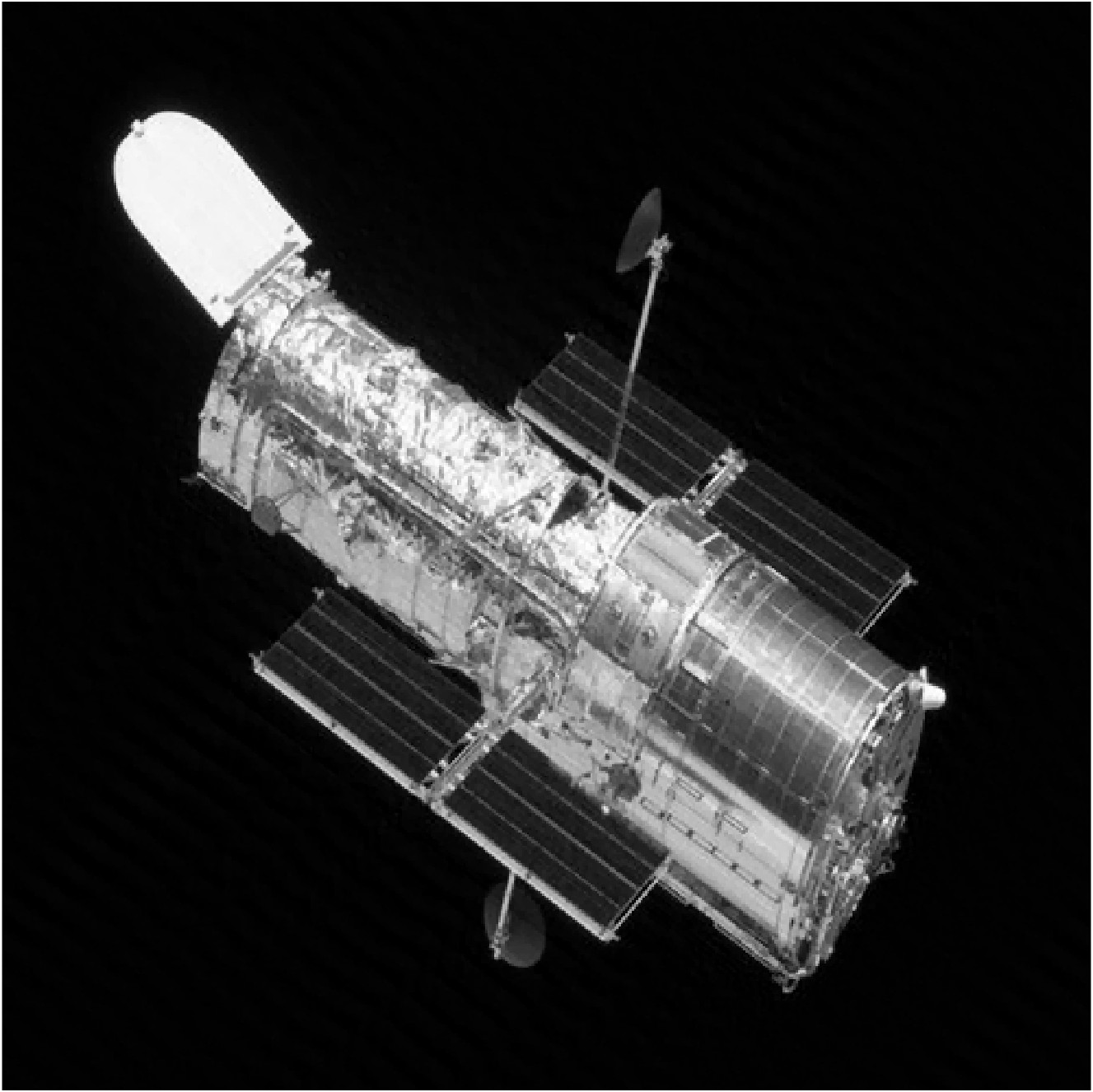}&
    \includegraphics[height = 0.13\textwidth, width = .15\textwidth]{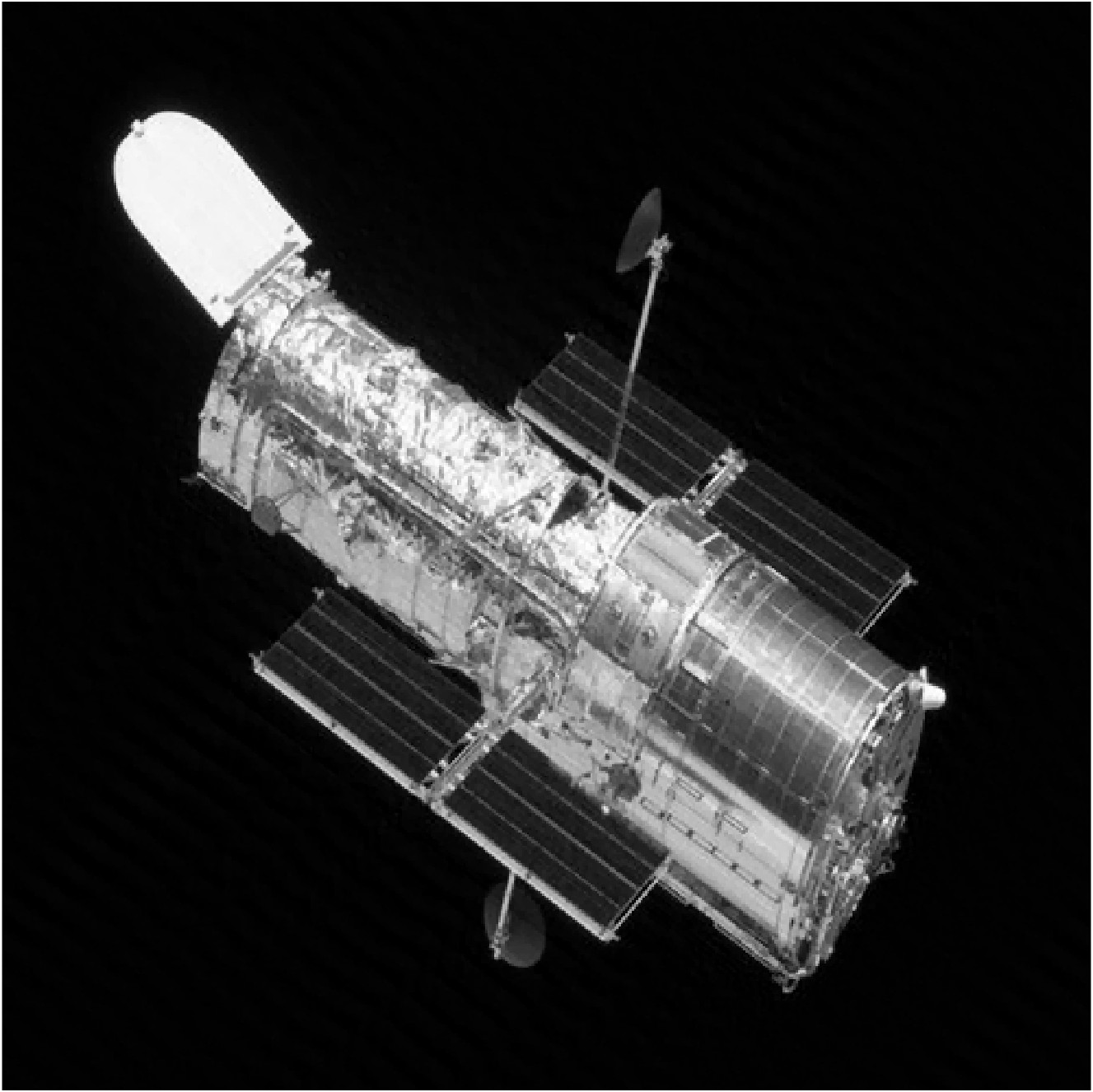} 
\end{tabular}
\begin{tabular}{ccccc}
  \includegraphics[height = 0.13\textwidth,  width = .15\textwidth]{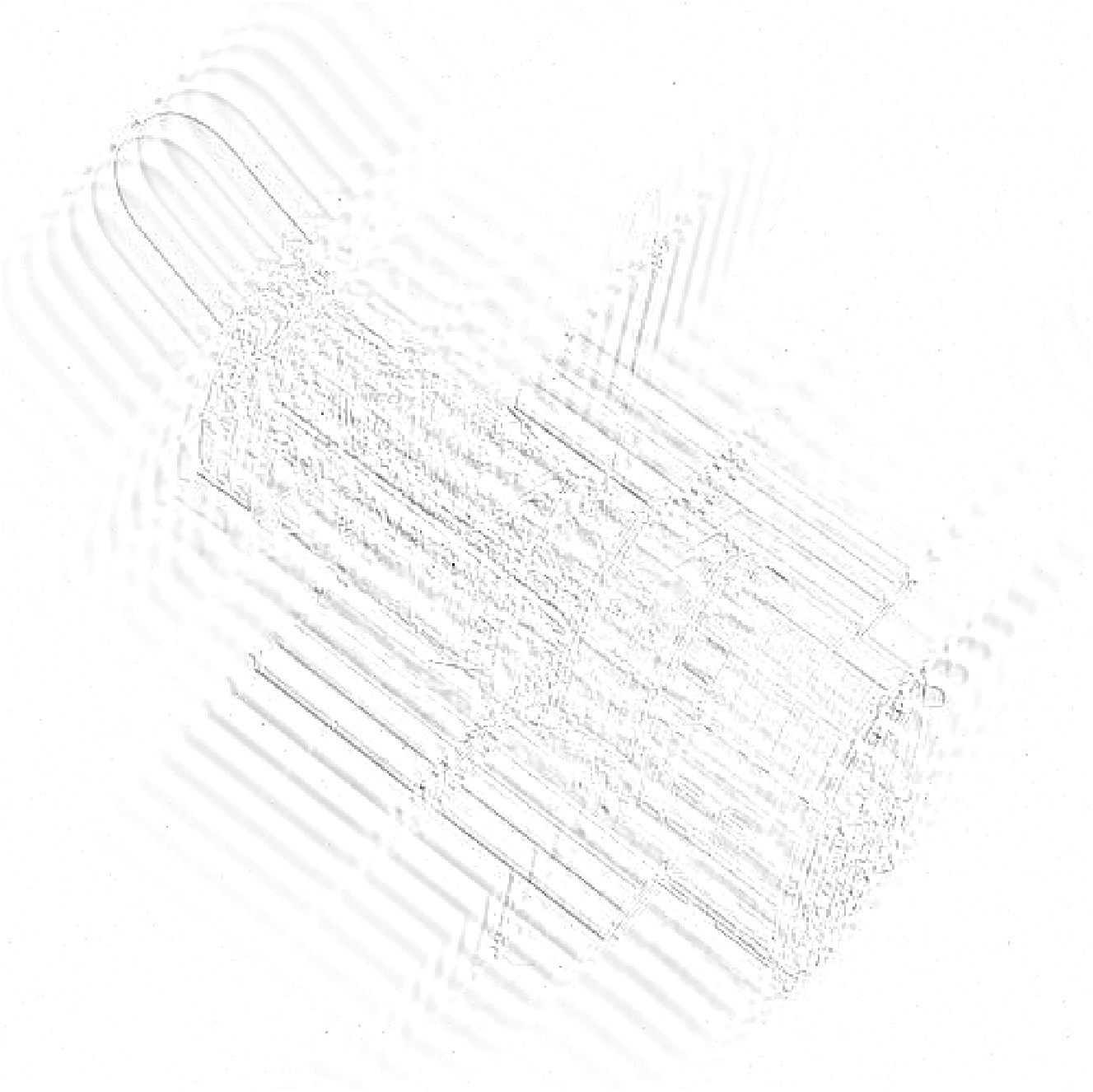}
    &
  \includegraphics[height = 0.13\textwidth,  width = .15\textwidth]{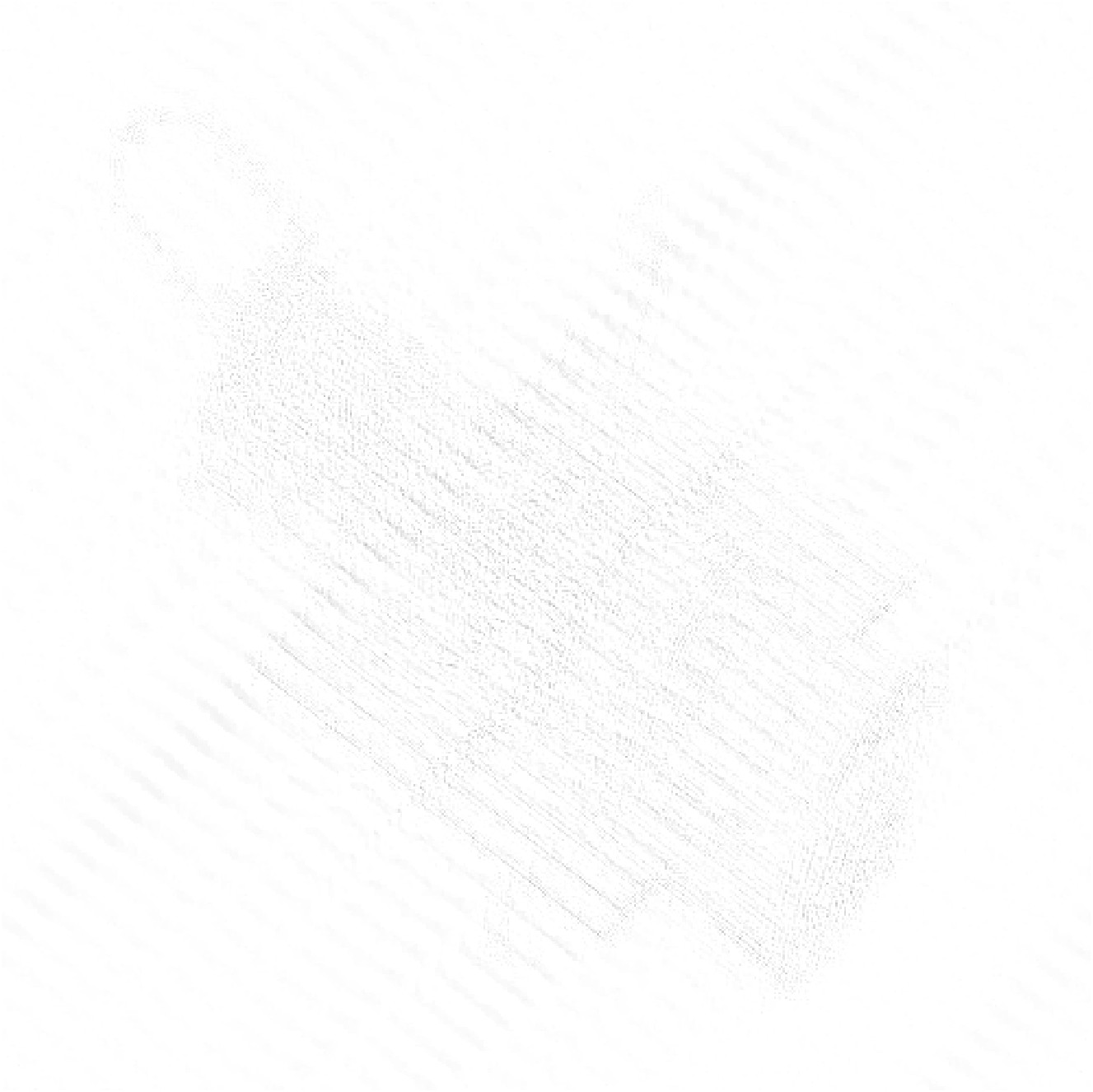}
    &
  \includegraphics[height = 0.13\textwidth, width = .15\textwidth]{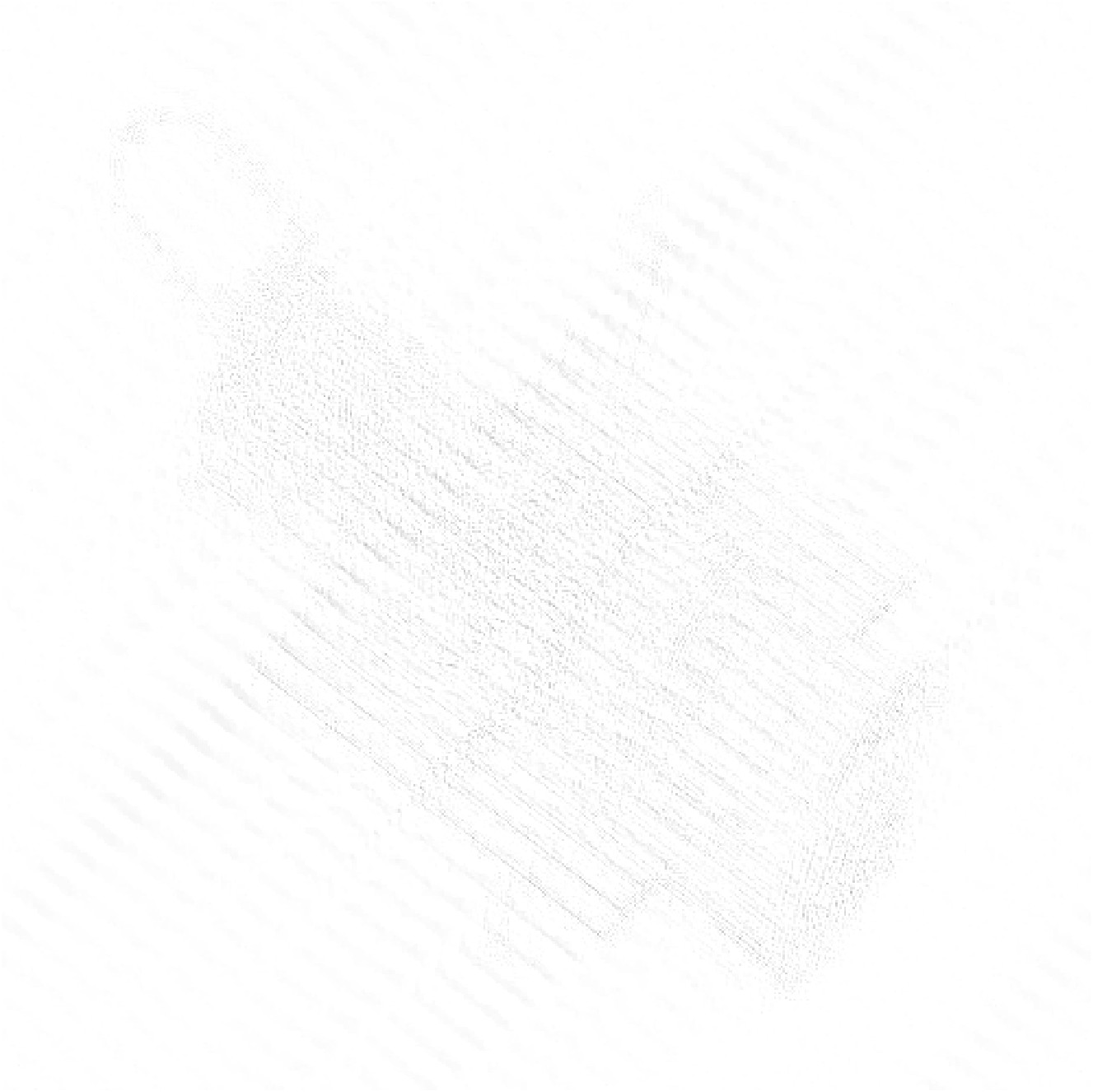}
     &
  \includegraphics[height = 0.13\textwidth, width = .15\textwidth]{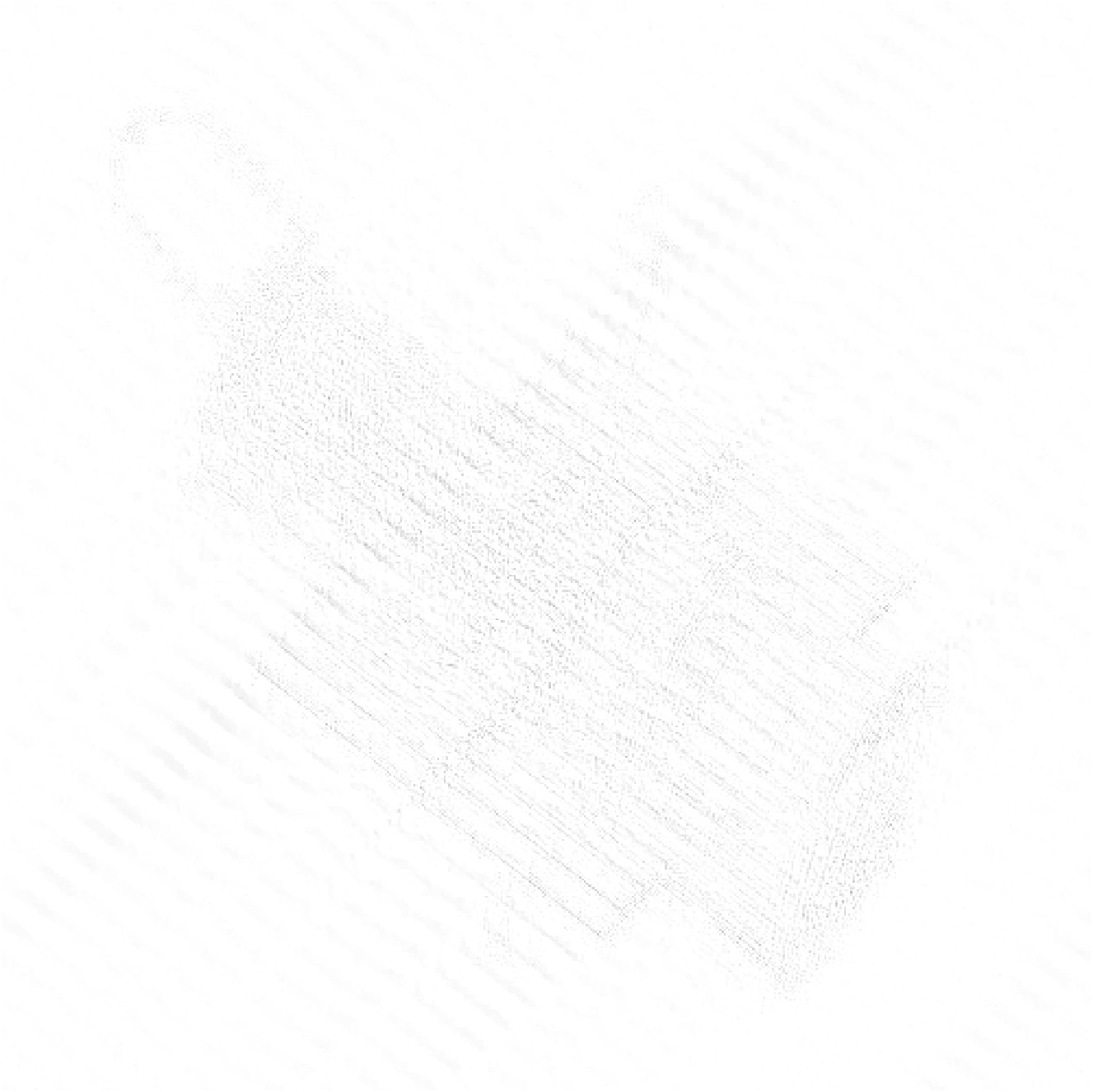}
  &
   \includegraphics[height = 0.13\textwidth, width = .15\textwidth]{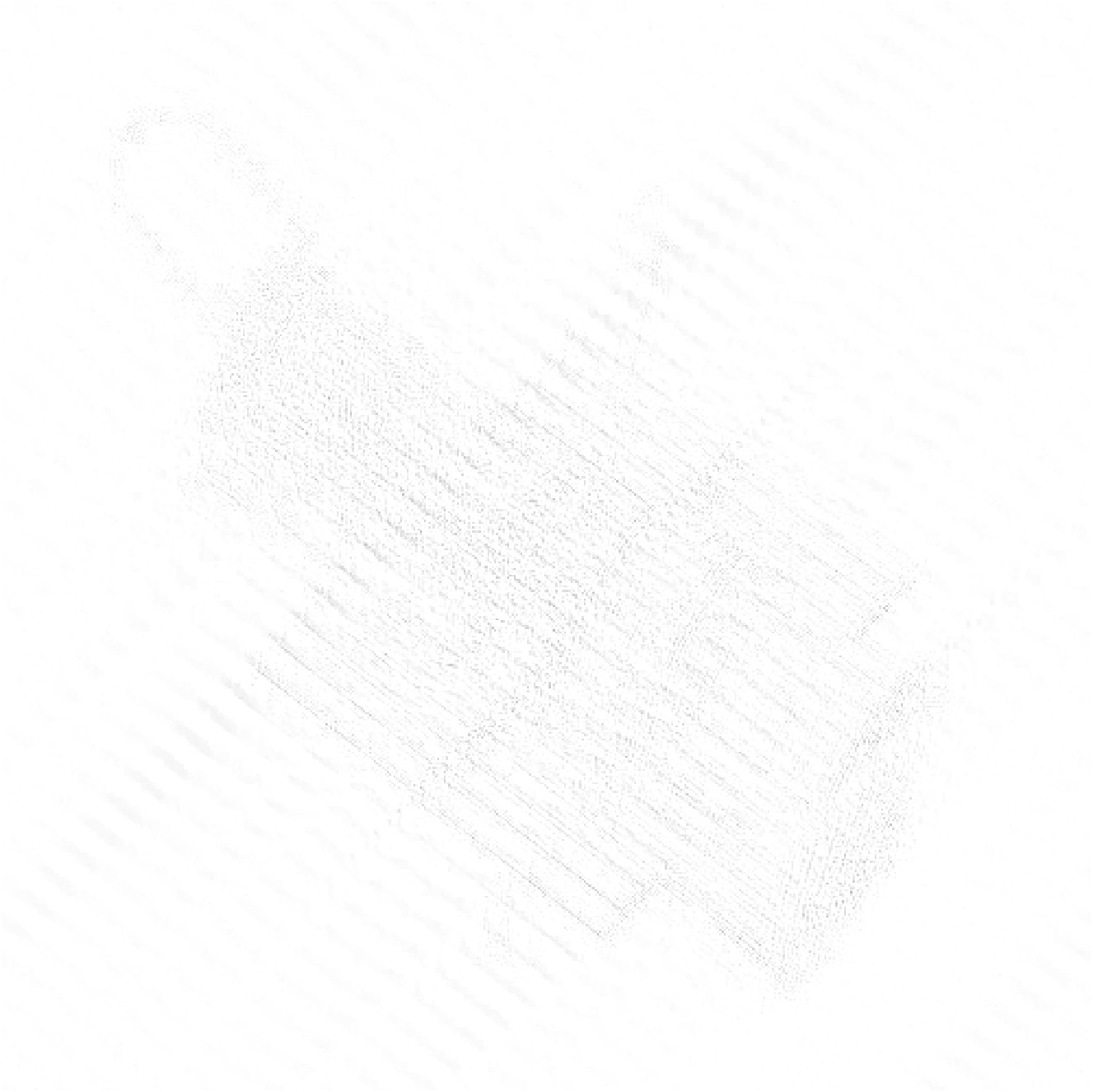}
\end{tabular}
\caption{Image deblurring (Hubble telescope). The first row shows the reconstructed images for the $0.1\%$ Gaussian noise case 
by MM-GKS (after 25 iterations with RRE = 0.1128 and HaarPSI = 0.9130) and LM-MM-GKS with compression approaches tSVD, RBD, SEC, and SOC (RREs and HPs reported in Table \ref{Table: RRE_telescope}).
The second row shows the corresponding error images.}
\label{Figure: Telescope_errors}
\end{figure}

We consider an
image deblurring problem, where the Hubble telescope image has been corrupted by motion blur. The aim is to reconstruct an approximation of the $500 \times 500$ pixels true telescope image shown in Figure \ref{Fig: telescopeImages}(a), given the observed blurred and noisy image with 0.1$\%$ Gaussian noise in Figure 
\ref{Fig: telescopeImages}(c). 
The motion blur point-spread function (PSF) of size $14\times 14$ pixels is shown in 
Figure \ref{Fig: telescopeImages}(b).
The PSF implicitly determines the blurring operator $\bA \in \R^{250,000 \times 250,000}$.  

The goals of this test are 1) to compare
the reconstruction quality of our proposed method across 
the four compression approaches (tSVD, RBD, SEC, and SOC), 2) to compare LM-MM-GKS reconstructions with those of MM-GKS when the memory capacity is limited and MM-GKS can only take a limited number of steps,
and 3) to investigate the effect
on reconstruction quality of
$k_{\min}$ (the number of solution basis vectors retained after compression). 
We set the memory capacity $k_{\max}=25$
and compute the reconstructions for the four compression approaches with 
$k_{\min} = 5, 10, \mbox{ and } 15$.
The compression tolerance 
was set to $\epsilon_{tol} = 10^{-5}$
for RBD and to $tol = 1$ for SOC. 
All methods are run until 
the relative 
difference of two consecutive reconstructions, cf. \eqref{eq: TOL1}, falls below  
$10^{-5}$, or
the maximum number of iterations (300) is reached.
Table~\ref{Table: RRE_telescope} reports the RRE and HaarPSI for LM-MM-GKS with each of the four compression approaches and three choices of $k_{\min}$ along with the number of expansion steps until convergence. 

To compare, we run MM-GKS up to a hypothetical memory capacity of $25$ basis vectors. All four LM-MM-GKS variants substantially outperform MM-GKS, with RREs significantly reduced (in the range $0.047$-$0.06$) and HaarPSI values above $0.975$. The results are stable across both the compression approaches and the values of $k_{\min}$, indicating that LM-MM-GKS is robust with respect to these choices.
The reconstructed images and the corresponding error images in inverted colormap (darker regions indicate larger error) are shown in Figure~\ref{Figure: Telescope_errors}; the difference in quality between MM-GKS and LM-MM-GKS is particularly evident in the error images.

Since all four compression approaches yield comparable reconstruction quality on this test problem, the choice of compression method is not critical. In the subsequent experiments, we use tSVD for compression unless stated otherwise, as it is straightforward to implement and performed consistently well across all settings.

{\color{black}{
\begin{table}[ht!]
\centering
\small
\setlength{\tabcolsep}{4pt}
\begin{tabular}{@{} c ccc ccc ccc ccc @{}}
\toprule
 & \multicolumn{3}{c}{tSVD}
 & \multicolumn{3}{c}{RBD}
 & \multicolumn{3}{c}{SOC}
 & \multicolumn{3}{c}{SEC} \\
\cmidrule(lr){2-4} \cmidrule(lr){5-7} \cmidrule(lr){8-10} \cmidrule(lr){11-13}
$k_{\rm min}$ & RRE & HP & iter & RRE & HP & iter & RRE & HP & iter & RRE & HP & iter \\
\midrule
5  & 0.0471 & 0.9861 & 280 & 0.0467 & 0.9865 & 300 & 0.0490 & 0.9848 & 300 & 0.0490 & 0.9848 & 300 \\[2pt]
10 & 0.0499 & 0.9839 & 285 & 0.0510 & 0.9831 & 285 & 0.0505 & 0.9835 & 300 & 0.0503 & 0.9837 & 300 \\[2pt]
15 & 0.0544 & 0.9800 & 250 & 0.0526 & 0.9816 & 300 & 0.0591 & 0.9766 & 300 & 0.0603 & 0.9758 & 300 \\
\bottomrule
\end{tabular}
\caption{Image deblurring (Hubble telescope). RRE and HaarPSI for $k_{\rm min} = 5, 10, 15$, and $k_{\rm max}= 25$ for the $500\times 500$ pixels Hubble telescope image with motion blur. LM-MM-GKS with each compression variant is run
until the relative difference of consecutive reconstructions falls below $10^{-5}$
or the maximum of $300$ expansion steps is reached. The column ``iter'' reports the number of expansion steps used; values less than $300$ indicate early convergence.}
\label{Table: RRE_telescope}
\end{table}
}}

\subsection{Computerized tomography}
\label{sec: CompTomo}
In this section, we test the streaming version of LM-MM-GKS described in Section \ref{sec: streaming}.  
In the ideal 
case, we would solve (\ref{eq:rtomoproblemall}). However, we consider two scenarios where this is not possible. For Test 1, we assume that 
only a fraction of the sinogram data is available for processing at any given time. 
For Test 2, we assume 
that all the data has been collected, but the system is too large to fit in memory. 
Therefore, in both cases, we use s-LM-MM-GKS to approximately solve instead the $n_t$
subproblems (\ref{eq:rtomo1})-(\ref{eq:rtomont}) and compare with HyBR-recycle.  

\subsubsection{Test 1} 
We consider the parallel tomography example from IRTools \cite{gazzola2018ir}
with 
a $500 \times 500$ pixels Shepp-Logan phantom. 
The data is streamed in three groups; $n_t = 3$ in \eqref{eq:rtomont}. The first and the second subproblems correspond to projection angles from $0^{\circ}-44^{\circ}$ and  $45^{\circ}-89^{\circ}$, respectively, with an angle gap of $1^{\circ}$. The third subproblem corresponds to $45$  
projection angles from  $90^{\circ}$ to $179^{\circ}$ with an angle gap of $2^{\circ}$.
All subproblems are based on $707$ parallel rays. 
This setup produces forward operators and observations $\bA_{1}, \bA_{2}, \bA_3 \in \R^{ 45 \cdot 707 \times 500^2}$ and $\bd_1, \bd_2, \bd_3 \in \R^{45\cdot 707}$, respectively. We 
add $0.1\%$ noise to each data
vector.
The sinograms and the true image are provided in Figure \ref{fig: Streaming3prob_true}. We set $k_{max}=40$ and $k_{min}=10$ for all limited memory methods described below (HyBR-recycle, LM-MM-GKS, and s-LM-MM-GKS).
Based on the image deblurring results in Section~\ref{Sec: ImgDebl} (Table~\ref{Table: RRE_telescope}), where all four compression approaches yield comparable reconstruction quality and tSVD exhibited the fastest convergence, we use tSVD compression for all LM-MM-GKS variants in this and subsequent experiments.

\begin{figure}[htbp]
\centering
  \begin{tabular}{cccc}
$\bx_{\rm true}$ &  $\bd_1$: $0^{\circ}-44^{\circ}$ &  $\bd_2$: $45^{\circ}-89^{\circ}$ & $\bd_3$: $90^{\circ}-179^{\circ}$\\
  \includegraphics[angle = 90, height = 0.15\textwidth, width = .18\textwidth]{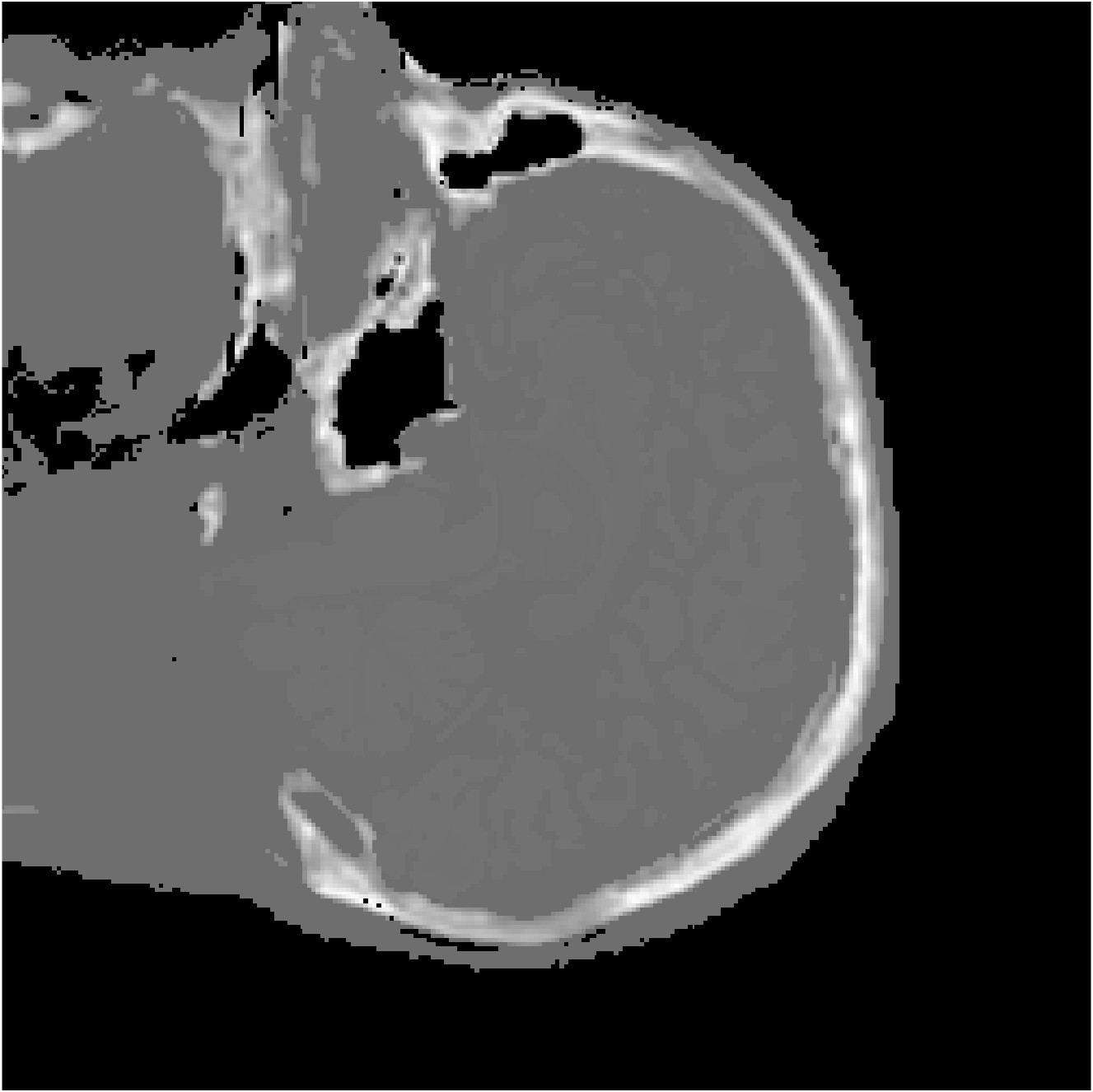} &
  \includegraphics[angle = 90,height = 0.1\textwidth, width = .22\textwidth]{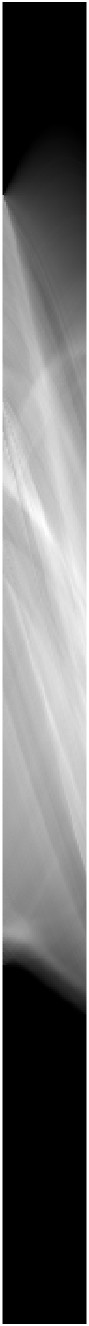} &
  \includegraphics[angle = 90,height = 0.1\textwidth, width = .22\textwidth]{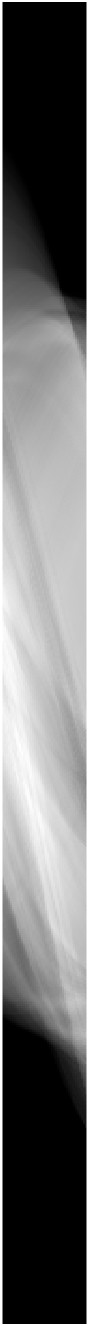} 
   &
  \includegraphics[angle = 90,height = 0.1\textwidth, width = .22\textwidth]{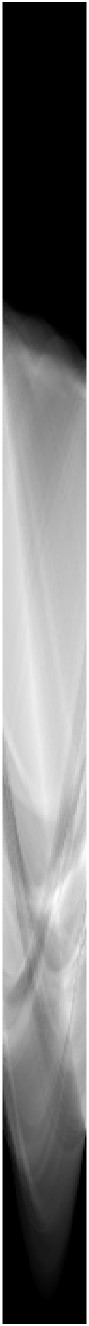}\\
(a) & (b)  & (c)  & (d)
\end{tabular}
  \caption{Test 1: Streaming tomography example. The $500 \times 500$ pixels true image is provided in (a), along with the three observed sinograms,
  $\bd_1$ in (b), $\bd_2$ in (c), and $\bd_3$ in (d), corresponding to projections taken at angles: $0^{\circ}-44^{\circ}$,  $45^{\circ}-89^{\circ}$, and $90^{\circ}-179^{\circ}$, respectively.}
  \label{fig: Streaming3prob_true}
\end{figure}
We compare s-LM-MM-GKS to the recently proposed hybrid projection method with recycling, HyBR-recycle \cite{jiang2021hybrid}, which uses tSVD type compression for its recycled subspace.
HyBR-recycle also employs subspace recycling in the streaming data setting and selects the regularization parameter using the discrepancy principle (DP) within the projected problem. Although GCV was considered to be consistent with the parameter-selection strategy used in the LM-MM-GKS and MM-GKS methods, it did not produce satisfactory results when used with HyBR-recycle for this problem. Therefore, we use the HyBR-recycle with DP regularization parameter selection instead.   
For the purpose of subspace expansion, recycling GKB is used to enlarge the solution search space in HyBR-recycle instead of GKS used in LM-MM-GKS methods. Moreover, HyBR-recycle uses $\Psi = \bI$ as a regularization operator and the 
$\ell_2$-norm. Thus, the method generates different solution spaces than LM-MM-GKS and selects the recycle space differently.
We expect that while the method may be competitive in terms of memory required
and runtime, it may not be qualitatively competitive, because it does not effectively enforce edge-constraints as s-LM-MM-GKS does.
However, we include HyBR-recycle in our comparisons to demonstrate that the nonlinear approach and the solution of the more structured regularization problem in s-LM-MM-GKS yields improved reconstructions, particularly for the streaming example. We consider the following setting.
\begin{enumerate}
    \item We run HyBR without recycling \cite{chung2008weighted} on the first subproblem \eqref{eq:rtomo1} 
    (HyBR 1st) 
    and on the full problem 
    \eqref{eq:rtomoproblemall} (HyBR all).
    \item We run HyBr on  
    the first subproblem \eqref{eq:rtomo1} 
    to obtain an approximate solution and an initial subspace. This information is used in solving the second subproblem \eqref{eq:rtomoi} with HyBR-recycle \cite{jiang2021hybrid}, and so on for the third subproblem 
    \eqref{eq:rtomont}.  
    We refer to this approach
    as HyBR-recycle or HyBR-rec.
    \item We run MM-GKS on the first subproblem (MM-GKS 1st) and on the full problem (MM-GKS all data). 
    \item We run LM-MM-GKS on the first subproblem \eqref{eq:rtomo1} (LM-MM-GKS 1st) and on all the data \eqref{eq:rtomoproblemall} (LM-MM-GKS all) 
   \item  
   First, we run s-LM-MM-GKS for a fixed number of {expansion steps} for each subproblem, choosing $i_{\max}$ in Algorithm ~\ref{Alg: RMMGKS} such that {$210$} {expansion steps are performed} for each subproblem. Second, we run a version that exits the Enlarge routine  and moves on the next subproblem (or ends) when the relative change of two consecutive approximations falls below $10^{-3}$; see \eqref{eq: TOL1}. 
\end{enumerate}
{For the first subproblem for s-LM-MM-GKS, the initial subspace of dimension $k_{\min} = 10$ is constructed by first using $\ell = 15$ steps of GKB bidiagonalization (Algorithm~\ref{Alg: RMMGKS}, line 10), followed by building  $\bV_{k_{\min}}$ from a $k_{\min}$-dimensional Krylov subspace associated with $\bA^T\bA + \lambda \bPsi^T \bP_{\epsilon}^2 \bPsi$, augmented with $\bx^{(1)}$ and $\br^{(1)}$ (lines 11-24).
For subsequent subproblems, the recycled compressed subspace $\bV_0$ of dimension $k_{\min} = 10$ is used directly.} The reconstructed images with their error images, i.e., the reconstructed image minus the true image, in inverted color map with a unified scaling (-1, 0) for all the error images, are shown in Figure \ref{fig: Streaming3prob_reconstructions}. 
\begin{figure}[bh!]
\centering
  \begin{tabular}{cccc}
   HyBR 1st &  MM-GKS 1st &  LM-MM-GKS 1st & s-LM-MM-GKS \\
  \includegraphics[angle=90, height = 0.15\textwidth, width = .17\textwidth]{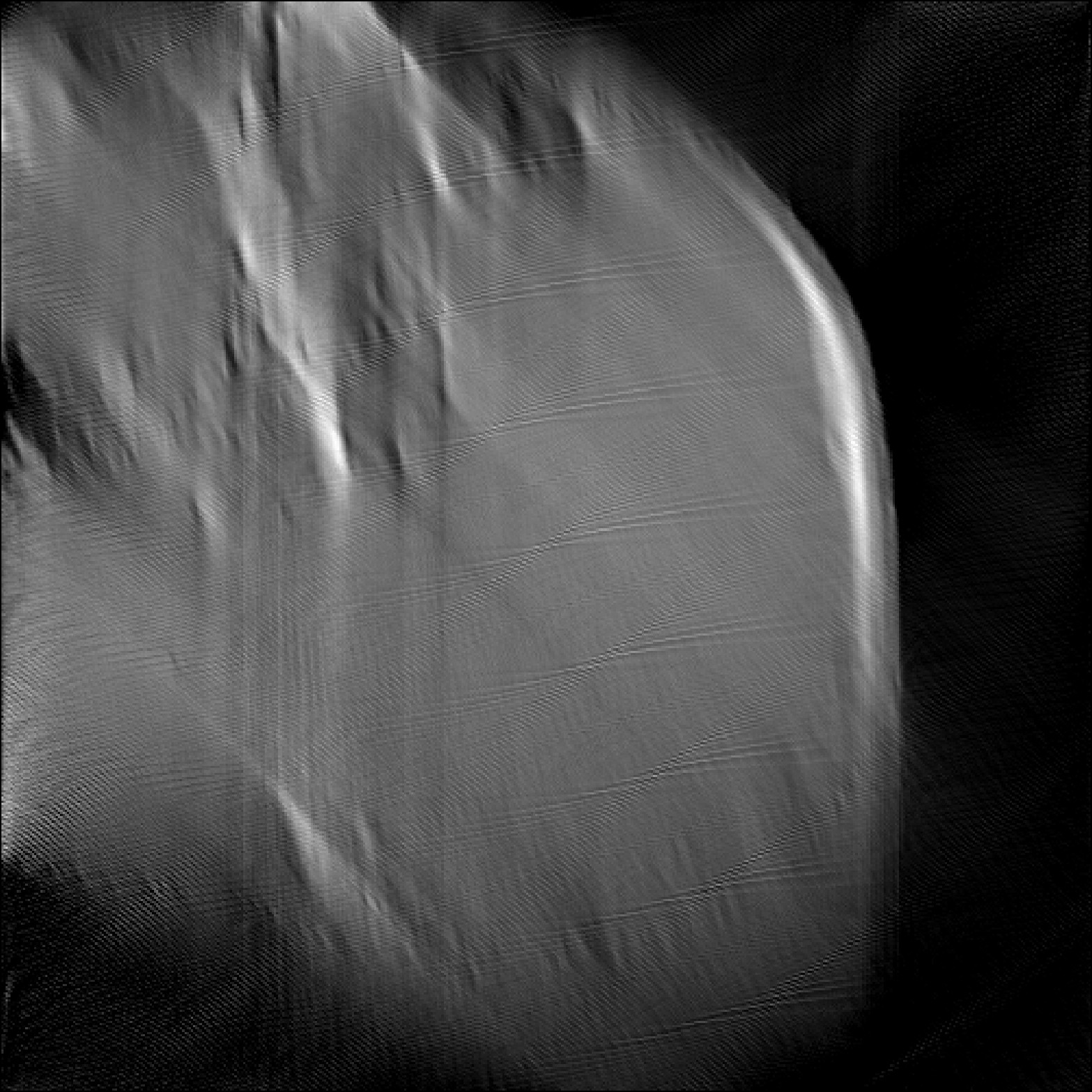} &
  \includegraphics[angle=90, height = 0.15\textwidth, width = .17\textwidth]{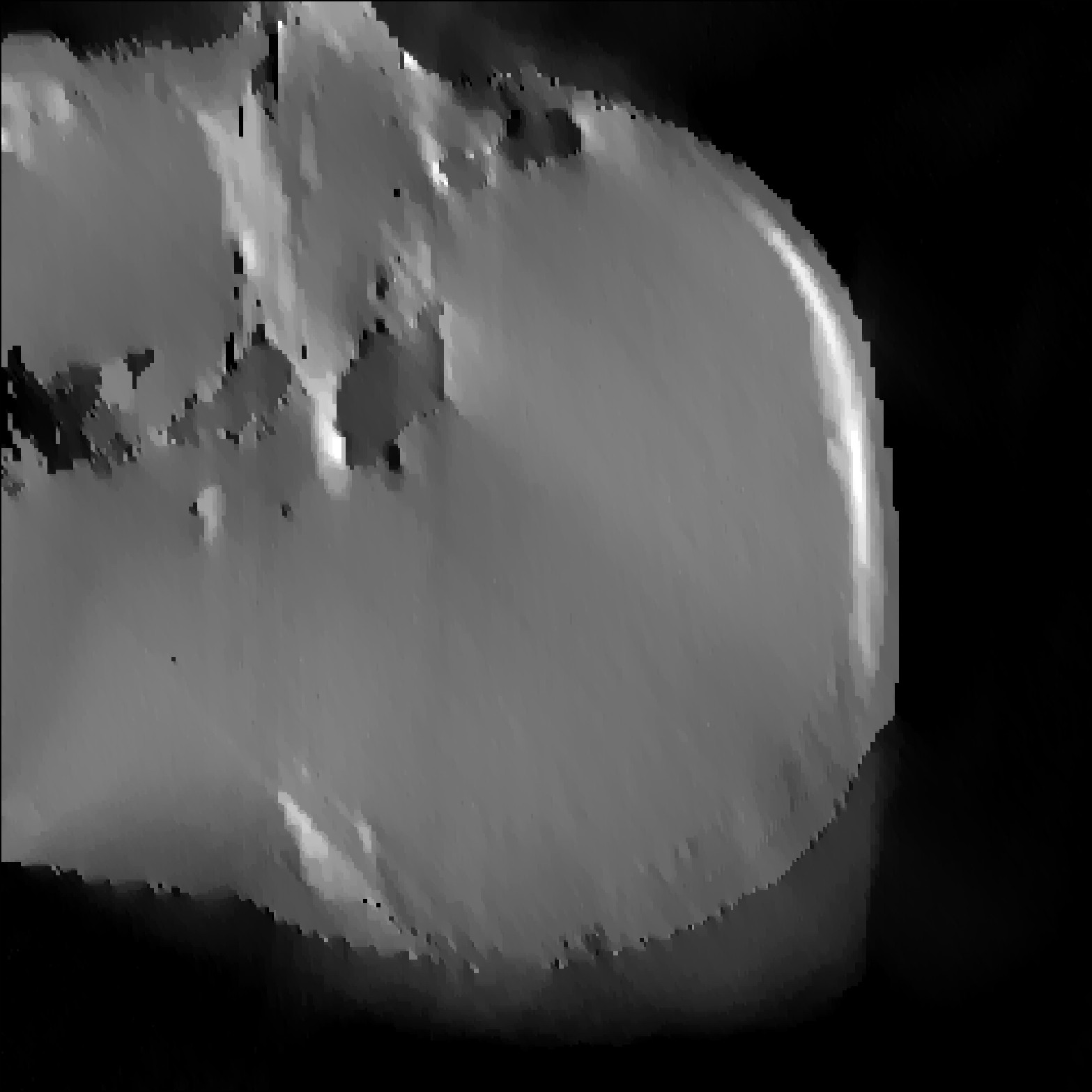} &
  \includegraphics[angle=90, height = 0.15\textwidth, width = .17\textwidth]{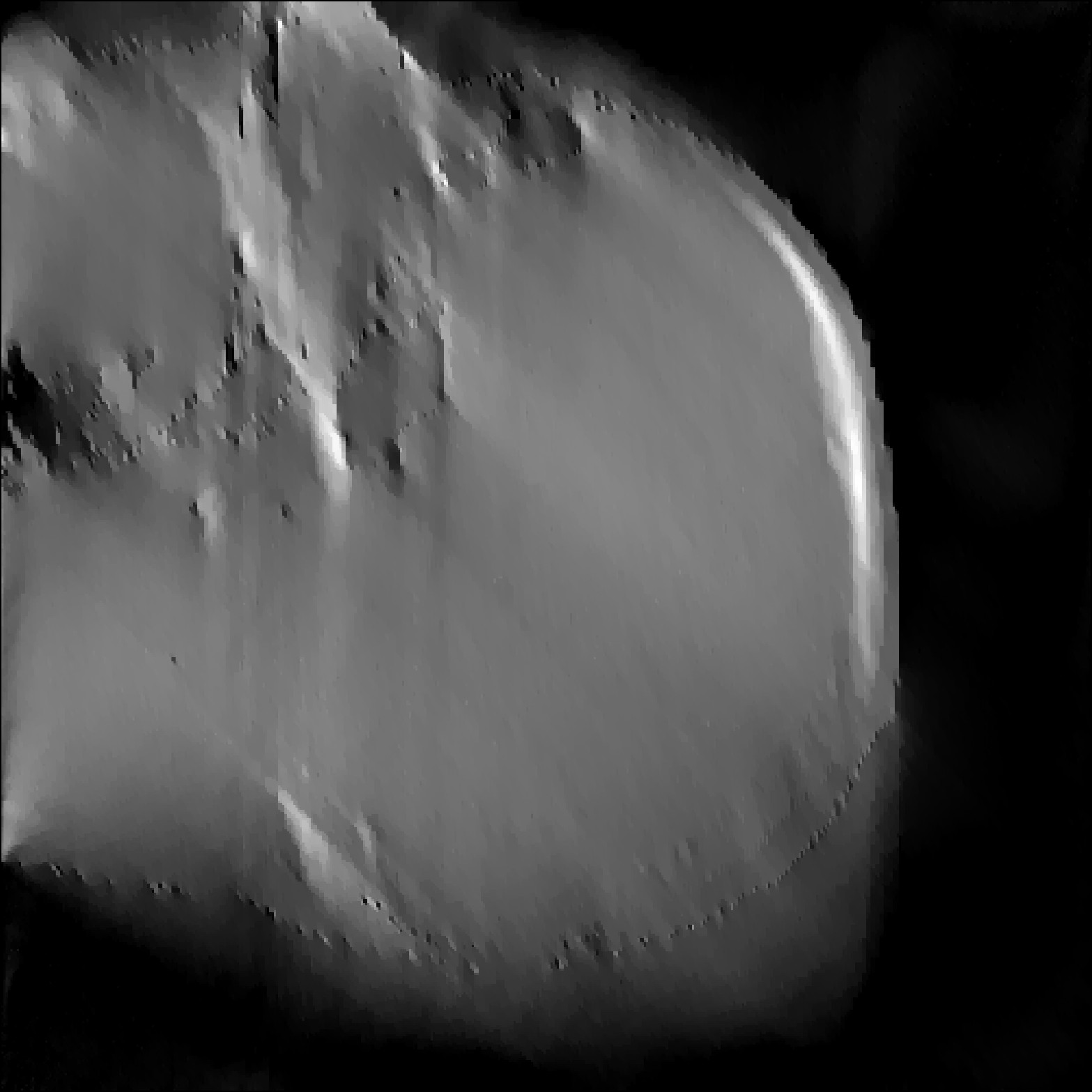}&
    \includegraphics[angle=90, height = 0.15\textwidth, width = .17\textwidth]{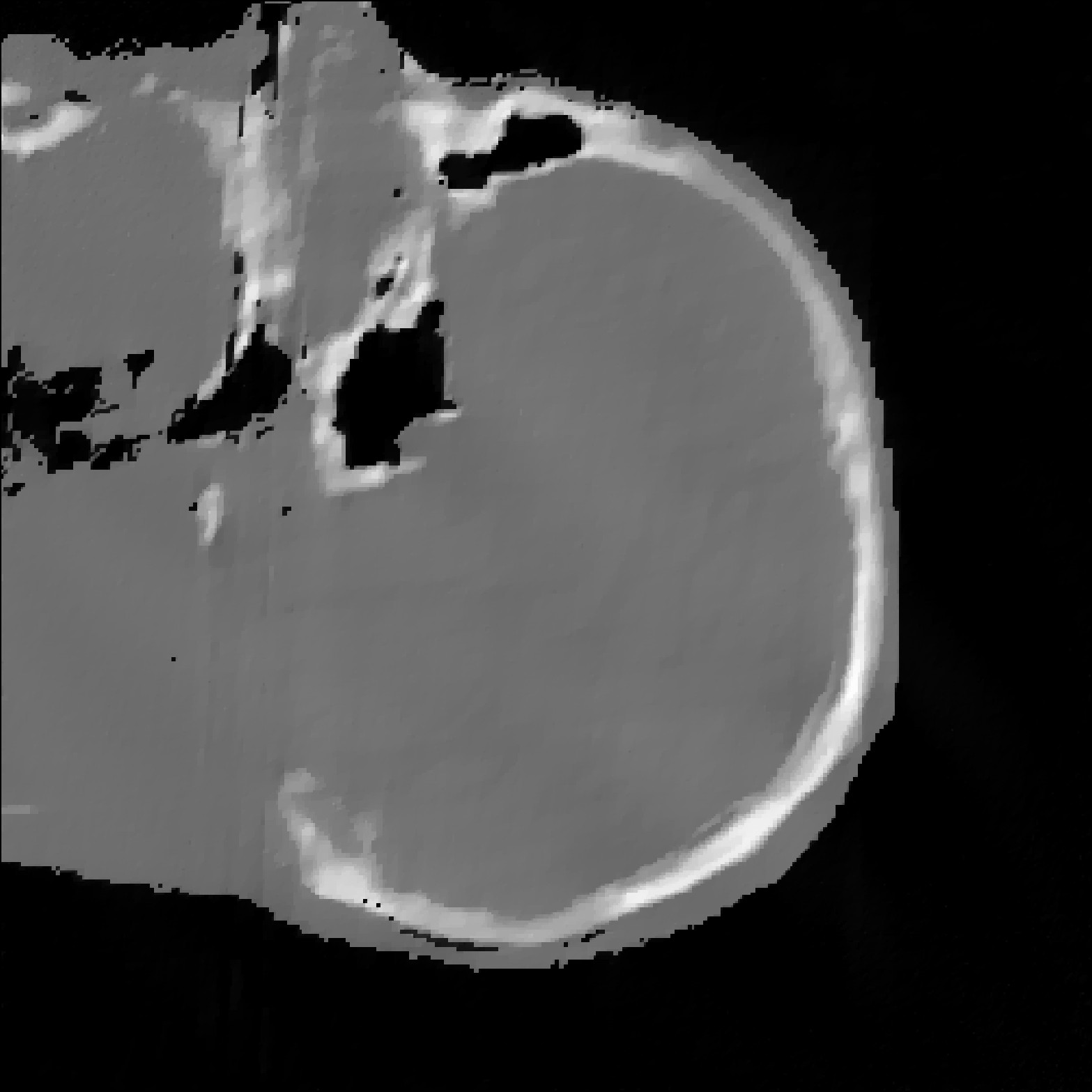}
\end{tabular}
\begin{tabular}{cccc}
  \includegraphics[angle=90, height = 0.15\textwidth, width = .17\textwidth]{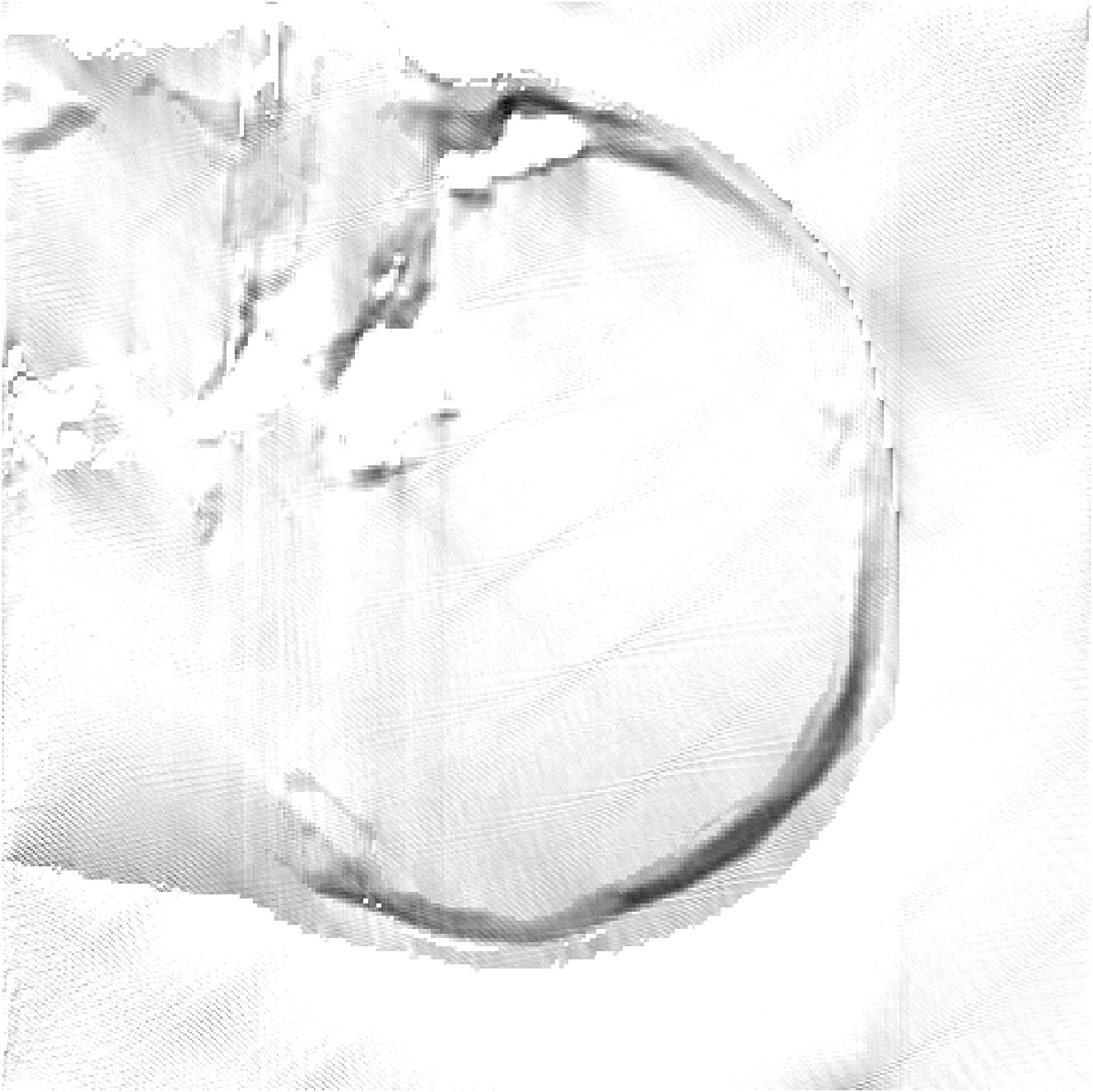}
    &
  \includegraphics[angle=90, height = 0.15\textwidth, width = .17\textwidth]{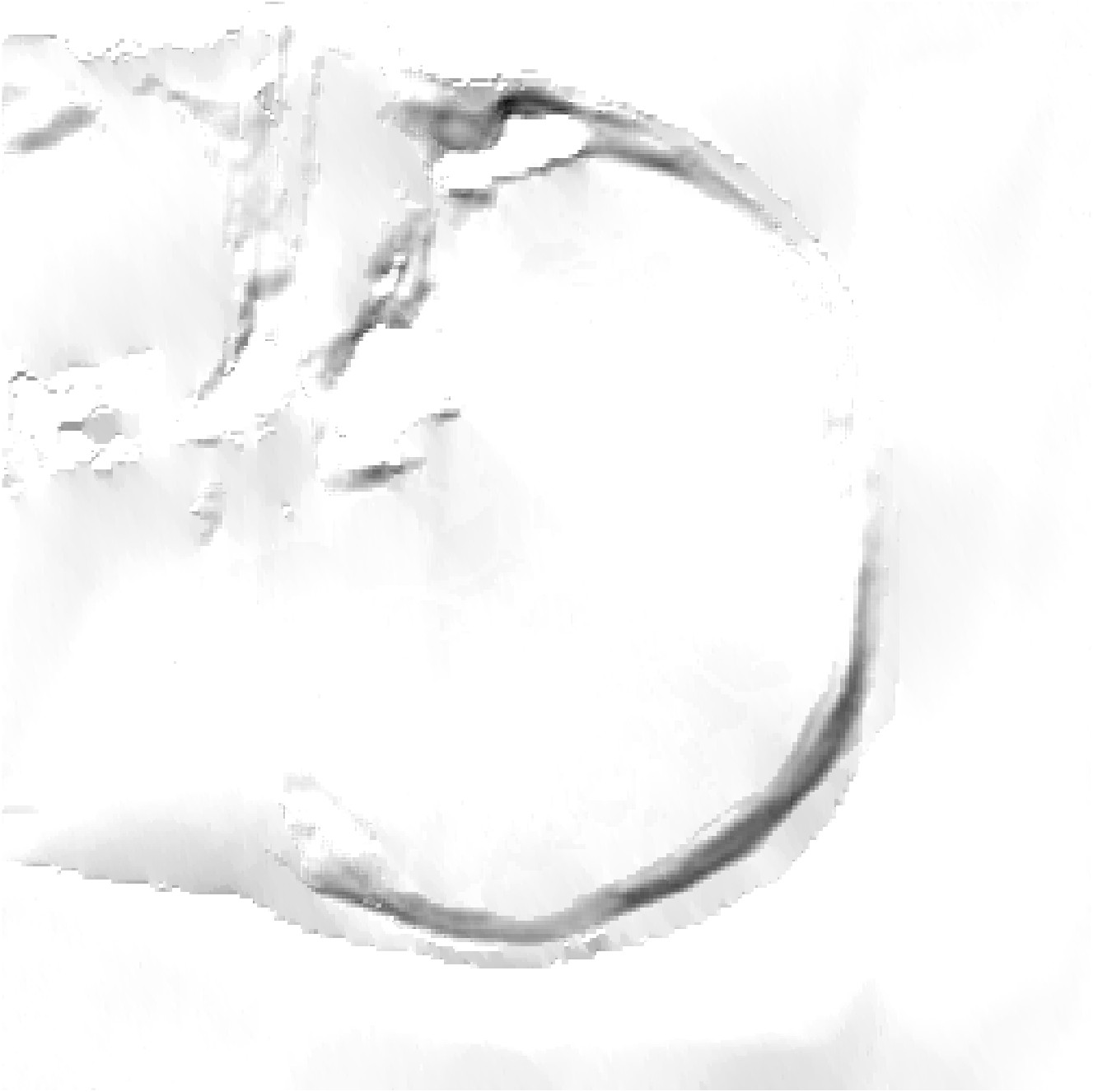}
    &
  \includegraphics[angle=90, height = 0.15\textwidth, width = .17\textwidth]{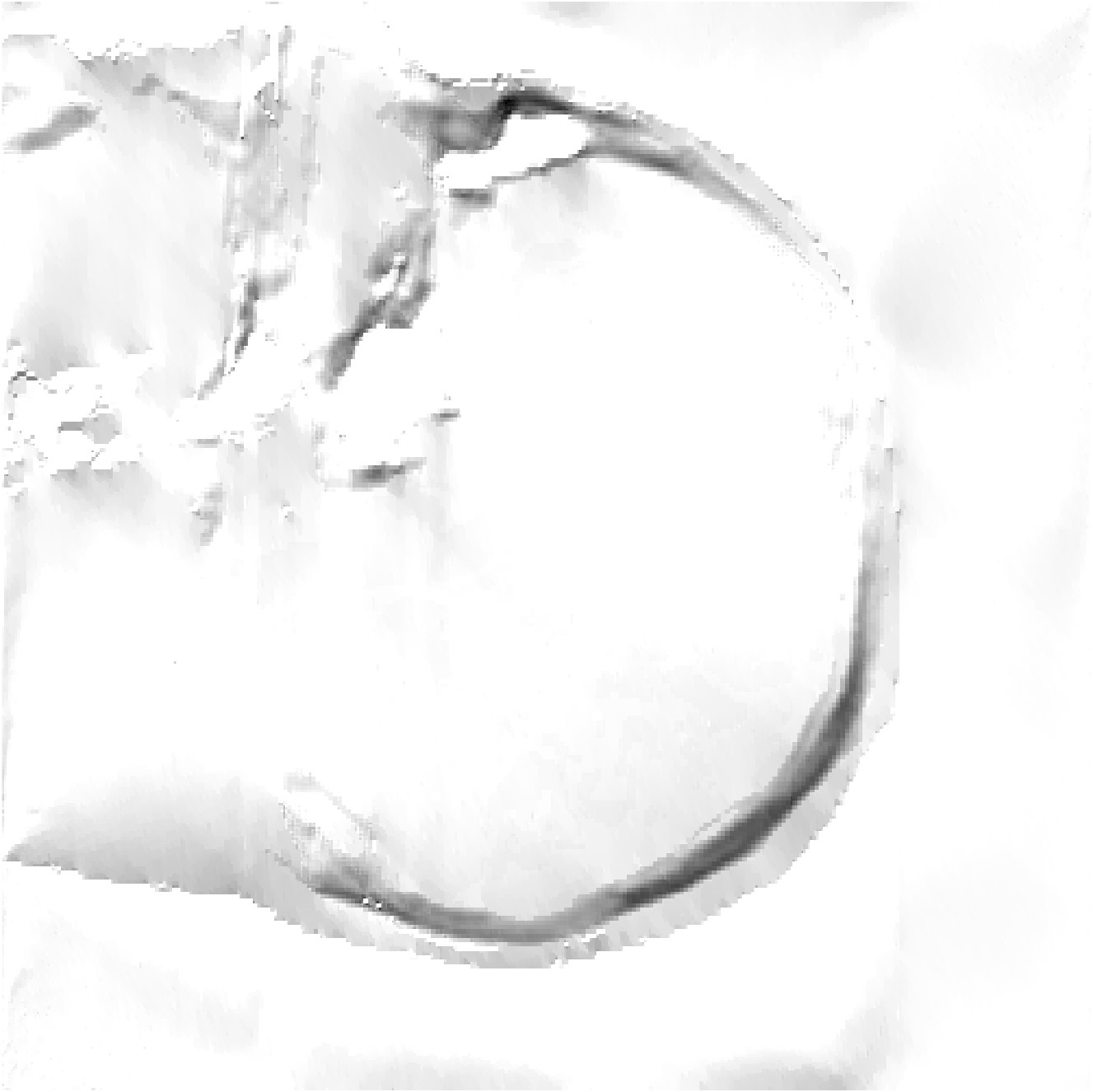}
     &
  \includegraphics[angle=90, height = 0.15\textwidth, width = .17\textwidth]{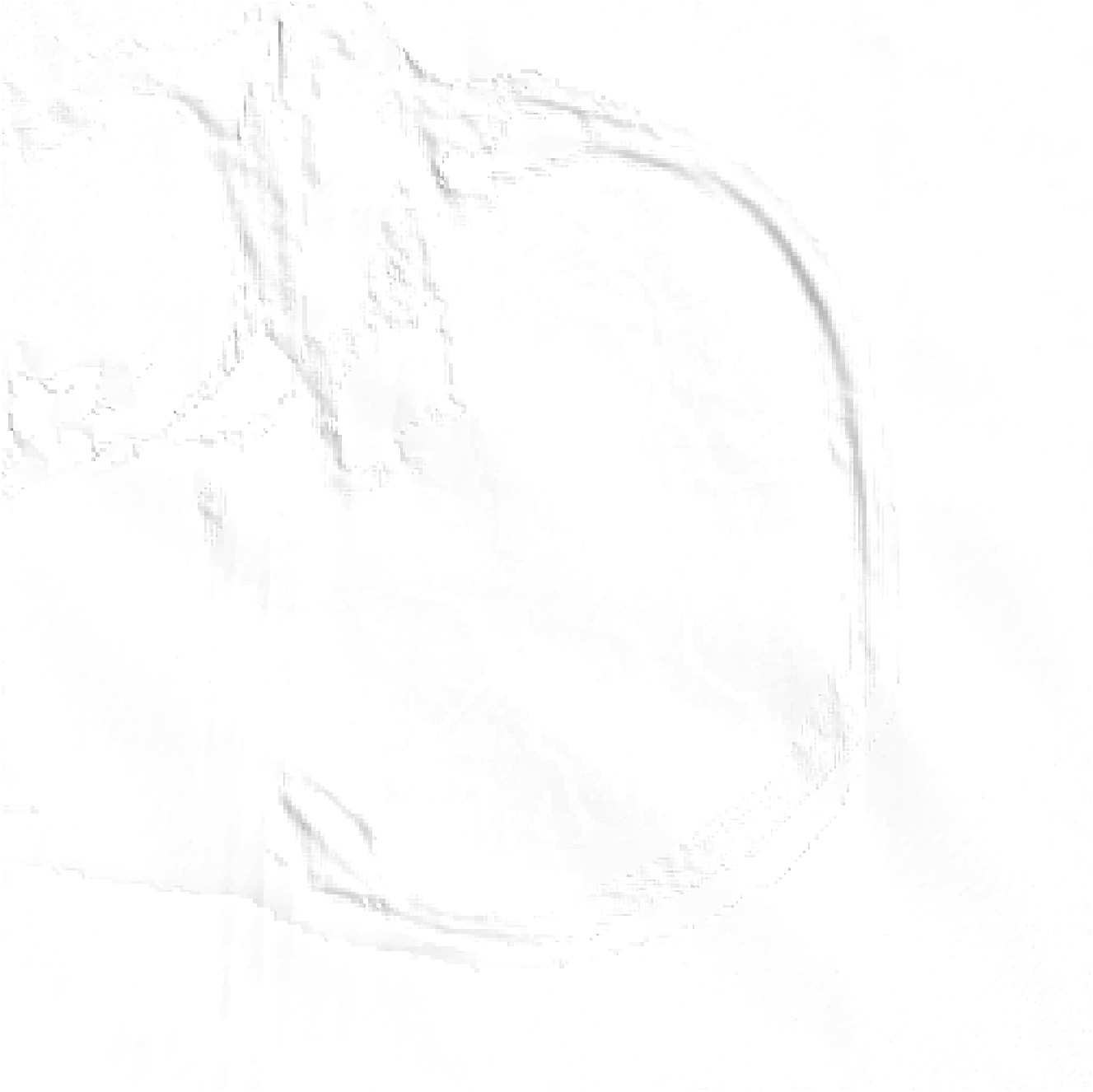}
\end{tabular}
\begin{tabular}{cccc}
HyBR all& HyBR-rec & MM-GKS all & LM-MM-GKS all \\
  \includegraphics[angle=90, height = 0.15\textwidth, width = .17\textwidth]{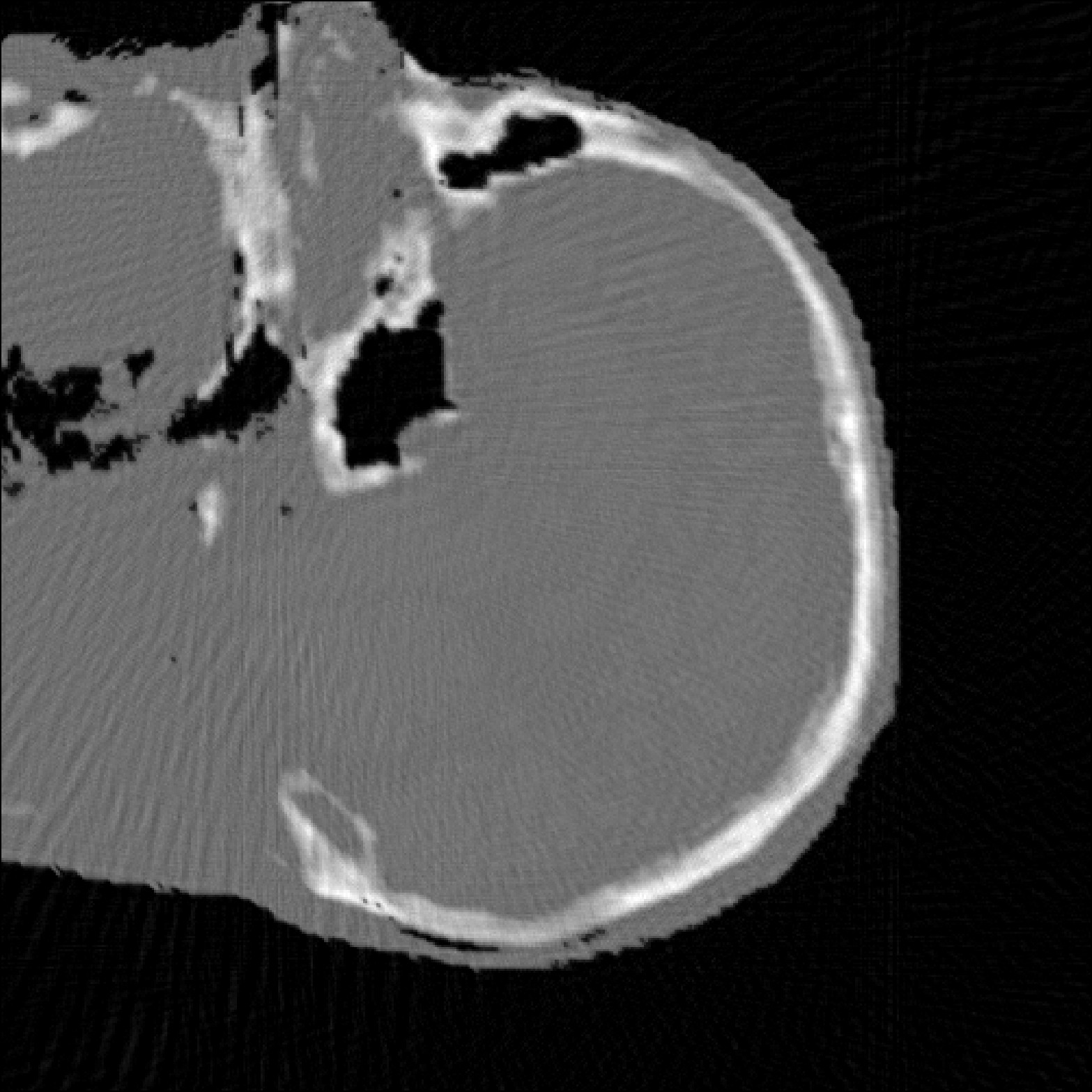}&
  \includegraphics[angle=90, height = 0.15\textwidth, width = .17\textwidth]{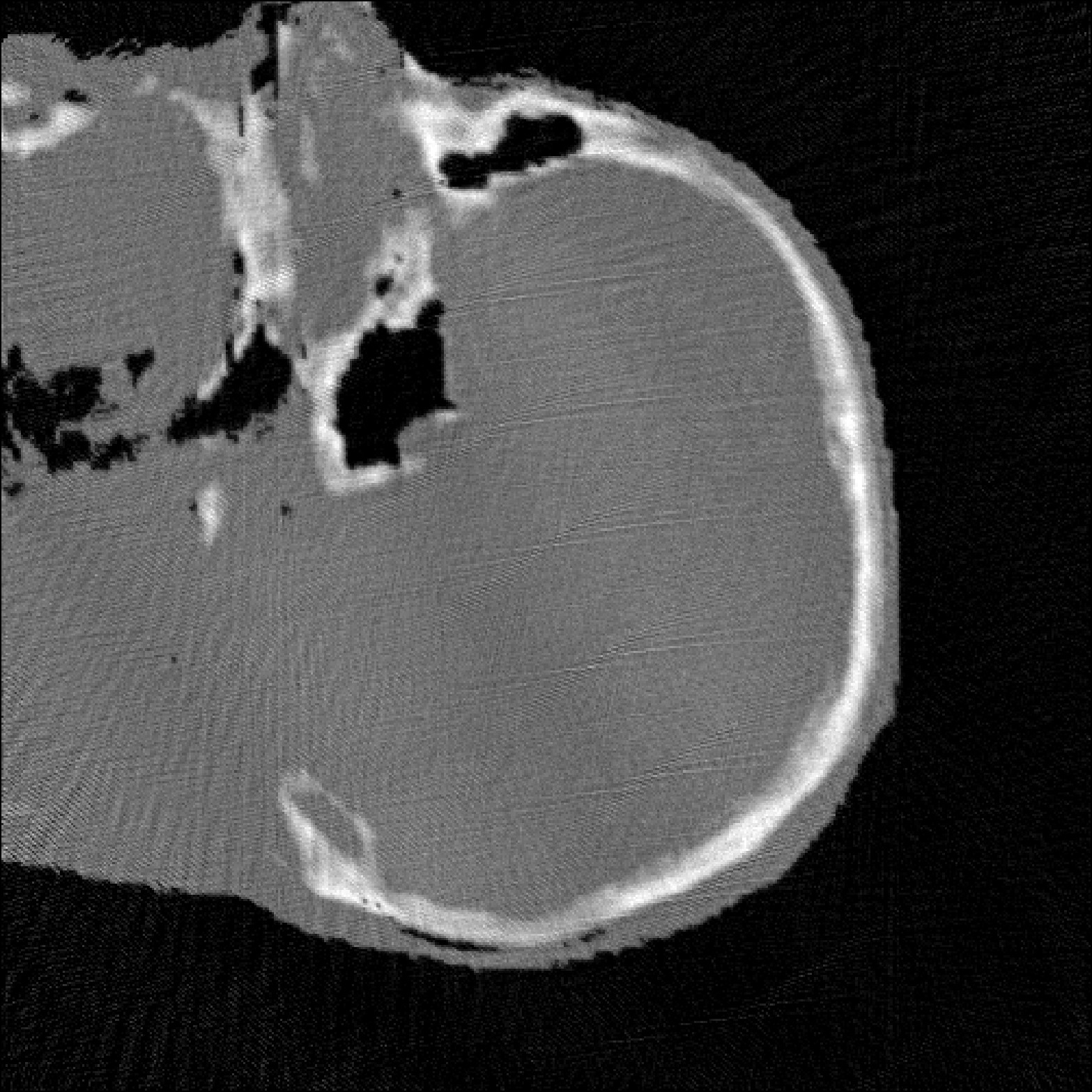}
    &
  \includegraphics[angle=90, height = 0.15\textwidth, width = .17\textwidth]{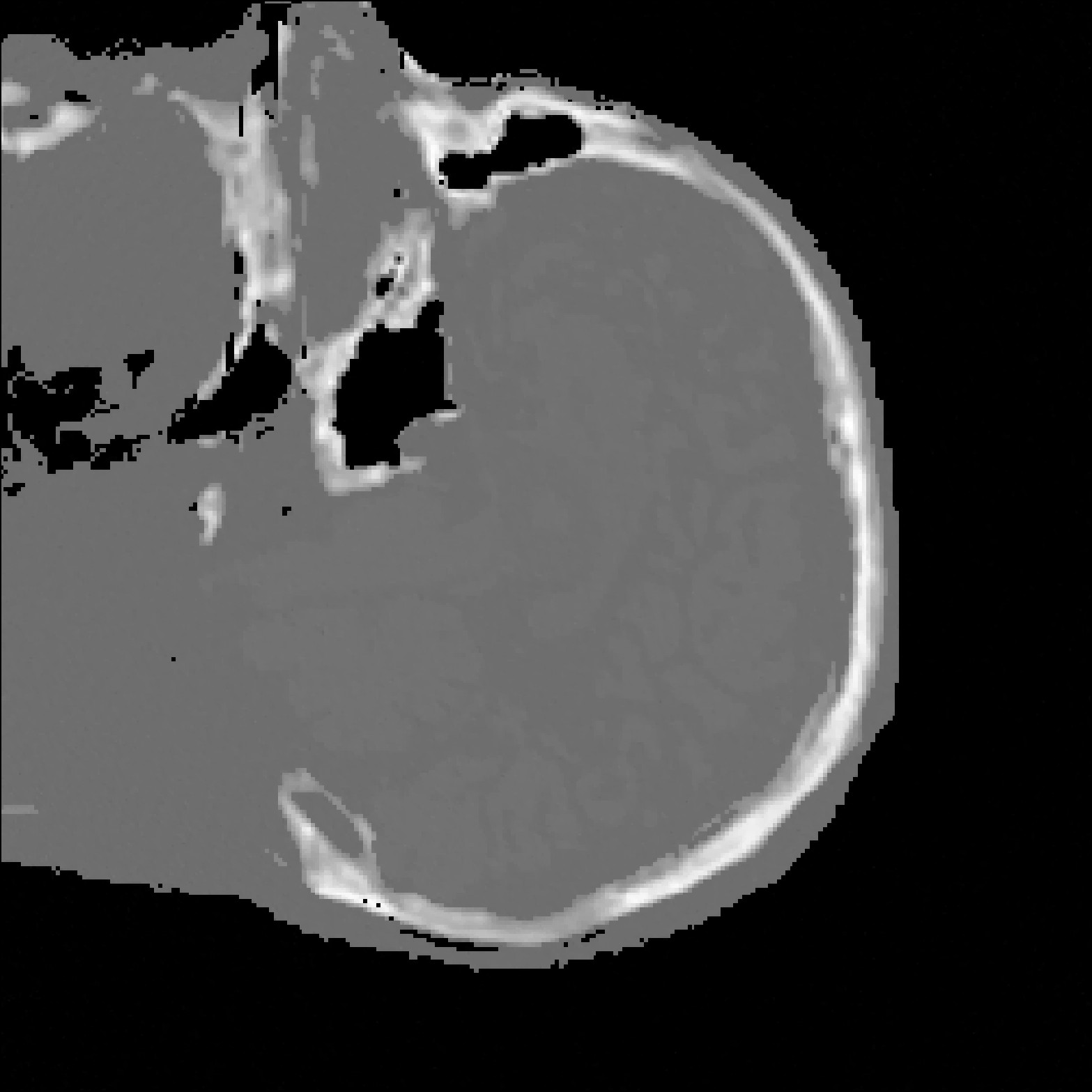}
    &
  \includegraphics[angle=90, height = 0.15\textwidth, width = .17\textwidth]{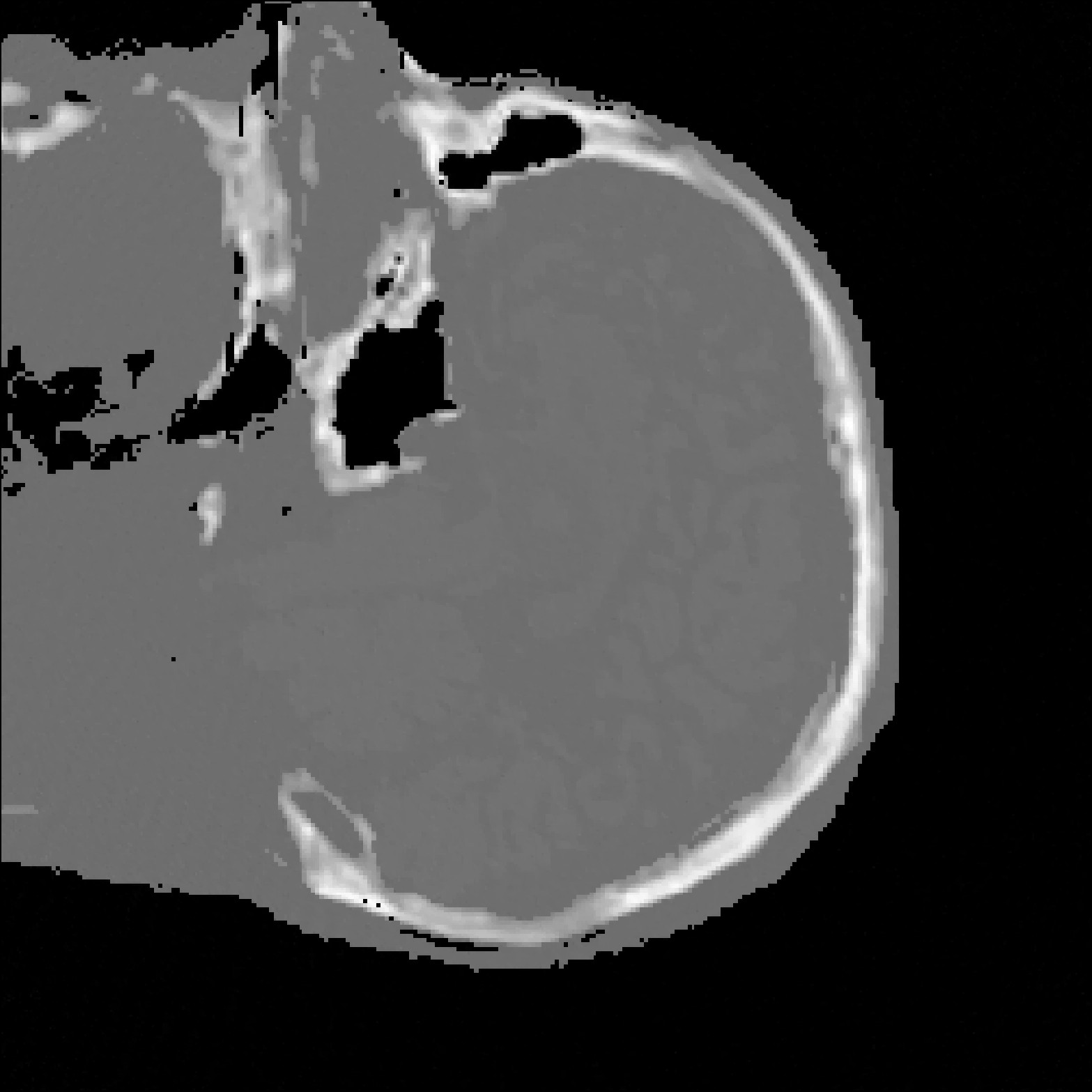}
\end{tabular}
\begin{tabular}{cccc}
  \includegraphics[angle=90, height = 0.15\textwidth, width = .17\textwidth]{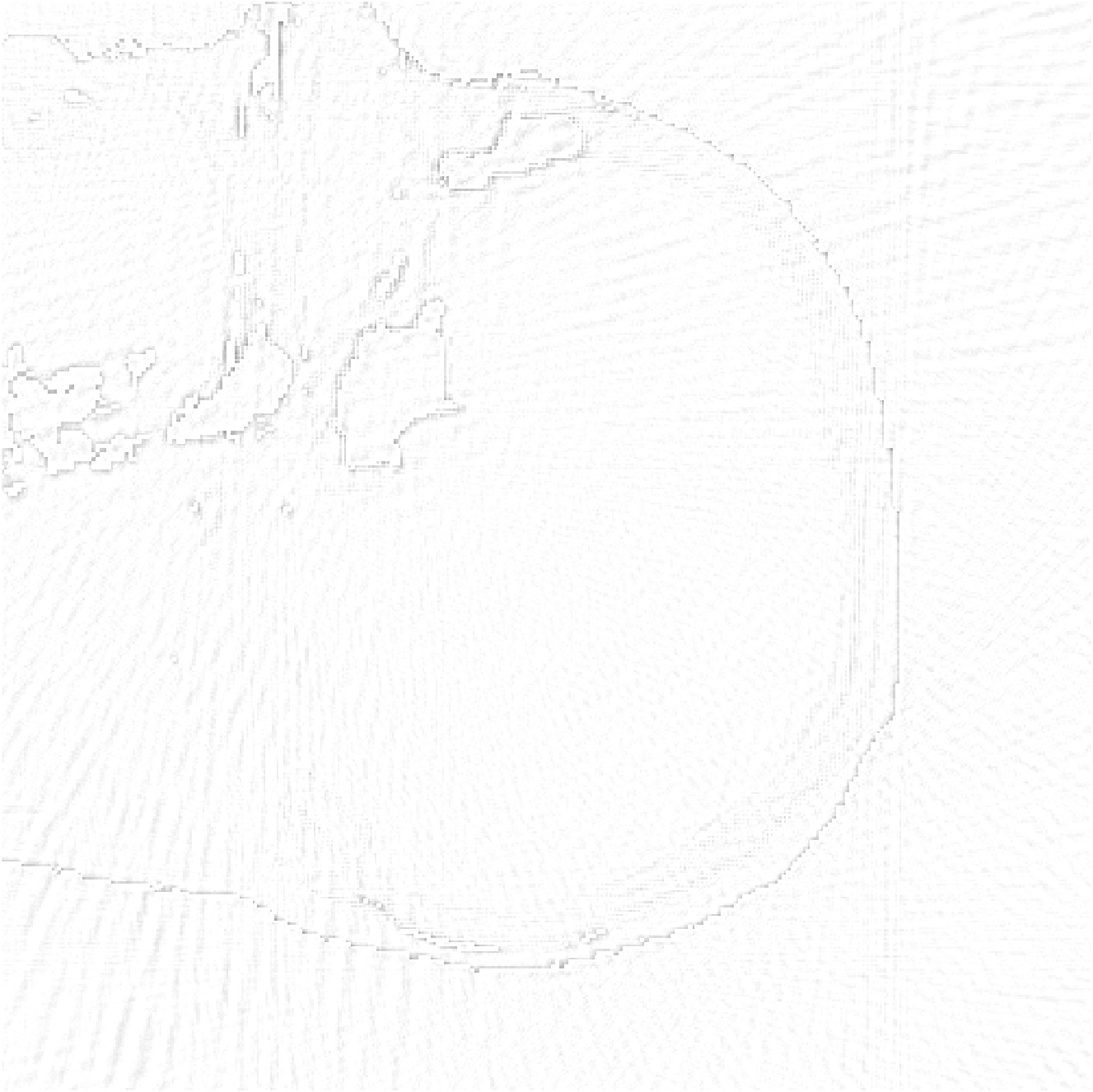}
    &
  \includegraphics[angle=90, height = 0.15\textwidth, width = .17\textwidth]{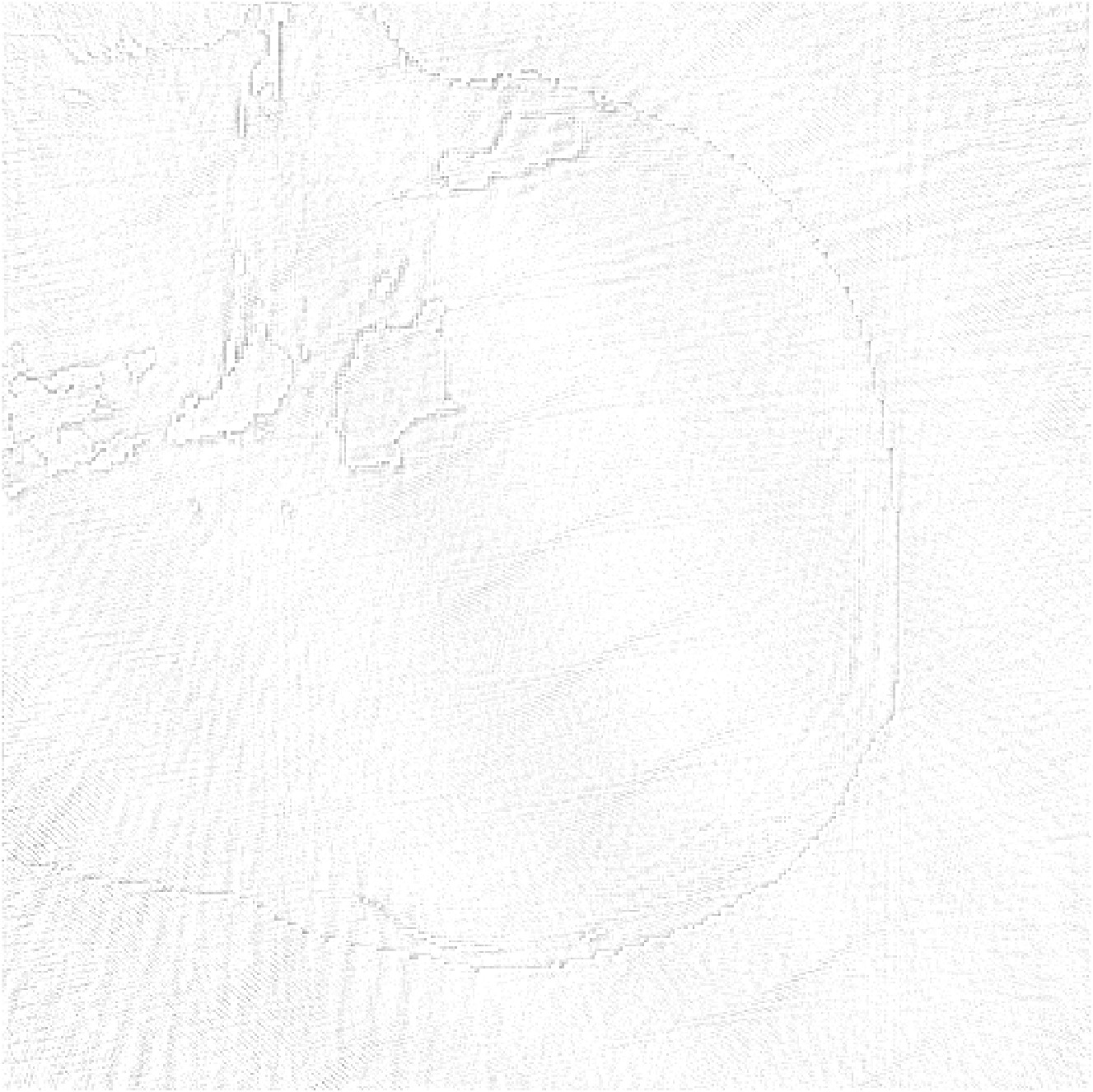}
    &
  \includegraphics[angle=90, height = 0.15\textwidth, width = .17\textwidth]{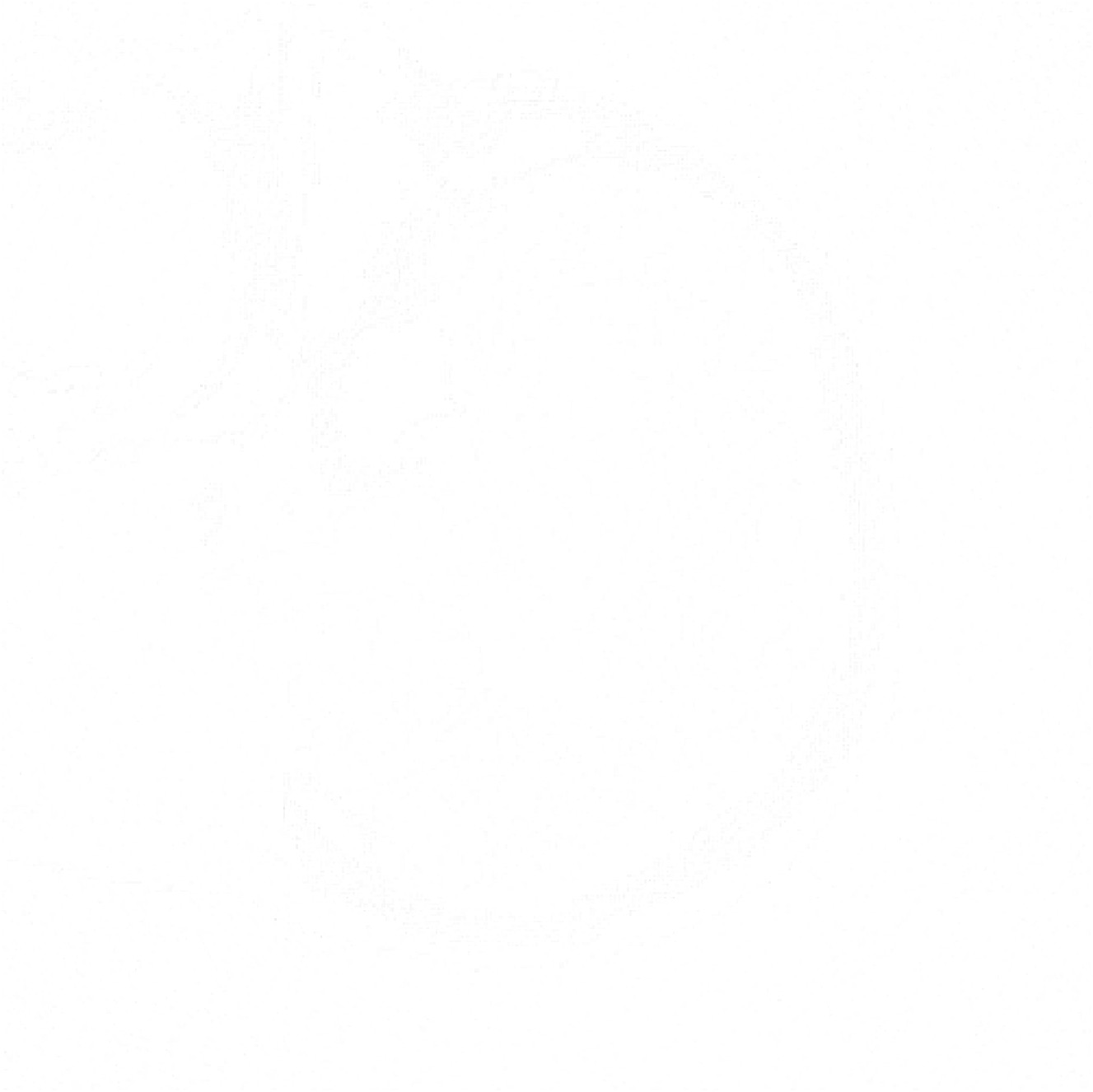}
     &
  \includegraphics[angle=90, height = 0.15\textwidth, width = .17\textwidth]{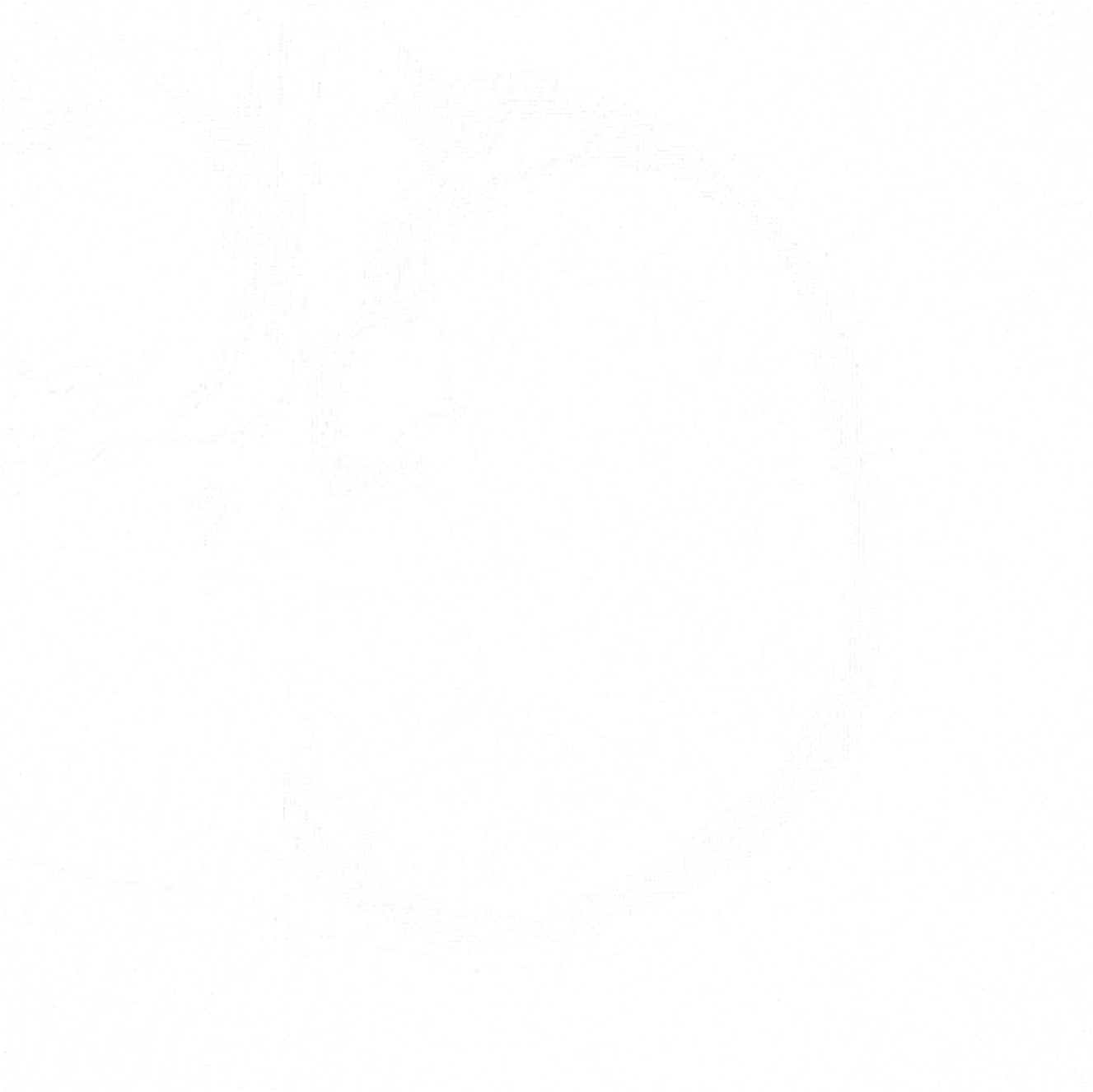}
\end{tabular}
\caption{Test 1: Streaming tomography example. Reconstructed images by method in rows 1 and 3. Corresponding error images in rows 2 and 4. The last two images in row 4 appear as white due to the small errors in the images (see the corresponding RREs for $\sigma = 0.1\%$ in Table~\ref{Table: tomo_stream_3prob_errors_noiselevles_differ}).
}
  \label{fig: Streaming3prob_reconstructions}
\end{figure}

For this test problem, we also vary the noise level from $0.1\%$ to $1\%$, and we summarize the results of the obtained RREs for the various methods and experiments in Table \ref{Table: tomo_stream_3prob_errors_noiselevles_differ}.
These results emphasize that the methods we propose are robust with respect to a range 
of noise levels.
\begin{table}[th!]
\centering
\small
\setlength{\tabcolsep}{3.5pt}
\begin{tabular}{@{} c cc ccc cc @{}}
\toprule
 & \multicolumn{2}{c}{Single subproblem}
 & \multicolumn{3}{c}{Full problem}
 & \multicolumn{2}{c}{Streaming} \\
\cmidrule(lr){2-3} \cmidrule(lr){4-6} \cmidrule(lr){7-8}
$\sigma$
 & HyBR\,1st & MM-GKS\,1st
 & HyBR\,all & MM-GKS\,all & LM-MM-GKS\,all
 & HyBR-rec & s-LM-MM-GKS \\
\midrule
0.1\% & {0.3765} & {0.2766}
      & {0.0925} & {0.0127} & {0.0142}
      & {0.1477} & {0.0644} \\[2pt]
0.5\% & {0.3927} & {0.3005}
      & {0.1029} & {0.0592} & {0.0786}
      & {0.2385} & {0.0942} \\[2pt]
1\%   & {0.4260} & {0.3307}
      & {0.1222} & {0.1284} & {0.1389}
      & {0.4095} & {0.1367}  \\
\bottomrule
\end{tabular}
\caption{Test 1: Computerized tomography. RRE for noise levels $0.1\%$, $0.5\%$, and $1\%$. The memory limit is $k_{\max} = 40$. All experiments are run for {$210$ expansion steps}.}
\label{Table: tomo_stream_3prob_errors_noiselevles_differ}
\end{table}
We observe that both MM-GKS and LM-MM-GKS
applied to the full problem \eqref{eq:rtomoproblemall} produce the best reconstruction results with RREs of {$1.27\%$ and $1.42\%$}, respectively (for
$\sigma = 0.1\%$){; see Table~\ref{Table: tomo_stream_3prob_errors_noiselevles_differ}}. 
{\em However, for (very) large problems the assumption is that addressing the full problem at once is not possible.}
MM-GKS needs to store all the solution subspace vectors, i.e., at iteration {$210$}, it has to store and operate with {$210$} vectors.
In contrast, for LM-MM-GKS, we assume 
the memory is sufficient 
to keep at most 40 search space vectors, and we achieve reconstructions of comparable quality with MM-GKS. We report these full-problem results as a benchmark to demonstrate that s-LM-MM-GKS, which processes the data in streaming fashion, achieves reconstruction quality close to the best achievable when all data is processes at once.
For scenarios where we cannot solve 
the full problem at once,
we observe that solving only the first 
subproblem produces relatively low quality reconstructions, as this corresponds to a limited angle tomography problem. For a more accurate reconstruction, we need information from other angles, which is provided
by subproblems \eqref{eq:rtomoi} and \eqref{eq:rtomont}. If memory is limited, s-LM-MM-GKS does so effectively by recycling the compressed solution space and approximate solution from one subproblem to the next. 
The effectiveness of this approach is demonstrated by the high quality reconstructions from the s-LM-MM-GKS method. 

Finally, we provide a comparison of the
convergence histories of the RRE between s-LM-MM-GKS and HyBR-recycle in Figure~\ref{Fig: sRMMGKSvsHyBRrecycle}. The blue line with square markers shows the RRE for s-LM-MM-GKS, and the red line with circle markers shows the RRE for HyBR-recycle. Both methods solve each subproblem for {$210$ expansion steps}, for a total of {$630$} iterations across the three subproblems.
{The green line with triangle markers shows the RRE for s-LM-MM-GKS with an adaptive stopping criterion, where each subproblem terminates when the relative difference between successive
approximations falls below $10^{-3}$. This variant achieves a comparable RRE using significantly fewer expansion steps.
We observe that s-LM-MM-GKS converges to a substantially lower RRE than HyBR-recycle, demonstrating the advantage of the edge-preserving $\ell_1$ regularization used in our method over the $\ell_2$ regularization employed by HyBR-recycle.}
\begin{figure}[h!]
\centering
\begin{tikzpicture}[font = \normalfont]

  \pgfplotstableread{Figures_rev/data/tomo_sLMMGKS.dat}\tablesLMMGKS
  \pgfplotstableread{Figures_rev/data/tomo_HyBRrec.dat}\tableHyBRrec
  \pgfplotstableread{Figures_rev/data/tomo_sLMMGKS_tol.dat}\tablesLMMGKStol

  \begin{axis}[%
    name                   = convplot,
    width                  = 0.990\textwidth,
    height                 = 8.5cm,
    ymode                  = log,
    xmin                   = 0,
    xmax                   = 620,
    xlabel                 = {Iteration},
    ylabel                 = {RRE},
    log ticks with fixed point,
    grid                   = major,
    legend pos             = north east,
    legend cell align      = {left},
    legend style           = {font=\scriptsize},
  ]

    \addplot[solid, thick, color=blue!70!black, mark=square*, mark repeat = 60]
      table[x = Iteration, y = RRE] {\tablesLMMGKS};
    \addlegendentry{s-LM-MM-GKS}

    \addplot[dashed, thick, color=red!70!black, mark=*, mark repeat = 60]
      table[x = Iteration, y = RRE] {\tableHyBRrec};
    \addlegendentry{HyBR-recycle}

    \addplot[dashdotted, thick, color=green!60!black, mark=triangle*, mark repeat = 24]
      table[x = Iteration, y = RRE] {\tablesLMMGKStol};
    \addlegendentry{s-LM-MM-GKS ($\mathrm{tol} = 10^{-3}$)
}
\end{axis}
\end{tikzpicture}
\caption{{Test 1: Streaming tomography example ($\sigma = 0.1\%$). RRE convergence history for three numerical experiments: s-LM-MM-GKS and HyBR-recycle for a fixed number of expansion steps per subproblem, and s-LM-MM-GKS with a relative difference stopping criterion ($\mathrm{tol} = 10^{-3}$).}}
\label{Fig: sRMMGKSvsHyBRrecycle}
\end{figure}
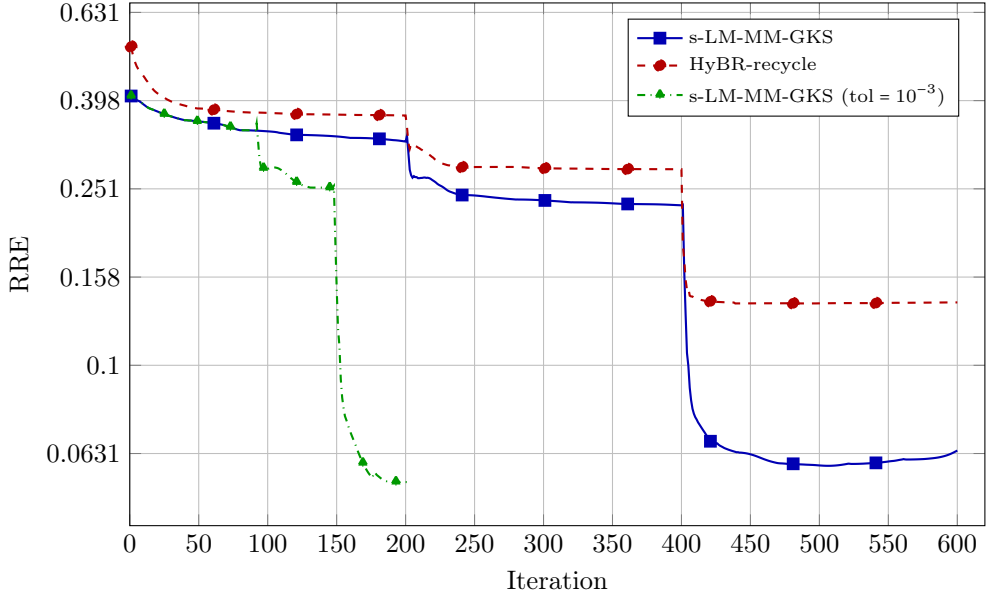

\subsubsection{Test 2}
{We consider a larger parallel tomography problem, generated using IRTools \cite{gazzola2018ir}, with a $1000 \times 1000$ Shepp-Logan phantom and $180$
projection angles at $1^{\circ}$ intervals.
This produces a forward operator $\bA \in \R^{254520 \times 1000^2}$ and a data vector $\bd \in \R^{254520}$. We add $0.1\%$ white Gaussian noise to $\bd$.
We assume insufficient memory to process the data 
all at once. 
Therefore, 
the $180$ projection angles are randomly permuted and divided into $6$ groups of $30$ angles each, which are processed successively.
This yields $6$ subproblems with $\bA_1, \ldots, \bA_6 \in \R^{42420 \times 1000^2}$ and $\bd_1, \ldots, \bd_6 \in \R^{42420}$.
The true image is shown in Figure~\ref{fig: test2_true}(a). The sinograms for the $6$ subproblems are shown in Figure~\ref{fig: test2_true}(b)-(g).}

\begin{figure}[htbp]
\centering
\begin{minipage}[c]{0.22\textwidth}
\centering
$\bx_{\rm true}$ \\[2pt]
\includegraphics[angle=90, width=\textwidth]{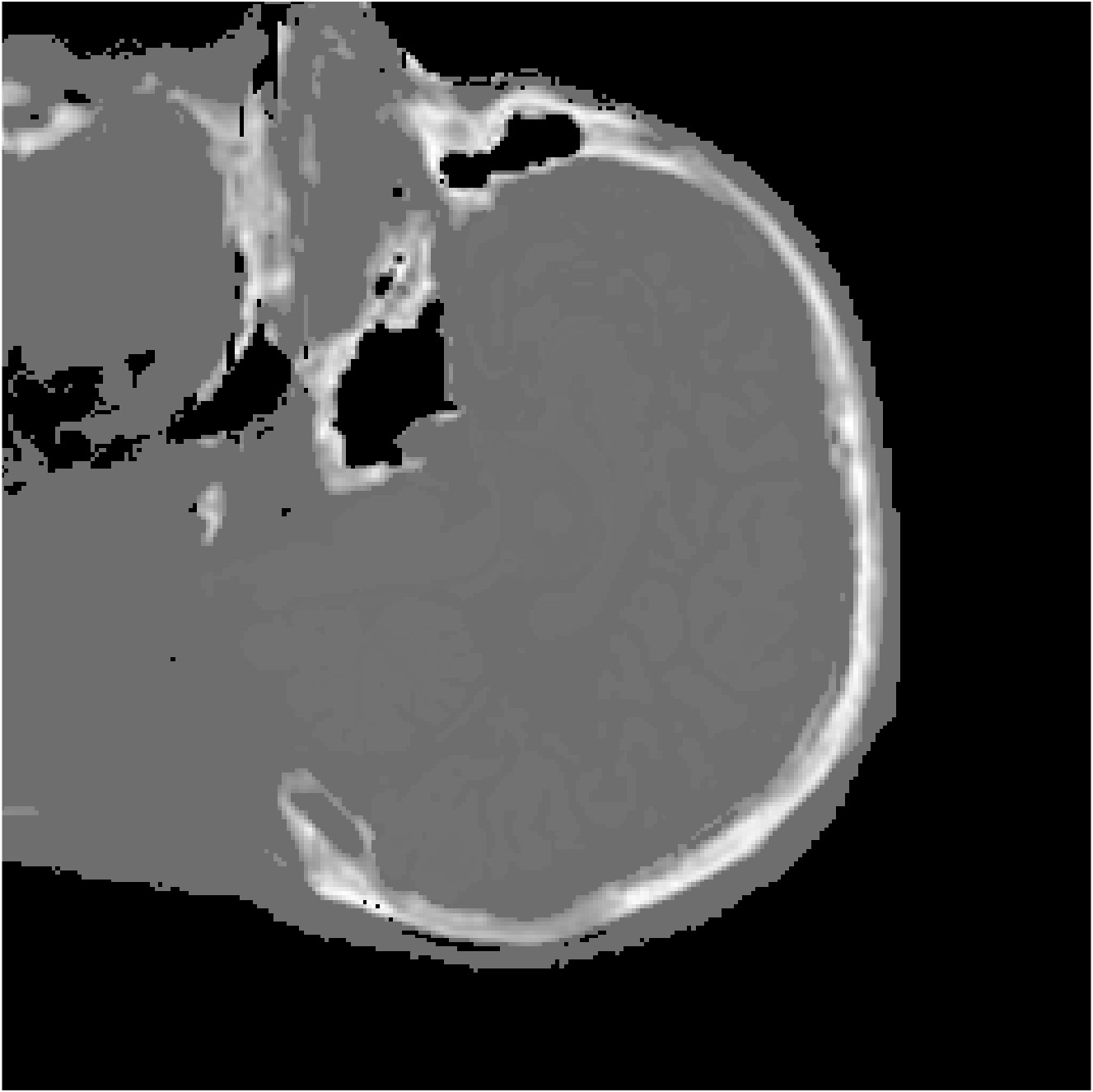} \\
(a)
\end{minipage}%
\hfill
\begin{minipage}[c]{0.73\textwidth}
\centering
\begin{tabular}{ccc}
$\bd_1$ & $\bd_2$ & $\bd_3$ \\
\includegraphics[height = 0.08\textwidth, width = .22\textwidth]{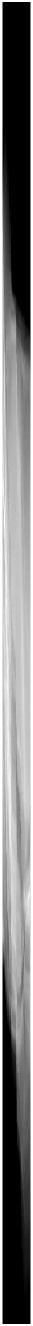} &
\includegraphics[height = 0.08\textwidth, width = .22\textwidth]{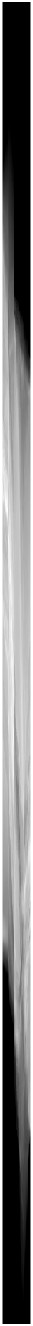} &
\includegraphics[height = 0.08\textwidth, width = .22\textwidth]{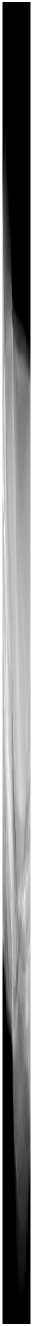} \\
(b) & (c) & (d) \\[4pt]
$\bd_4$ & $\bd_5$ & $\bd_6$ \\
\includegraphics[height = 0.08\textwidth, width = .22\textwidth]{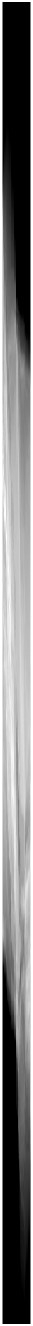} &
\includegraphics[height = 0.08\textwidth, width = .22\textwidth]{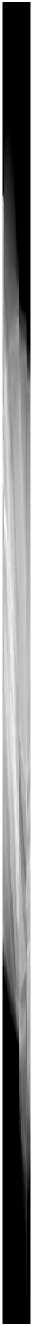} &
\includegraphics[height = 0.08\textwidth, width = .22\textwidth]{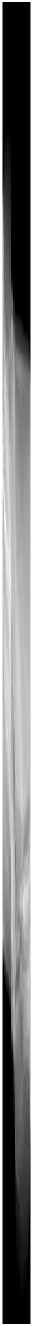} \\
(e) & (f) & (g)
\end{tabular}
\end{minipage}
\caption{{Test 2: Streaming tomography example. The $1000 \times 1000$ pixels true image is provided in (a). The sinograms $\bd_1, \ldots, \bd_6$ for the $6$ randomly formed subproblems are shown in (b)--(g).}}
\label{fig: test2_true}
\end{figure}
This example focuses 
on a severely limited memory case, where we assume that we have 
only enough memory to process 30 basis vectors. To mimic this, we run MM-GKS type methods and we report their result for only $30$ expansion iterations.
LM-MM-GKS and s-LM-MM-GKS keep the memory for the solution subspace bounded as well, but the enlarging and compressing approach allows us to run for a larger number of iterations. We set $k_{\max} = 30$ and $k_{\min} = 10$.
We compare the performance of LM-MM-GKS on streaming data by considering the following scenarios.
\begin{enumerate}
    \item Run MM-GKS on the first subproblem (MM-GKS 1st). Also run MM-GKS on all the data, i.e., solve \eqref{eq:rtomoproblemall} with $n_t = 6$ (MM-GKS all) for $30$ iterations.
    \item Run LM-MM-GKS on the first subproblem (LM-MM-GKS 1st) and LM-MM-GKS on the full data problem (LM-MM-GKS all) for 200 expansion steps.
    \item Run s-LM-MM-GKS method first for a fixed number of iterations (here we run a fixed total of $360$ expansion steps). Then, we consider the case of stopping the iterations based on a stopping criterion. In particular, we stop if either the maximum number of iterations ($100$ for each subproblem) or a desired tolerance is achieved (we set $tol_1 = 10^{-3}$).
\end{enumerate}

\begin{figure}[ht!]
\centering
  \begin{tabular}{ccccc}
   MM-GKS 1st &  LM-MM-GKS 1st &  MM-GKS all & LM-MM-GKS all & s-LM-MM-GKS \\
  \includegraphics[angle=90,height = 0.14\textwidth, width = .16\textwidth]{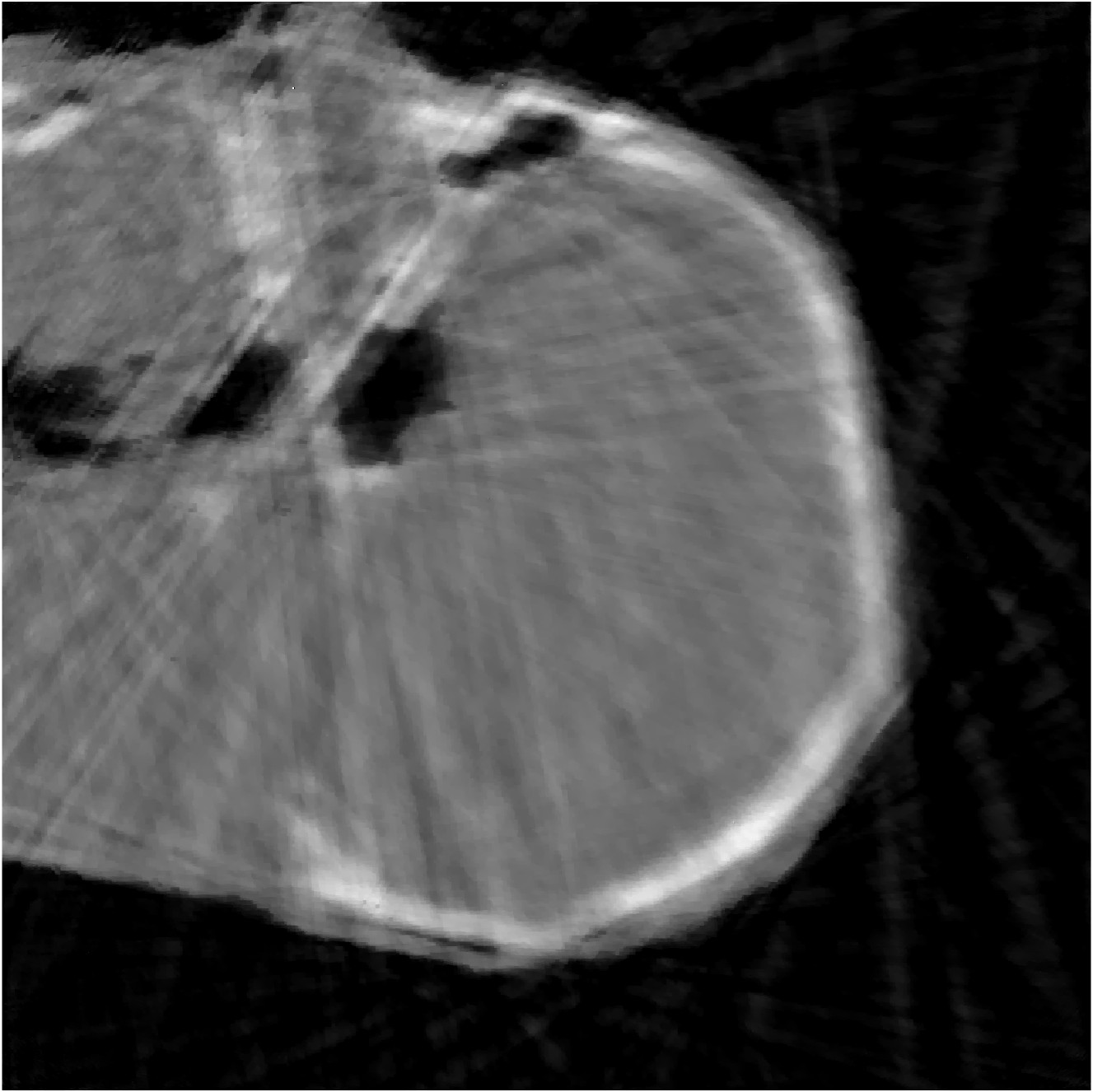} &
  \includegraphics[angle=90,height = 0.14\textwidth, width = .16\textwidth]{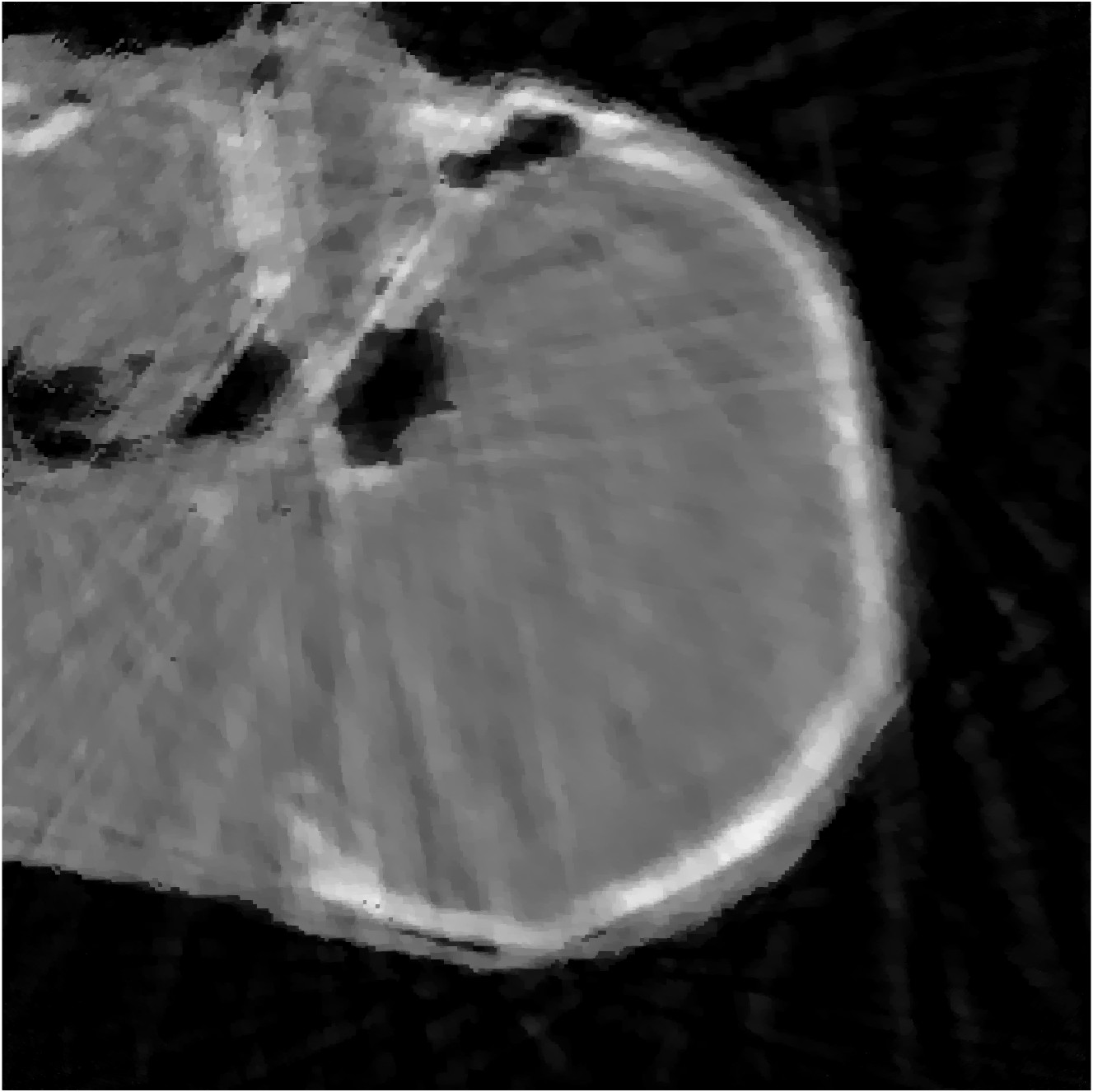} &
  \includegraphics[angle=90,height = 0.14\textwidth, width = .16\textwidth]{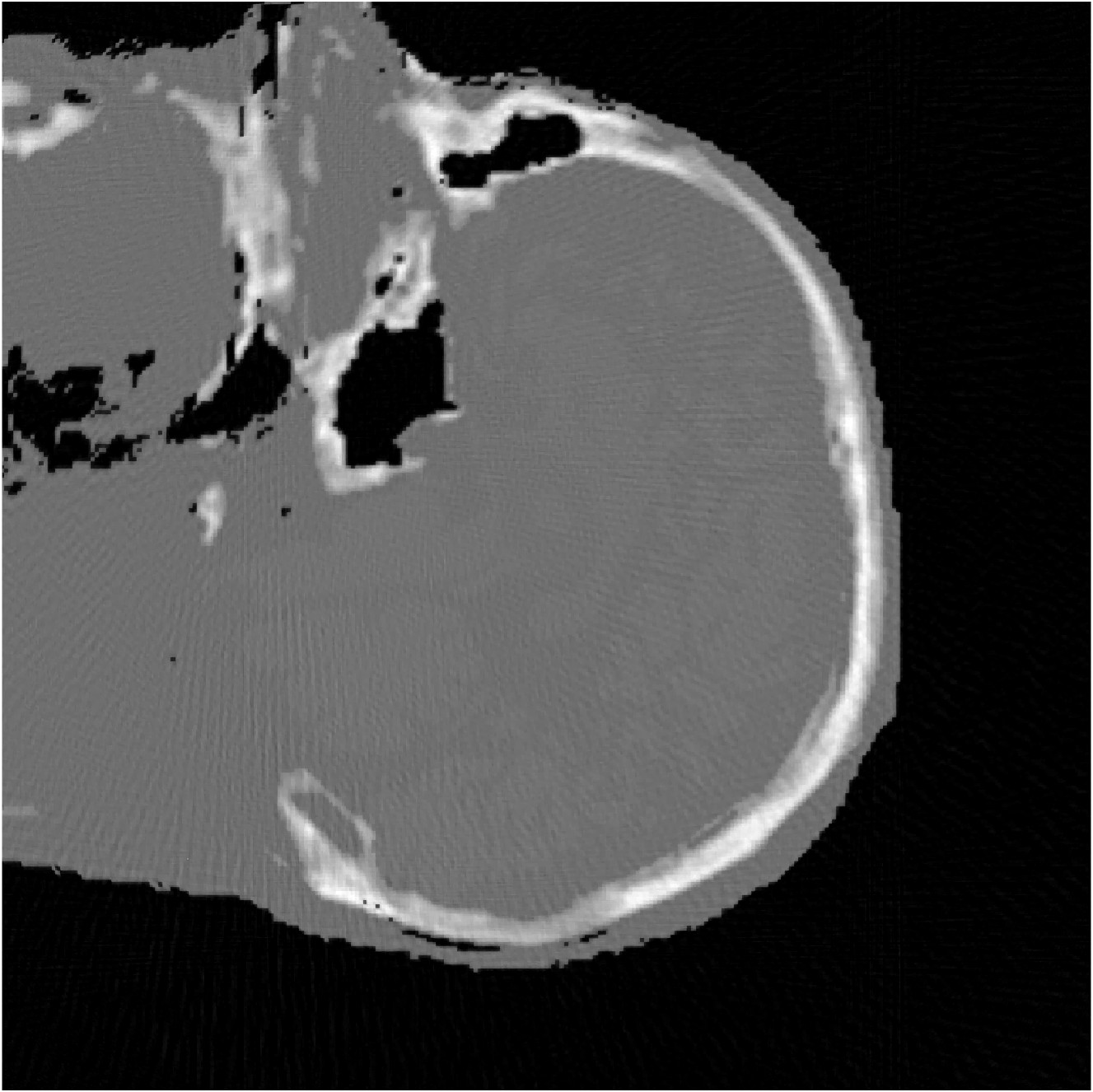} &
  \includegraphics[angle=90,height = 0.14\textwidth, width = .16\textwidth]{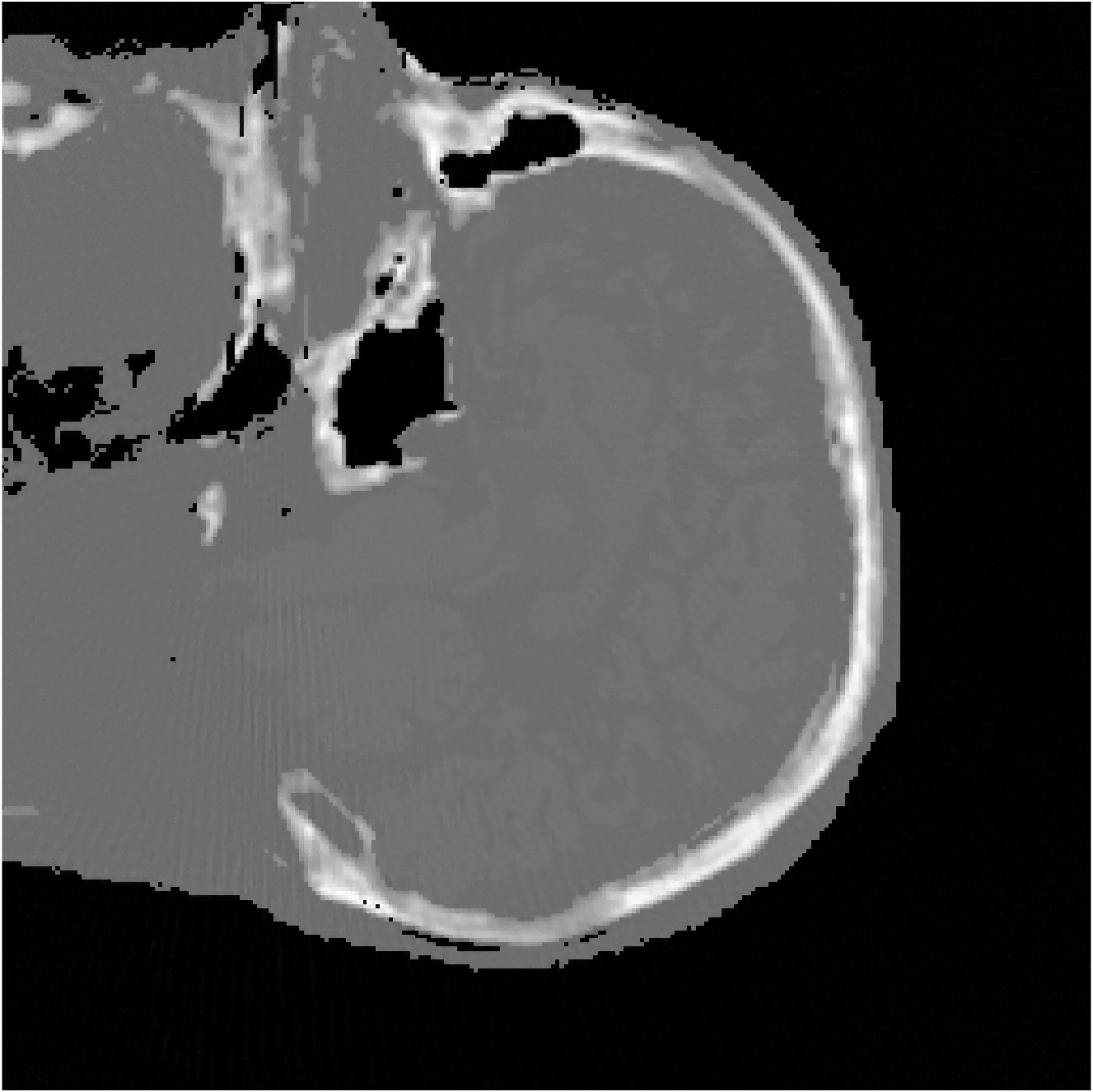} &
  \includegraphics[angle=90,height = 0.14\textwidth, width = .16\textwidth]{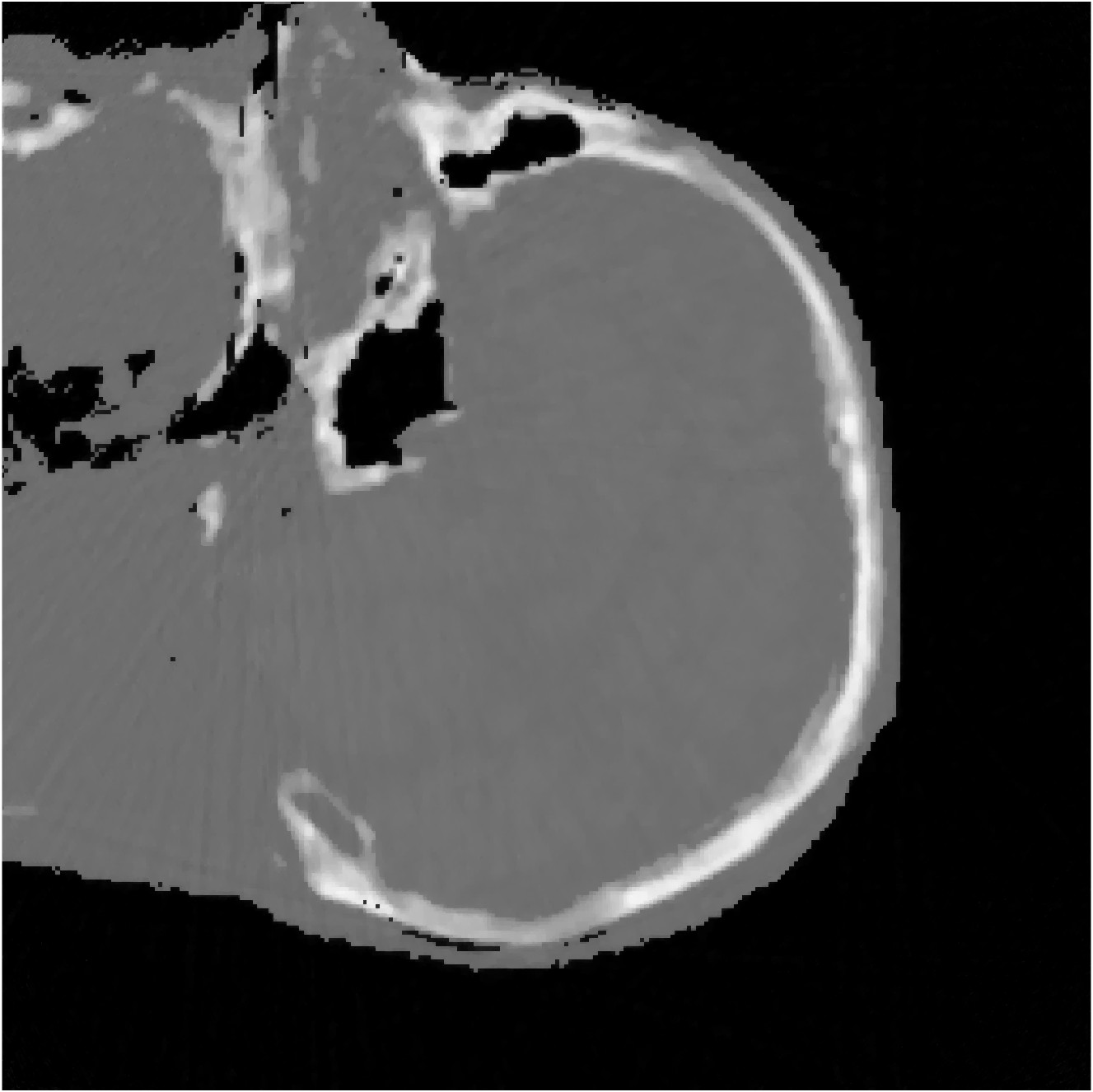}
\end{tabular}
\begin{tabular}{ccccc}
  \includegraphics[angle=90,height = 0.14\textwidth, width = .16\textwidth]{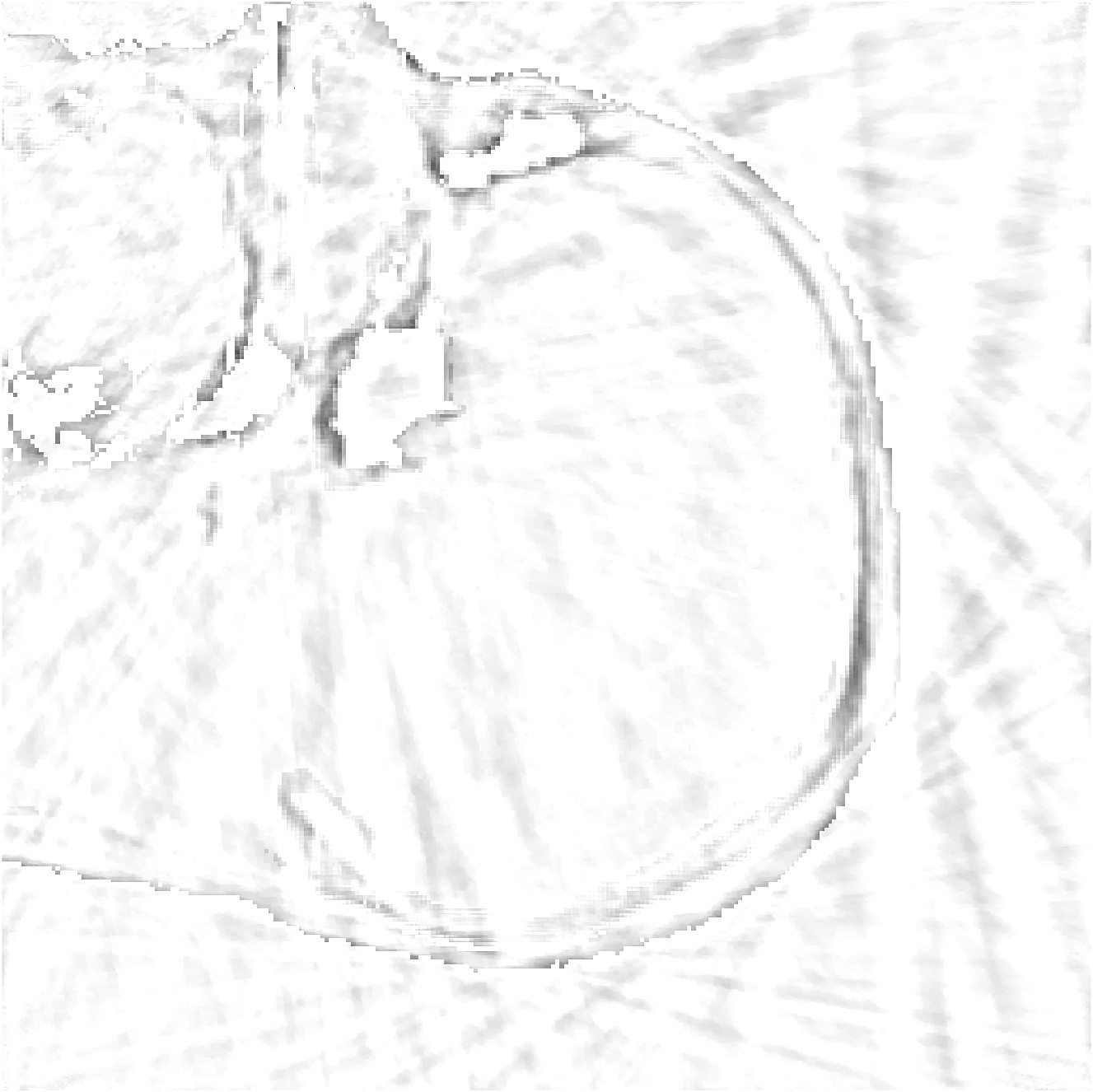} &
  \includegraphics[angle=90,height = 0.14\textwidth, width = .16\textwidth]{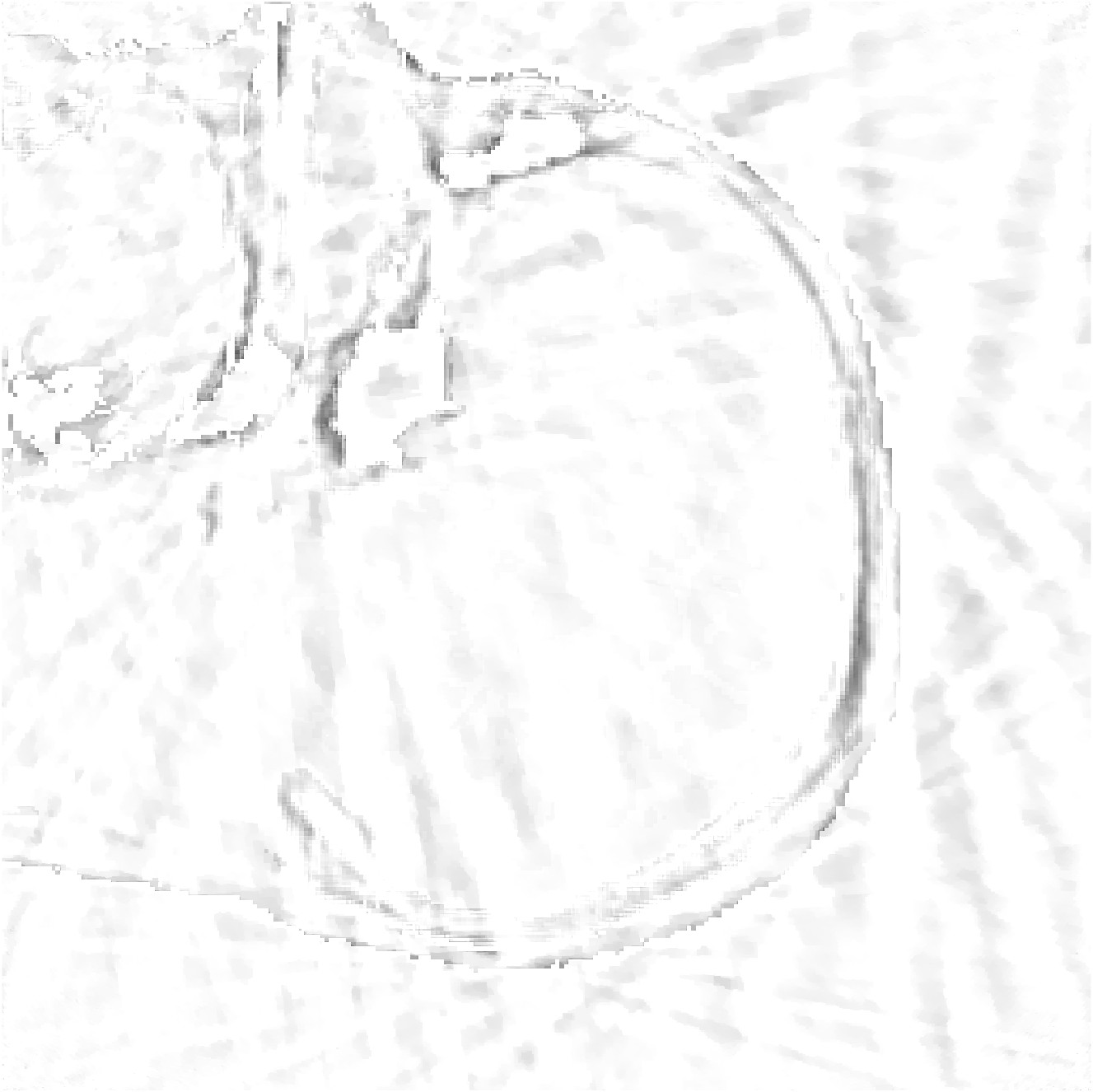} &
  \includegraphics[angle=90,height = 0.14\textwidth, width = .16\textwidth]{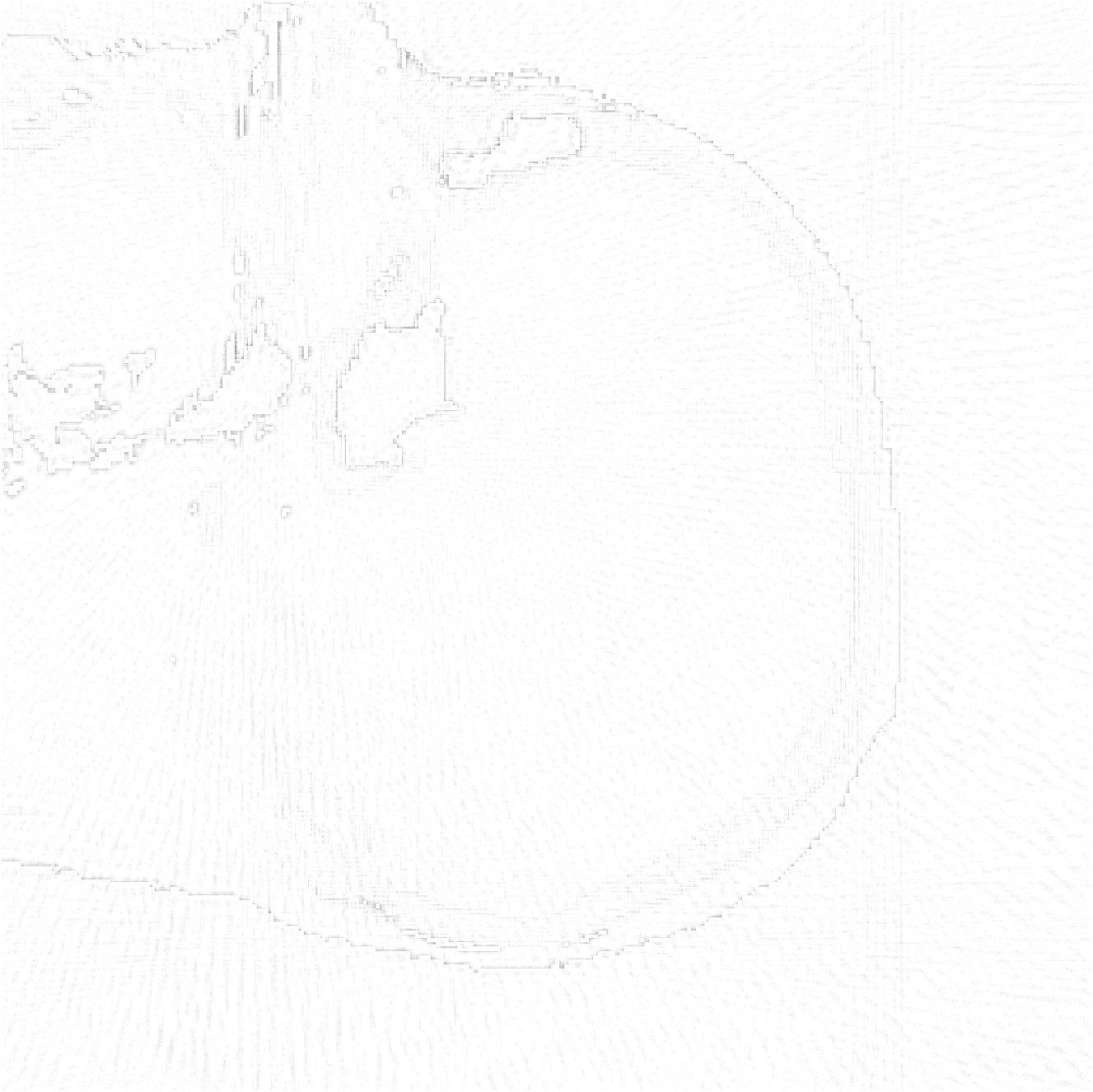} &
  \includegraphics[angle=90,height = 0.14\textwidth, width = .16\textwidth]{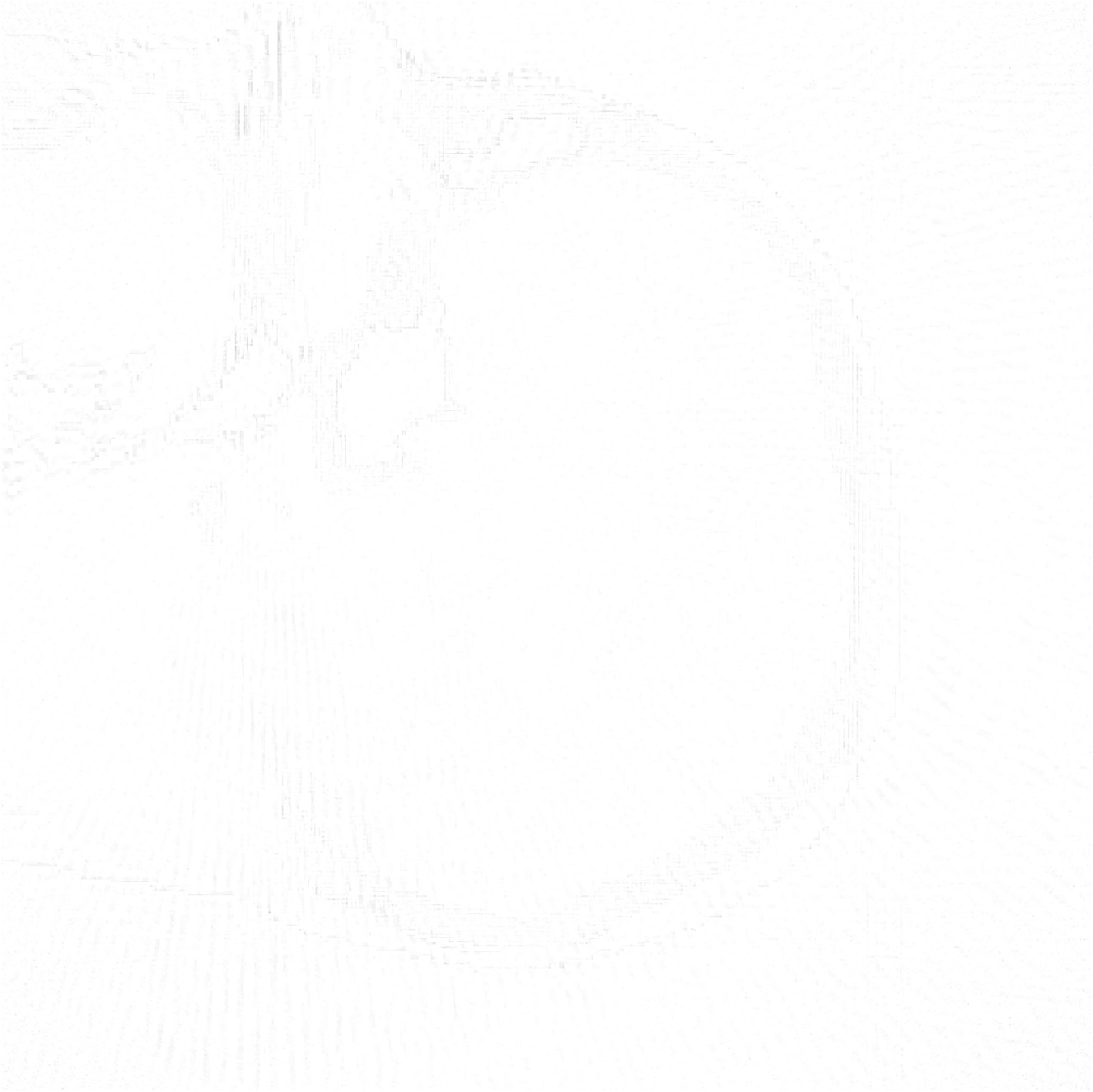} &
  \includegraphics[angle=90,height = 0.14\textwidth, width = .16\textwidth]{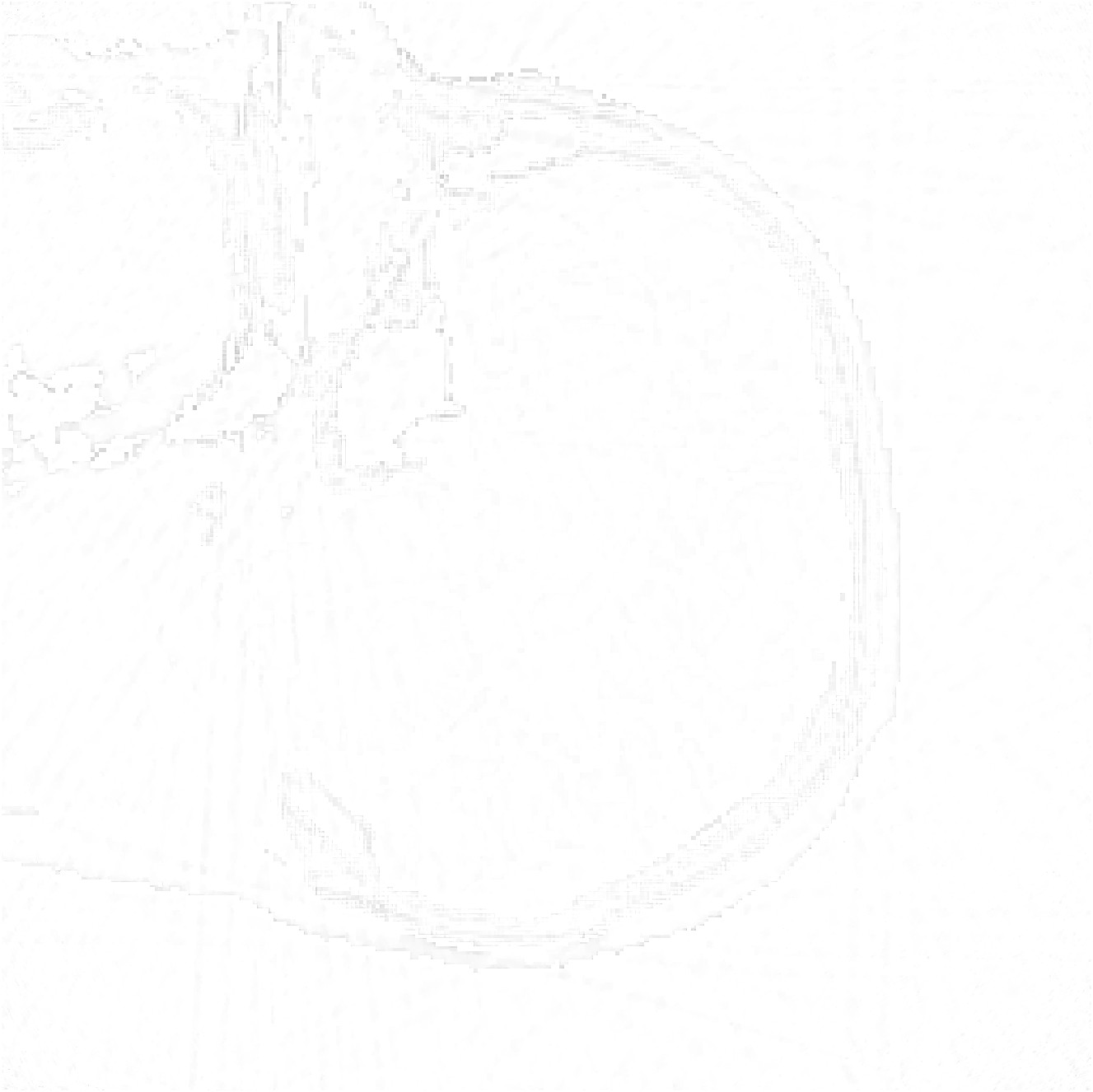}
\end{tabular}
\caption{Test 2: Streaming tomography example ($1000 \times 1000$, $\sigma = 0.1\%$). Reconstructions for MM-GKS and LM-MM-GKS on the limited data from the first dataset and on all the data, from left to right, respectively. The last image shows the reconstruction using s-LM-MM-GKS. The second row displays the corresponding error images.}
\label{fig: test2_reconstructions}
\end{figure} 

\begin{table}[th!]
\centering
\small
\setlength{\tabcolsep}{3.5pt}
\begin{tabular}{@{} c cc cc cc @{}}
\toprule
 & \multicolumn{2}{c}{Single subproblem}
 & \multicolumn{2}{c}{Full problem}
 & \multicolumn{2}{c}{Streaming} \\
\cmidrule(lr){2-3} \cmidrule(lr){4-5} \cmidrule(lr){6-7}
$\sigma$
 & MM-GKS\,1st & LM-MM-GKS\,1st
 & MM-GKS\,all & LM-MM-GKS\,all
 & s-LM-MM-GKS & s-LM-MM-GKS (tol) \\
\midrule
0.1\% & 0.1906 & 0.1690
      & 0.0633 & 0.0275
      & 0.0424 & 0.0635 \\[2pt]
0.5\% & 0.1967 & 0.1898
      & 0.0963 & 0.0993
      & 0.0748 & 0.0856 \\
\bottomrule
\end{tabular}
\caption{Test 2: Streaming CT ($1000\times 1000$, $6$ random subproblems). RRE for noise levels $0.1\%$ and $0.5\%$. Memory limit $k_{\max} = 30$, $k_{\min} = 10$.}
\label{Table: test2_RRE}
\end{table}

{The RREs for the methods we consider are shown in Table~\ref{Table: test2_RRE}. For low noise ($\sigma = 0.1\%$), LM-MM-GKS applied to all the data achieves the best RRE of $0.0275$ with only $k_{\max} = 30$ stored basis vectors. At higher noise ($\sigma = 0.5\%$), s-LM-MM-GKS achieves the best RRE of $0.0748$, outperforming both MM-GKS and LM-MM-GKS on the full problem. The reconstructed images with all the methods along with the error images in the inverted colormap are displayed in Figure~\ref{fig: test2_reconstructions}.}
\subsection{Dynamic photoacoustic tomography (dPAT)}
\label{subsec:DynPAT}
Photoacoustic tomography
(PAT) is an emerging hybrid imaging modality that shows great potential for pre-clinical research. 
PAT combines the rich contrast of optical imaging with the high resolution of ultrasound imaging. It promises high resolution images with lower cost and fewer side effects than other imaging modalities. 
We consider a discrete dPAT problem, where we let $\bx_{i}\in \R^{n}$, with $n = n_x\cdot n_y$, be the (vectorized) true image at time step $i$, for $i = 1, 2, \dots, n_t$. This yields a dynamic model. PAT has been considered in a 
dynamic framework before; see, for instance,  \cite{chung2017motion,chung2018efficient,lucka2018enhancing}. 
At timestep $i$, measurements $\bd_i$ are taken at a set of transducers
surrounding the object of interest (represented here by the images $\bx_i$)
at $n_a$ angles and for 
$n_r$ radii at each transducer, leading to 
a discrete forward operator 
$\bA_{i} \in \R^{n_a\cdot n_r \times n}$
and the system,  
\begin{equation}\label{eq: lin_eq_spherical}
\bd_{i} = \bA_{i} \bx_{i} + \be_i,
\end{equation}
with $\be_i$ being Gaussian noise. 
The goal is to estimate approximations 
to the images $\bx_{i}$ given the observations $\bd_i$.

\begin{figure}[ht!]
\centering
\begin{tabular}{cccccc}
     $t =1$ & \qquad  $t =10$ & \qquad  $t =20$ &   \qquad  $t =30$ & \qquad $t =40$ &  \qquad $t =50$  \\
\end{tabular}
\begin{tabular}{cccccc}
\includegraphics[height = 0.13\textwidth, width = .13\textwidth]{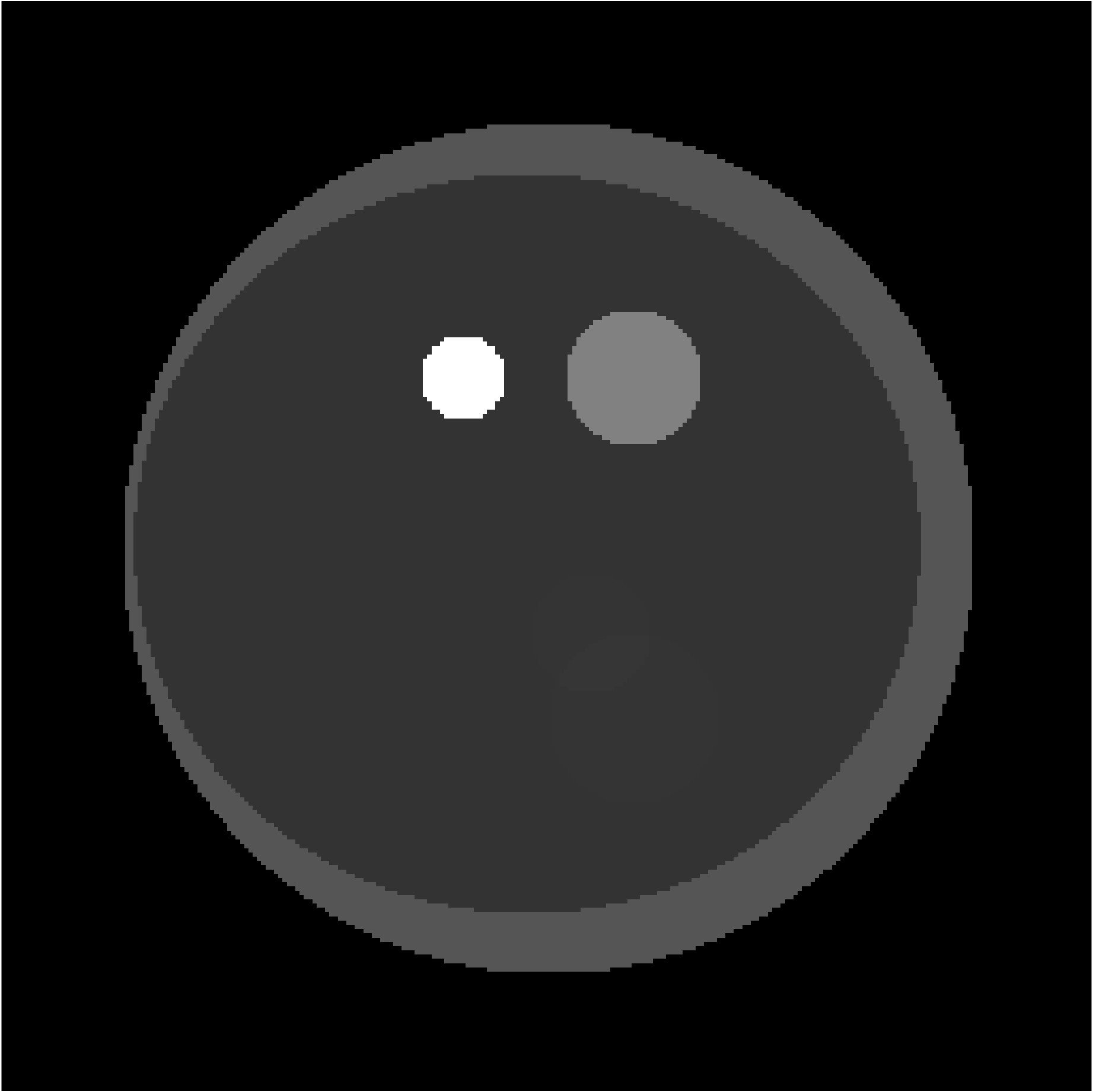} &
\includegraphics[height = 0.13\textwidth, width = .13\textwidth]{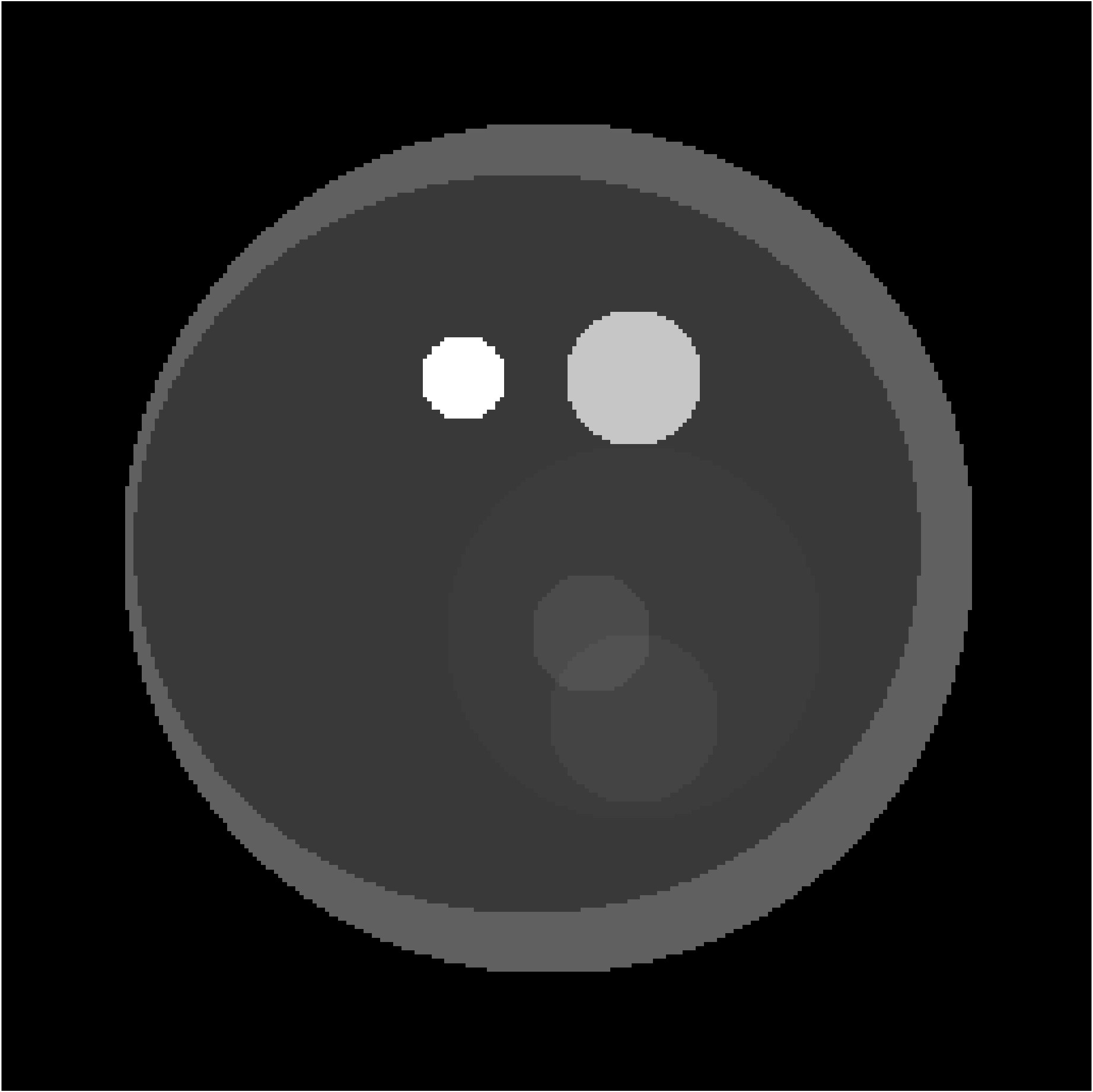} &
\includegraphics[height = 0.13\textwidth, width = .13\textwidth]{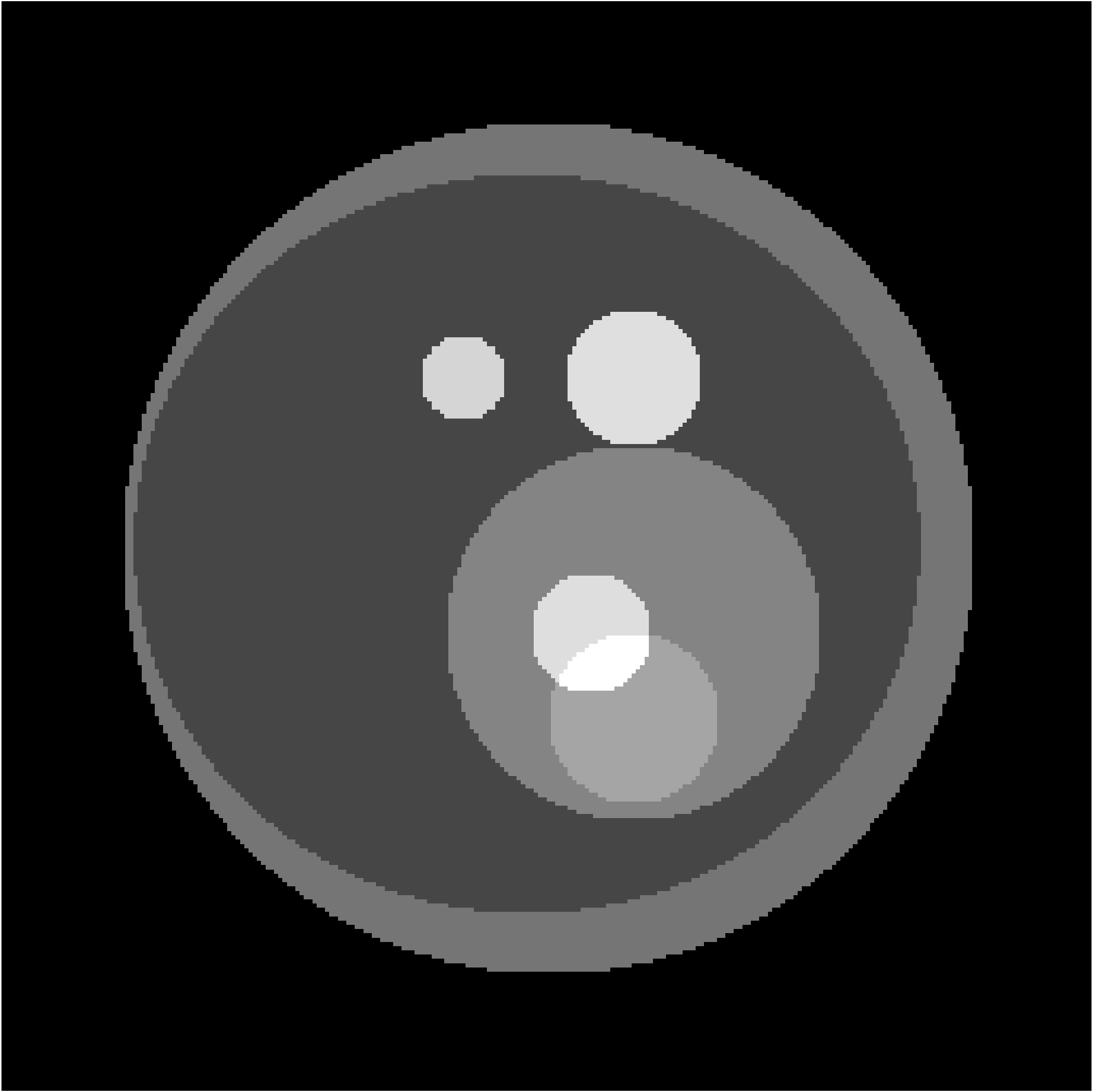}& \includegraphics[height = 0.13\textwidth, width = .13\textwidth]{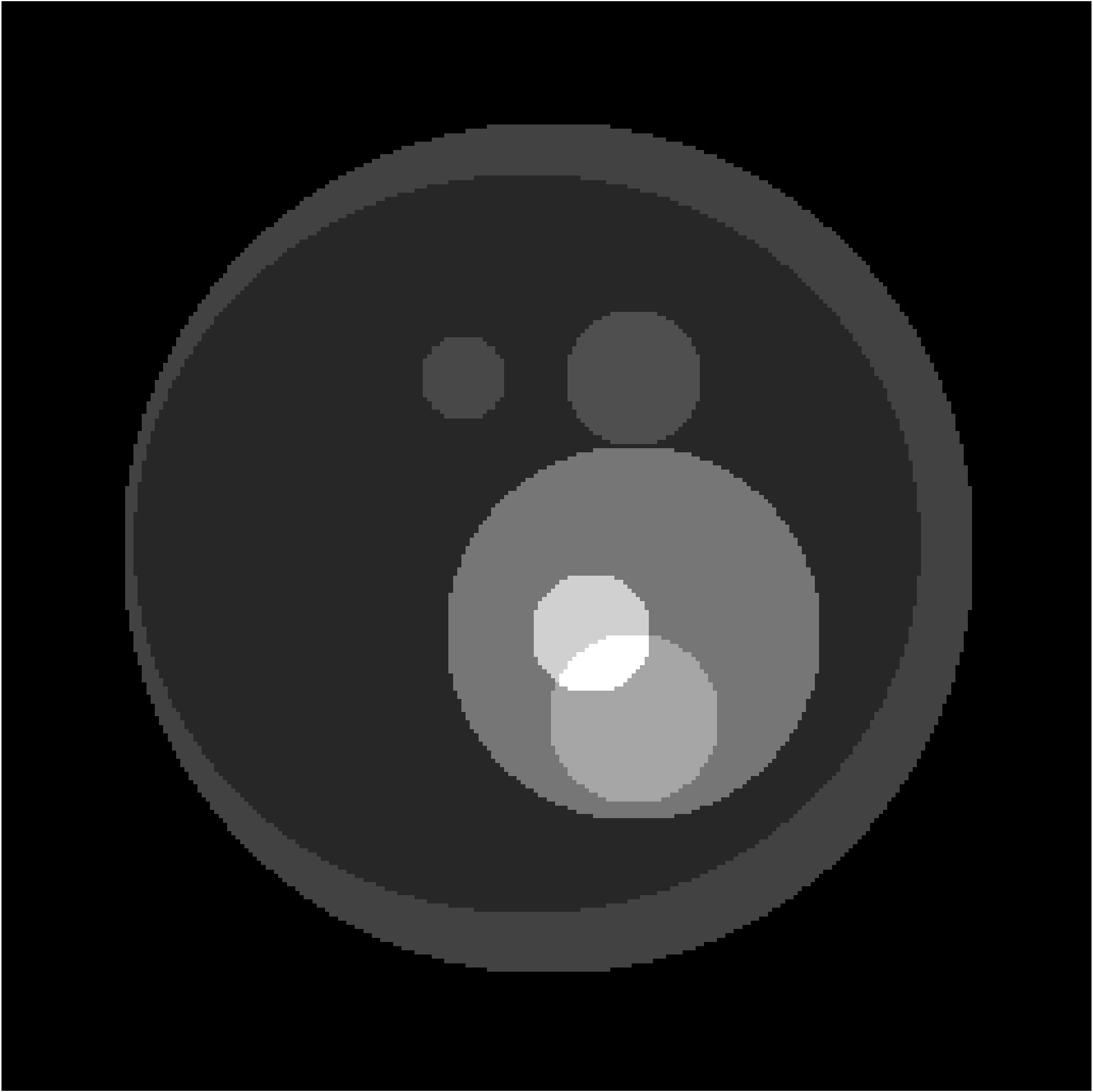} &
\includegraphics[height = 0.13\textwidth, width = .13\textwidth]{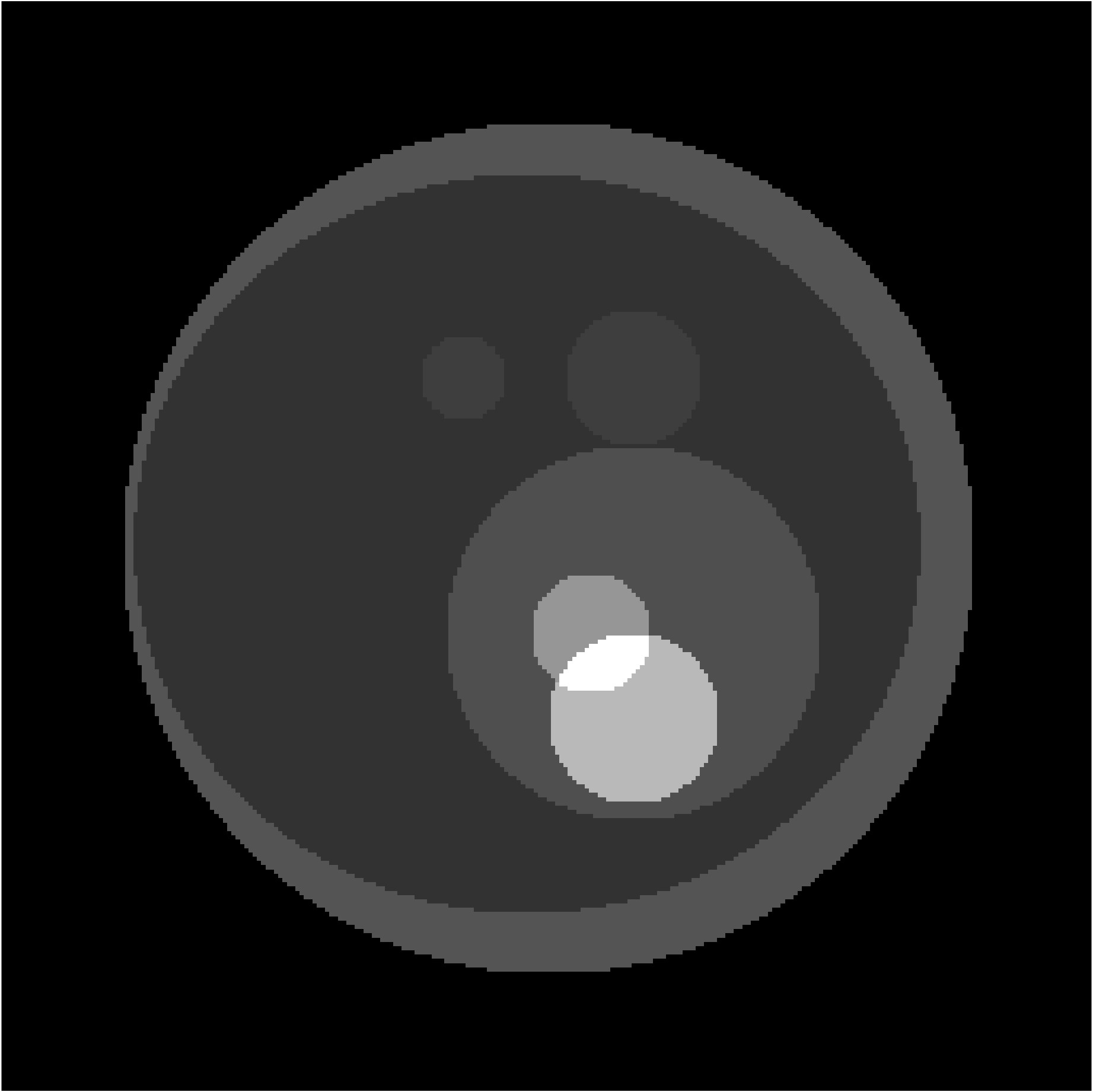}& 
\includegraphics[height = 0.13\textwidth, width = .13\textwidth]{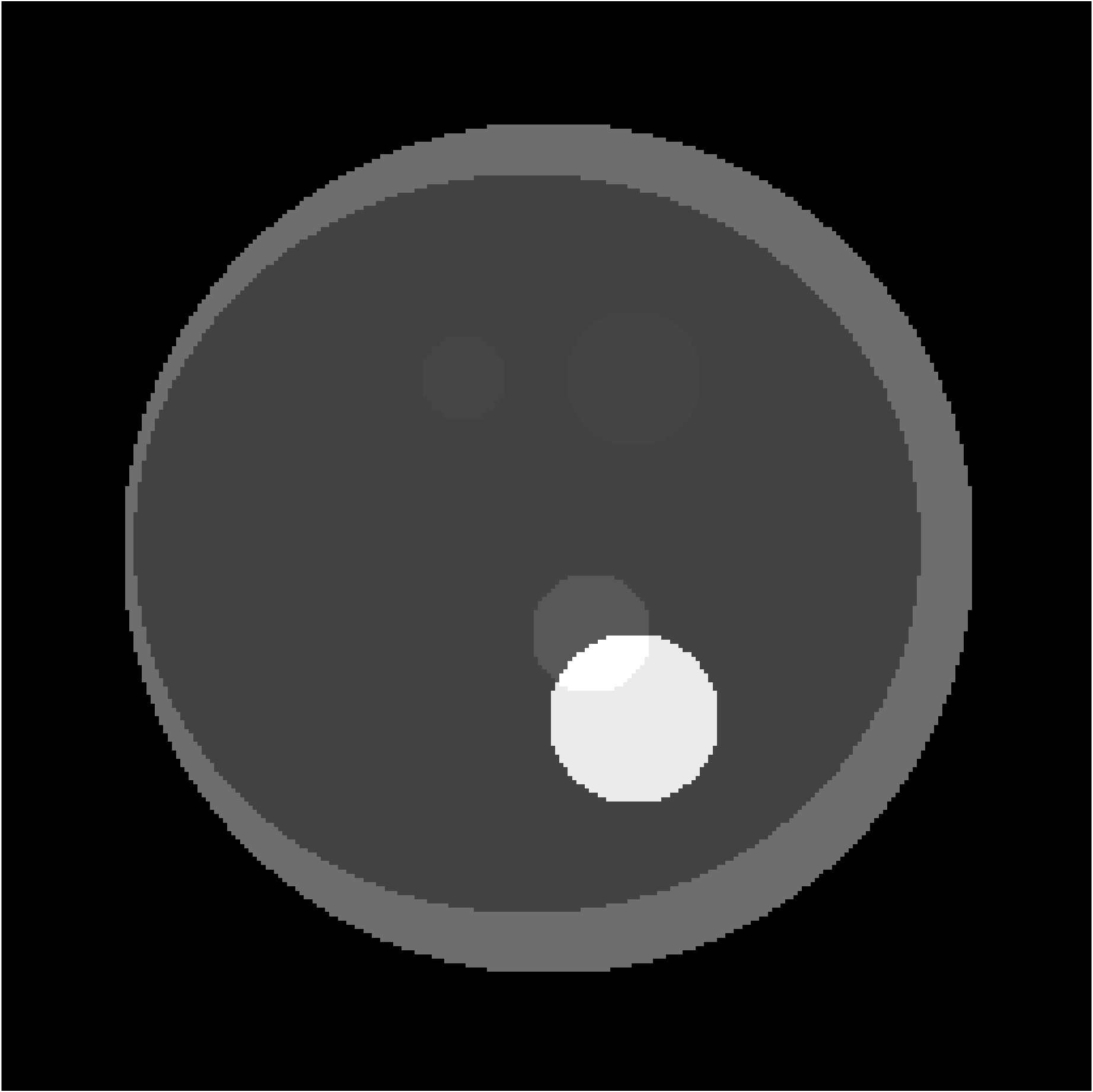}
\end{tabular}
\begin{tabular}{cccccc}
\includegraphics[height = 0.13\textwidth, width = .13\textwidth]{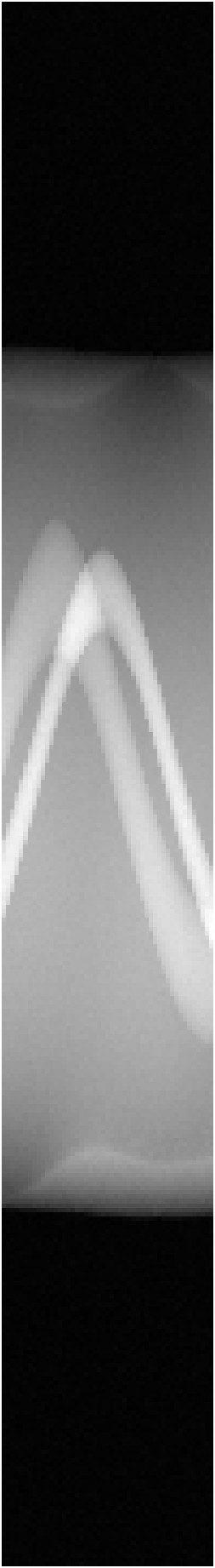} &
\includegraphics[height = 0.13\textwidth, width = .13\textwidth]{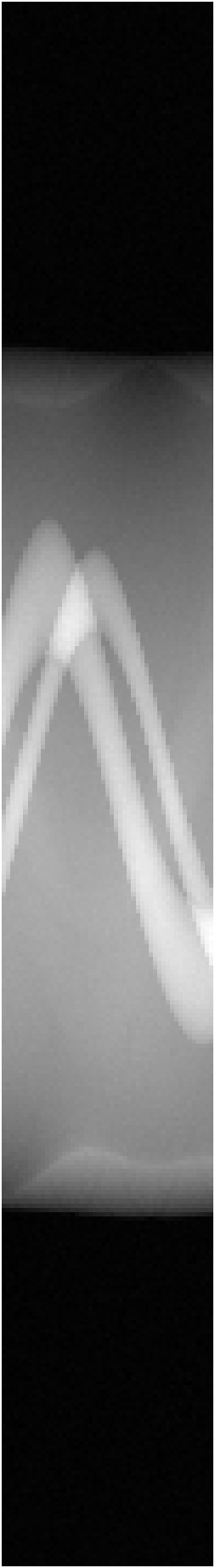} &   
\includegraphics[height = 0.13\textwidth, width = .13\textwidth]{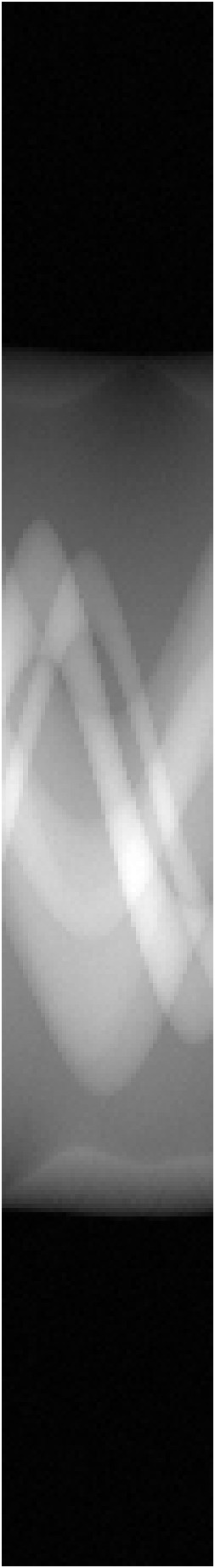} &
\includegraphics[height = 0.13\textwidth, width = .13\textwidth]{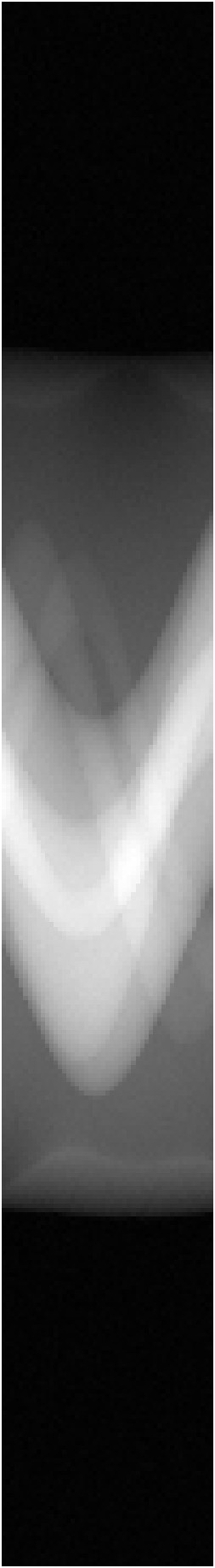} &
\includegraphics[height = 0.13\textwidth, width = .13\textwidth]{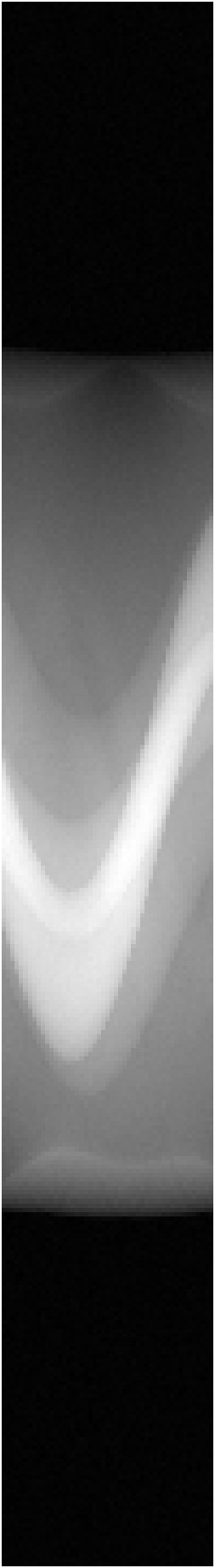} &
\includegraphics[height = 0.13\textwidth, width = .13\textwidth]{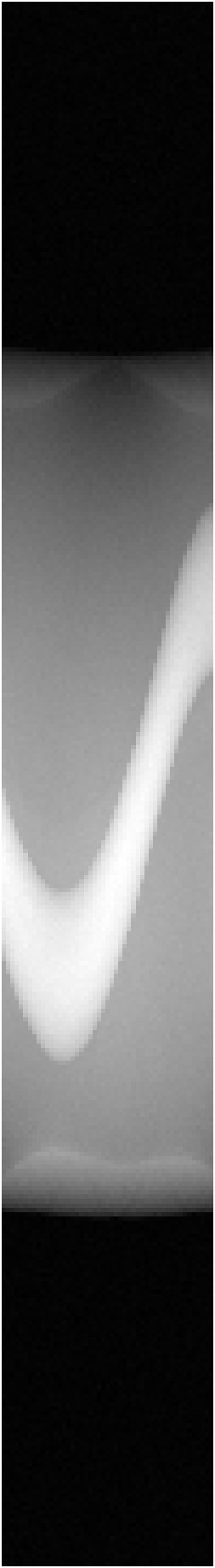} 
\end{tabular}
\caption{Dynamic PAT test problem. First row: true images of size $256\times 256$
at time steps $1, 10, 20, 30, 40, 50$. Second row: sinograms of size $362 \times 49$ at the same time steps. }
\label{Fig: PATtrueandsino}
\end{figure}

\paragraph{Comparison of LM-MM-GKS with MM-GKS and a restarted MM-GKS}
In this example, we consider a problem with 
$n_t = 50$ time steps. Each image is $256\times 256$ pixels, and it is composed of the superposition of six circles. 
A sample of the true phantoms, at time steps $i = 1, 10, 20, 30, 40, 50$, is given in the first row of Figure \ref{Fig: PATtrueandsino}. 
At time step $i$, measurements are taken at 
transducers placed at 49 equidistant angles with 
$5^{\circ}$ degree intervals
from $i^{\circ}$ to $(240+i)^{\circ}$. Hence, 
the set of angles at time step $i$ 
is obtained by shifting the previous set by 
$1^{\circ}$. At each transducer, we obtain measurements for $n_r = 362$ radii.
The measurements, $\bd_i \in \R^{362 \cdot 49}$, also known as sinograms, contain identically distributed and independent Gaussian noise with a noise level of $1\%$. 
A sample of sinograms at time steps $t = 1,10, 20, 30, 40, 50$ is shown in the second row of Figure \ref{Fig: PATtrueandsino}. The complete 
data vector, $\bd$, combining $50$ time steps, contains  
886,900 measurements, i.e., $\bd\in \R^{886,900}$. Note that the problem is severely underdetermined. Given $\bd$, the goal is to reconstruct an approximate solution, 
combining 50 time steps,
with 
$3,276,800$ pixels. The full problem can be 
stated as follows 
\begin{equation}\label{eq: blockF}
\underbrace{
\begin{bmatrix} \bA_{1} 
\\ &  \ddots\\
& & \bA_{n_t} 
\end{bmatrix}}_{\bA}
\underbrace{\begin{bmatrix}
    \bx_1\\
    \vdots\\
    \bx_{n_t}
\end{bmatrix}}_{\bx}
+
\underbrace{\begin{bmatrix}
    \be_1\\
    \vdots\\
    \be_{n_t}
\end{bmatrix}}_{\be}
=
\underbrace{
\begin{bmatrix}
    \bd_1\\
    \vdots\\
    \bd_{n_t}
\end{bmatrix}}
_{\bd}.
\end{equation}
Regularization leads to the large-scale coupled minimization problem 
\begin{equation}
    \min_{\bx \in \R^{n_x\cdot n_y \cdot n_t}} \|\bA\bx - \bd\|_2^2 + \lambda\|\Psi\bx\|_1,
\quad \text{with} \quad \Psi  = \begin{bmatrix} \bI_{n_t}\otimes\bI_{n_y}\otimes\bL_{x} \\  
\bI_{n_t}\otimes \bL_{y}\otimes\bI_{n_x} \\
\bL_{t}\otimes\bI_{n_y}\otimes\bI_{n_x} 
\end{bmatrix}, 
\end{equation}
where $\bL_d$ represents the discrete first derivative operator 
in the $d$-direction, for $d = x$ (vertical), $d=y$ (horizontal), and $d=t$ (time), resulting in $\Psi \in \R^{9,764,864  \times  3,276,800
}$. 
Other ways to define $\Psi$ for dynamic edge-preserving inverse problems can be found in \cite{pasha2021efficient} and,
for the Bayesian framework, in \cite{lan2023spatiotemporal}.

\begin{figure}[h!]
	\centering
	\begin{minipage}{0.15\textwidth}
		\centering
		\includegraphics[width=1\textwidth]{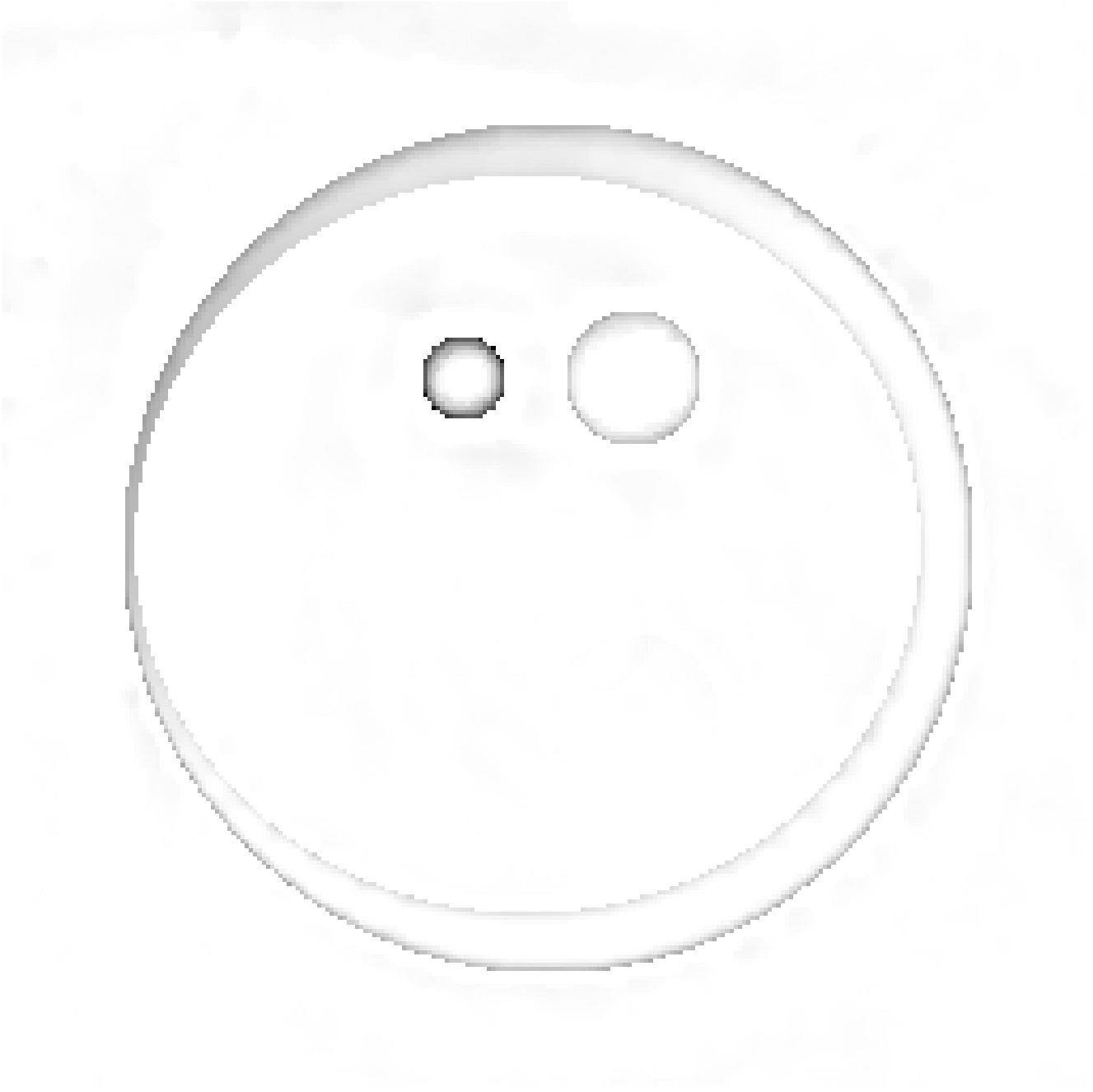}\\
	\end{minipage}
	\begin{minipage}{0.15\textwidth}
		\centering
		\includegraphics[width=\textwidth]{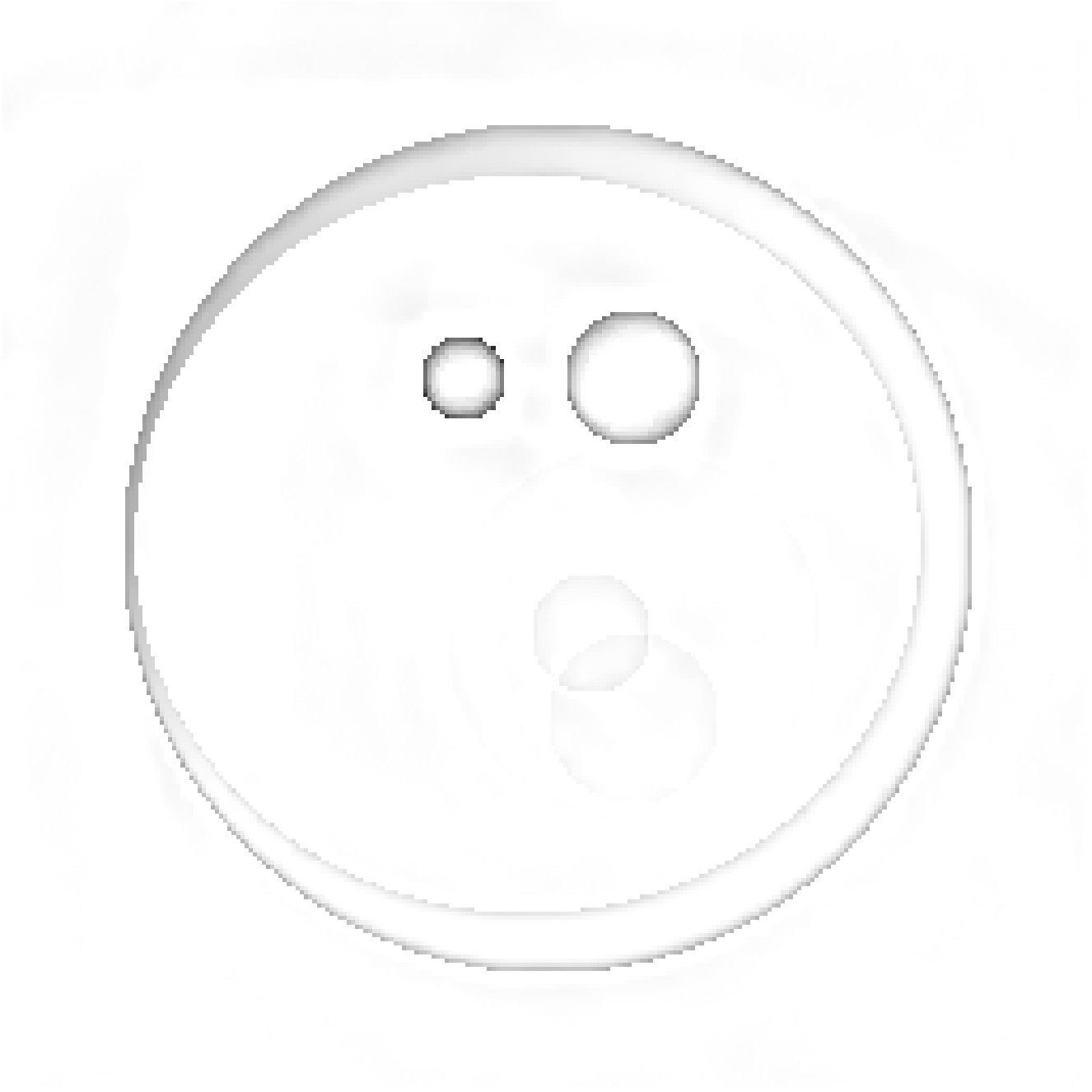}\\
	\end{minipage}
	\begin{minipage}{0.15\textwidth}
		\centering
		\includegraphics[width=\textwidth]{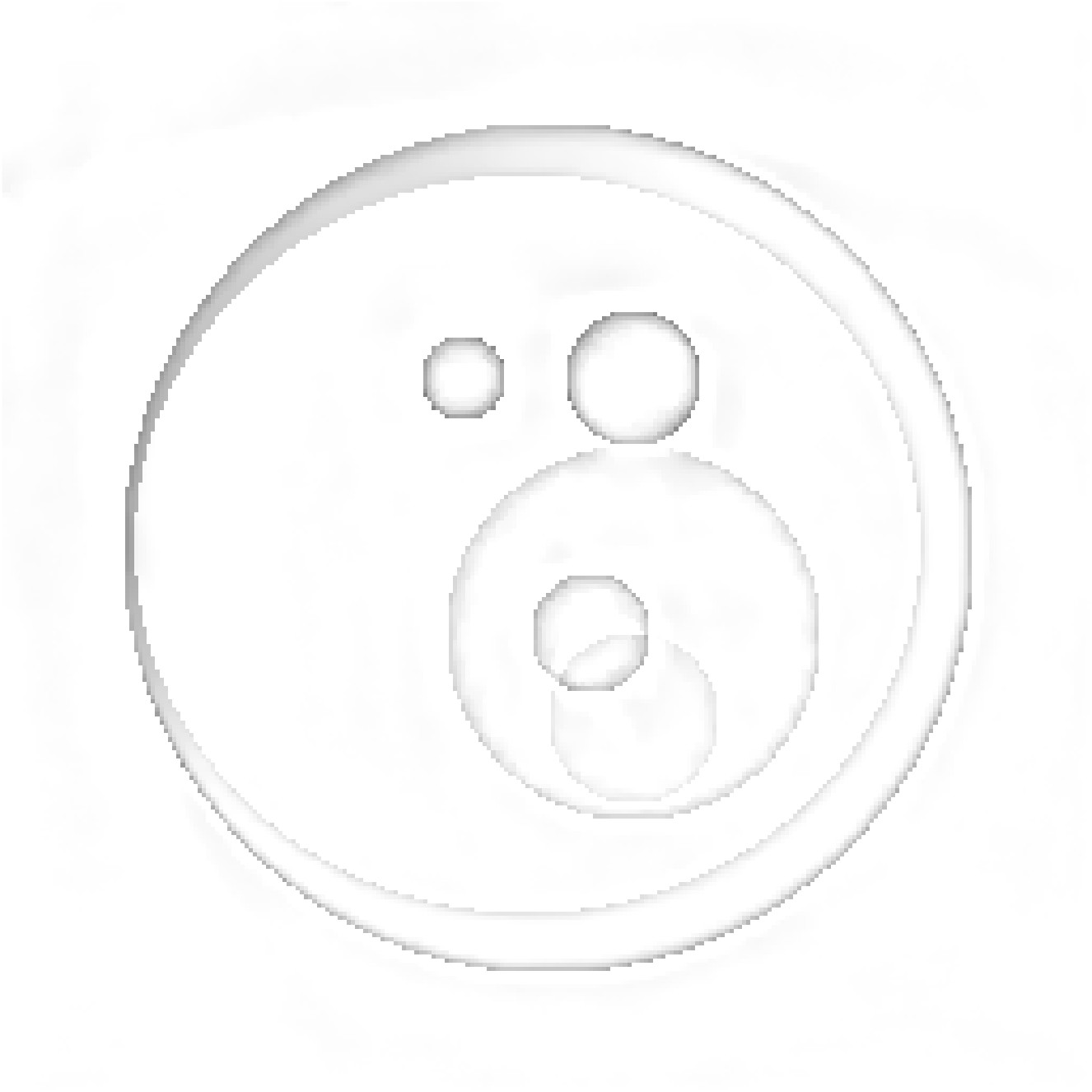}\\
	\end{minipage}
		\begin{minipage}{0.15\textwidth}
		\centering
		\includegraphics[width=\textwidth]{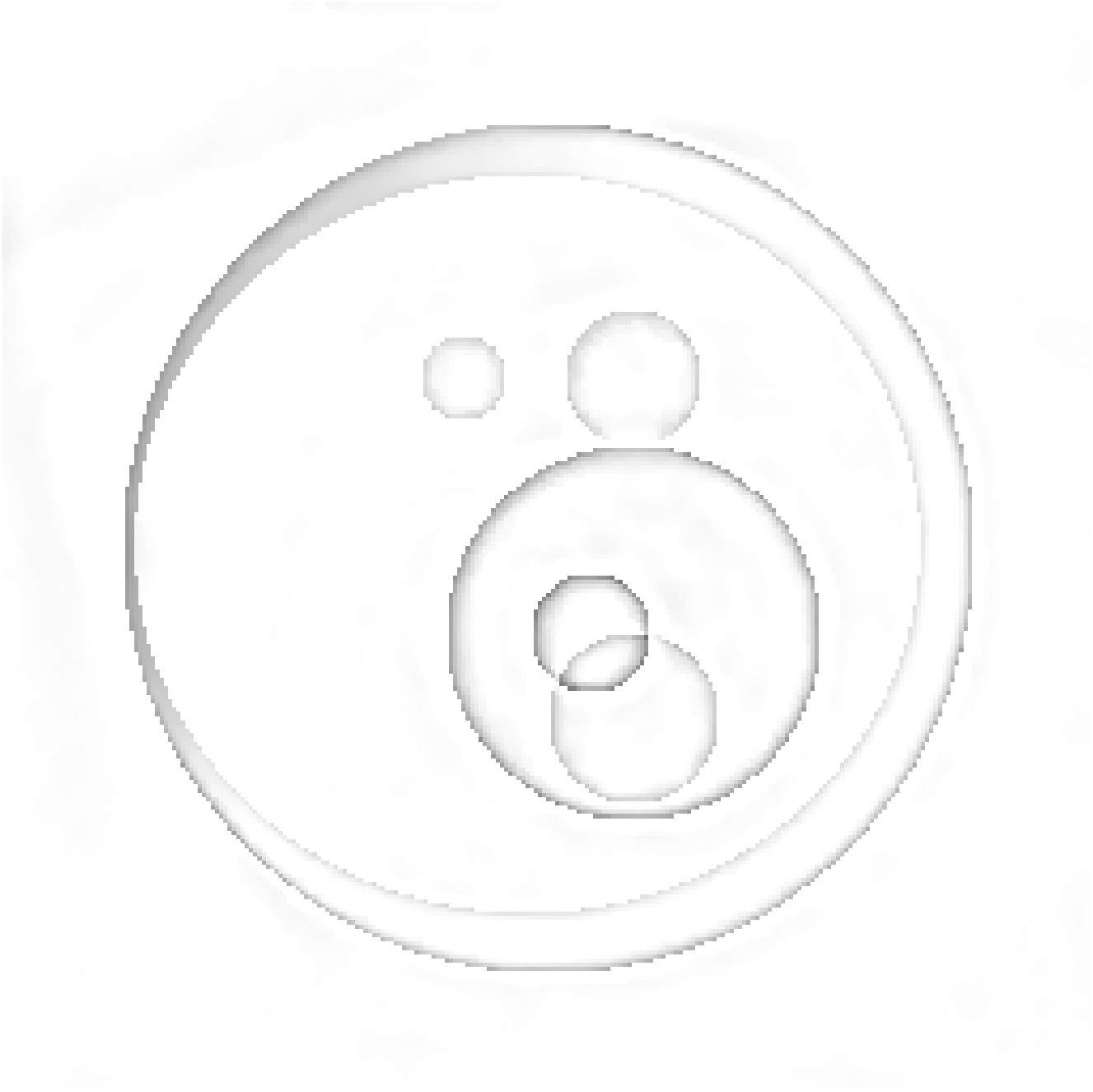}\\
	\end{minipage}
 \begin{minipage}{0.15\textwidth}
		\centering
		\includegraphics[width=\textwidth]{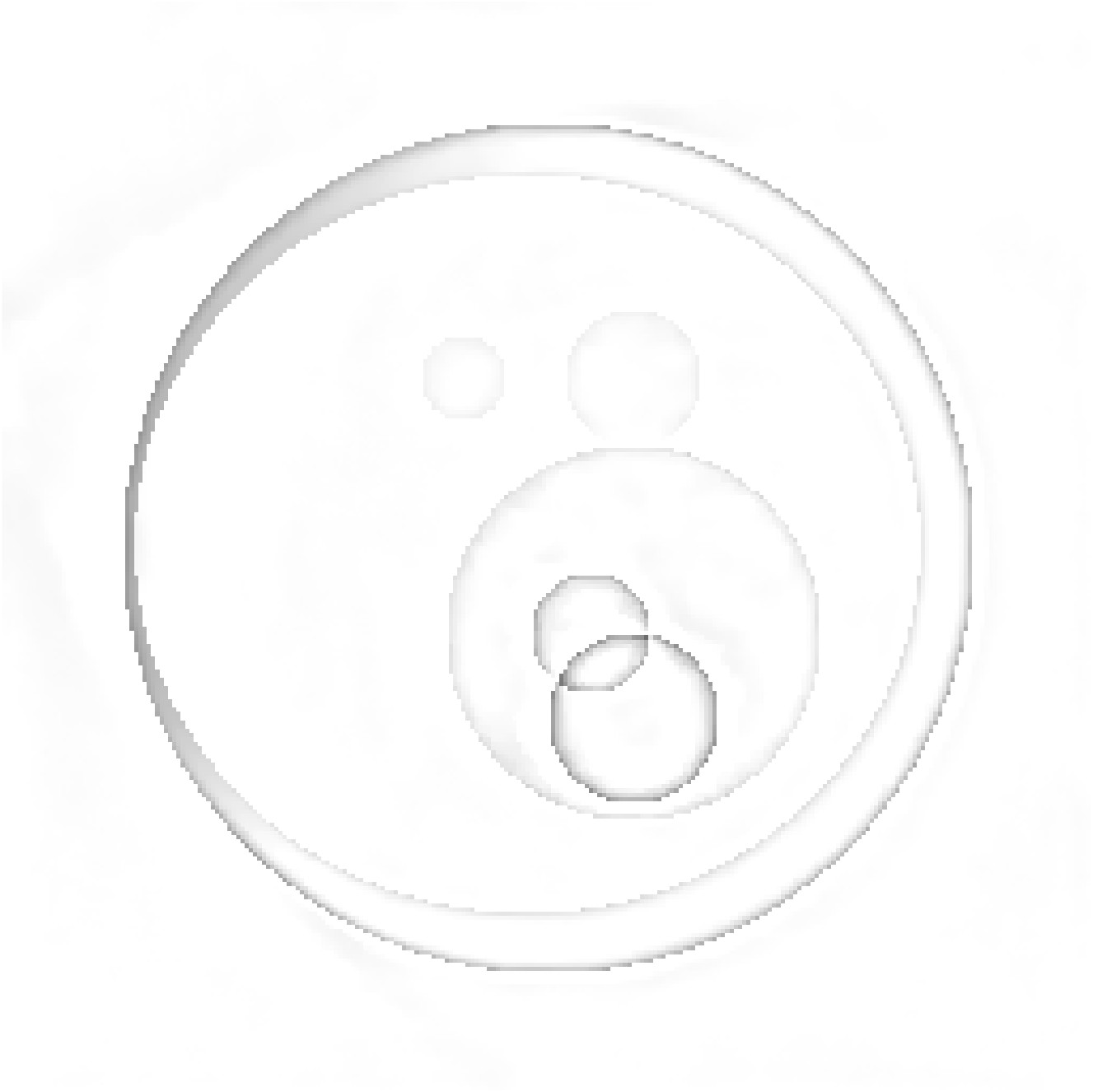}\\
	\end{minipage}
	\begin{minipage}{0.15\textwidth}
		\centering
		\includegraphics[width=\textwidth]{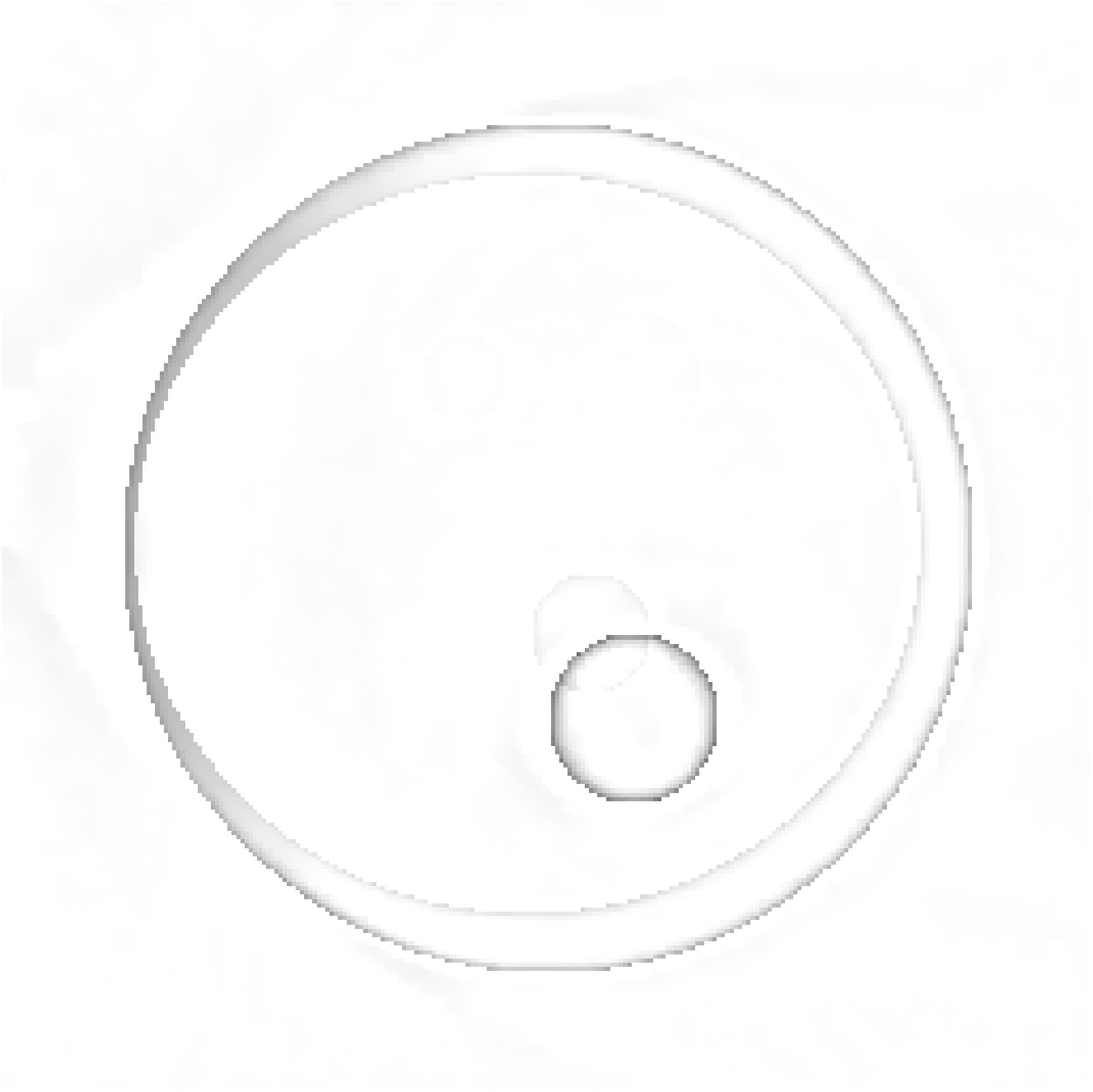}\\
	\end{minipage}
 
	\begin{minipage}{0.15\textwidth}
		\centering
		\includegraphics[width=\textwidth]{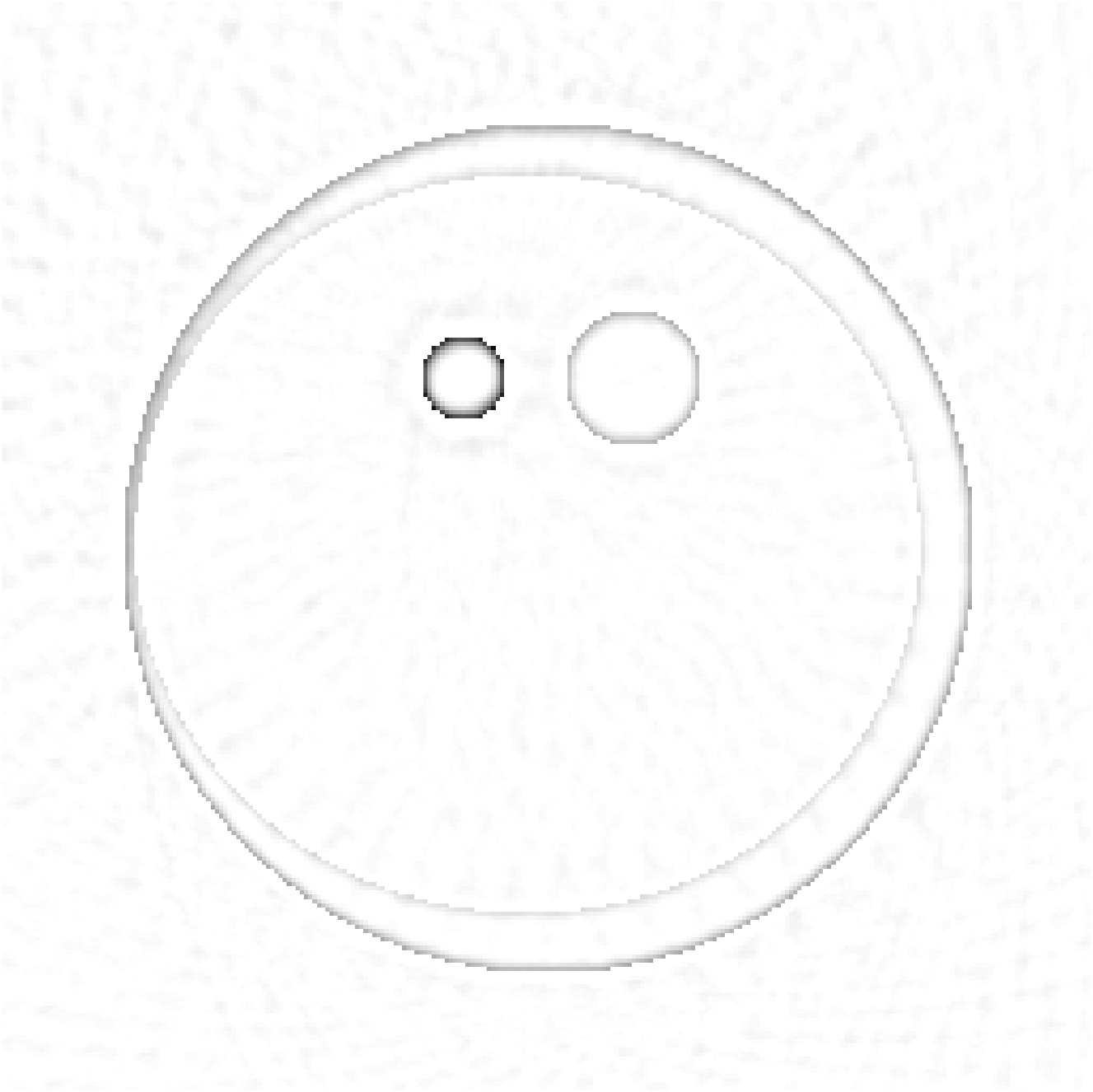}\\
	\end{minipage}
	\begin{minipage}{0.15\textwidth}
		\centering
		\includegraphics[width=\textwidth]{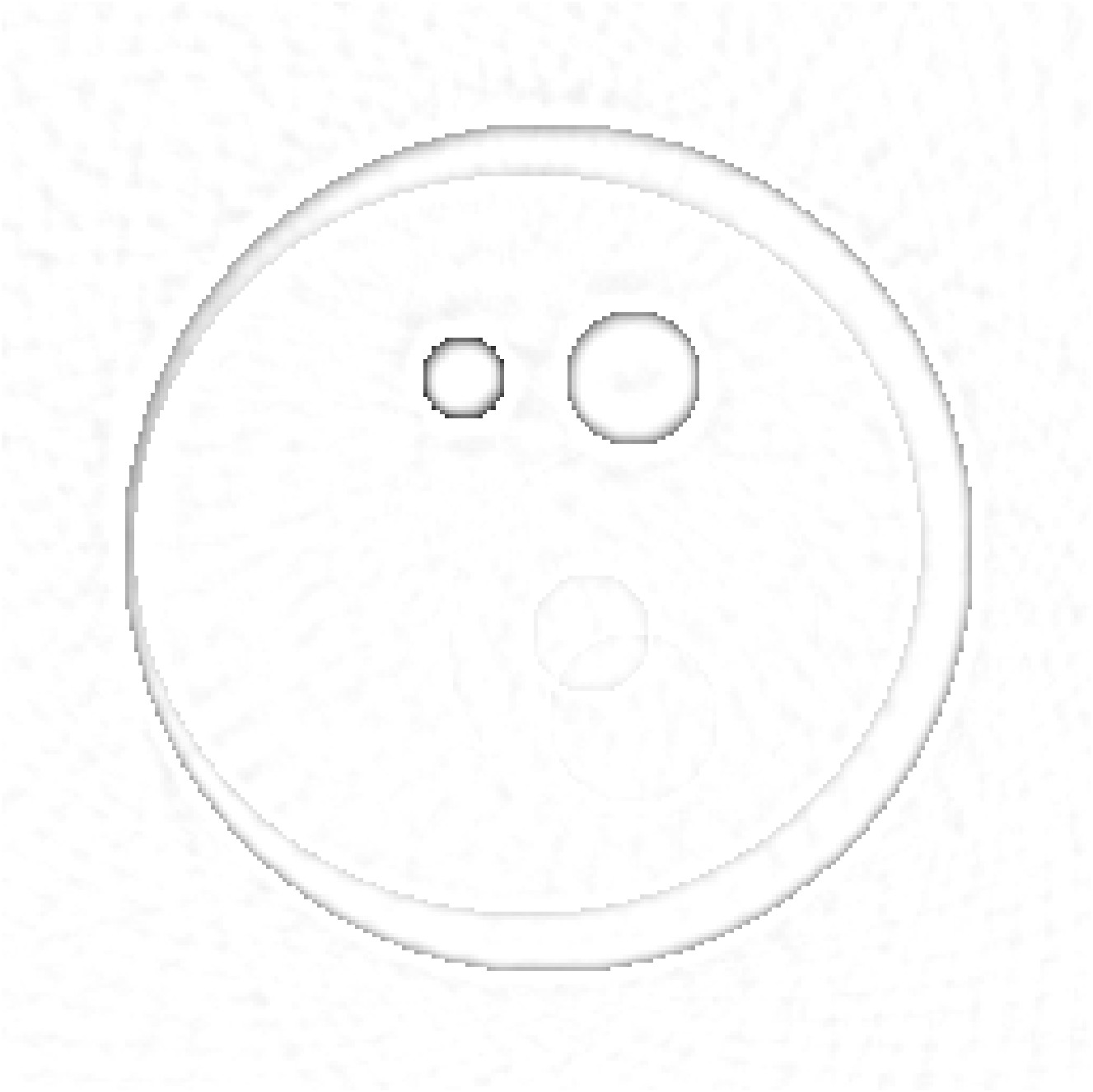}\\
	\end{minipage}
	\begin{minipage}{0.15\textwidth}
		\centering
		\includegraphics[width=\textwidth]{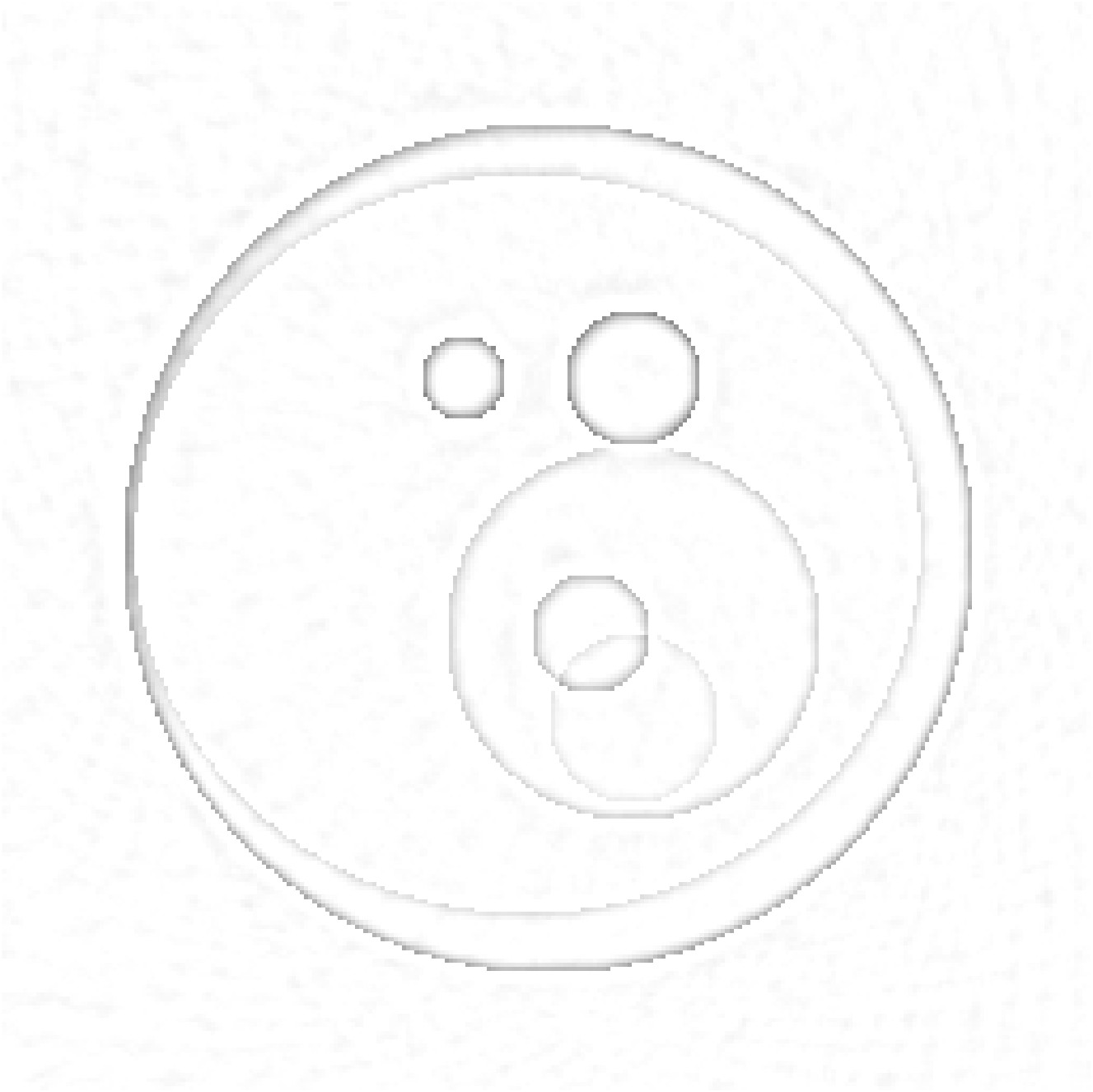}\\
	\end{minipage}
		\begin{minipage}{0.15\textwidth}
		\centering
		\includegraphics[width=\textwidth]{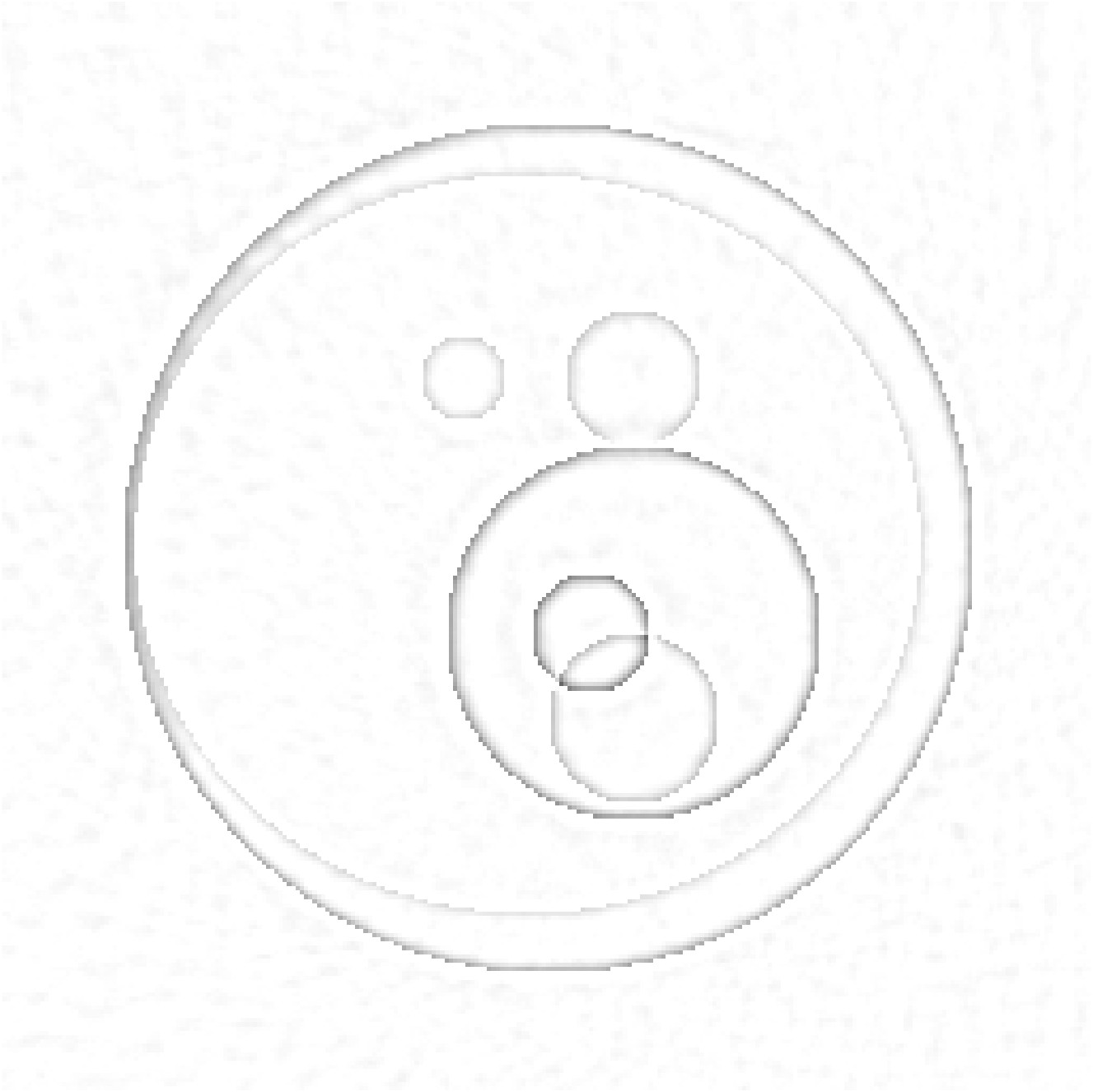}\\
	\end{minipage}
		\begin{minipage}{0.15\textwidth}
		\centering
		\includegraphics[width=\textwidth]{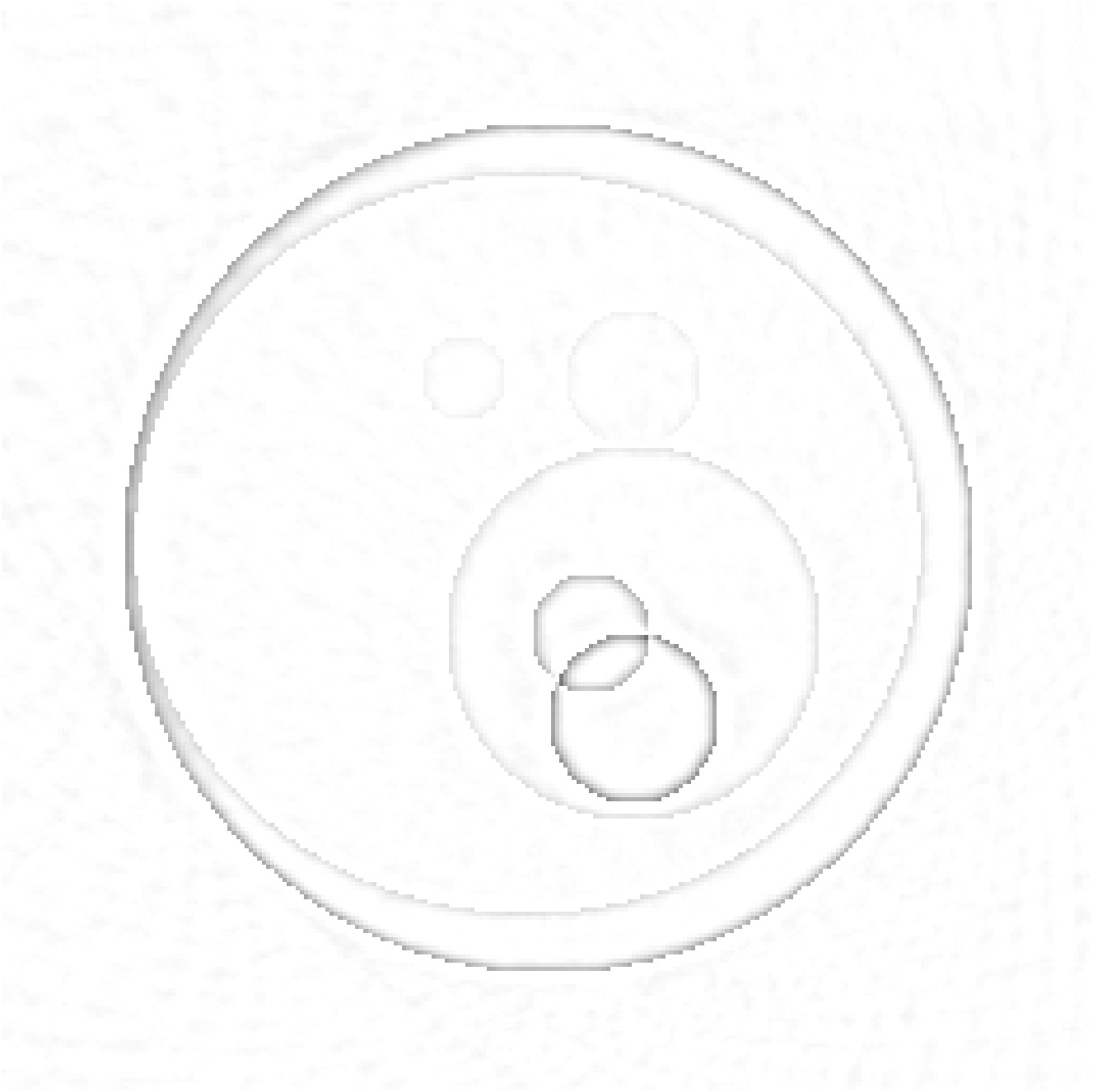}\\
	\end{minipage}
	\begin{minipage}{0.15\textwidth}
		\centering
		\includegraphics[width=\textwidth]{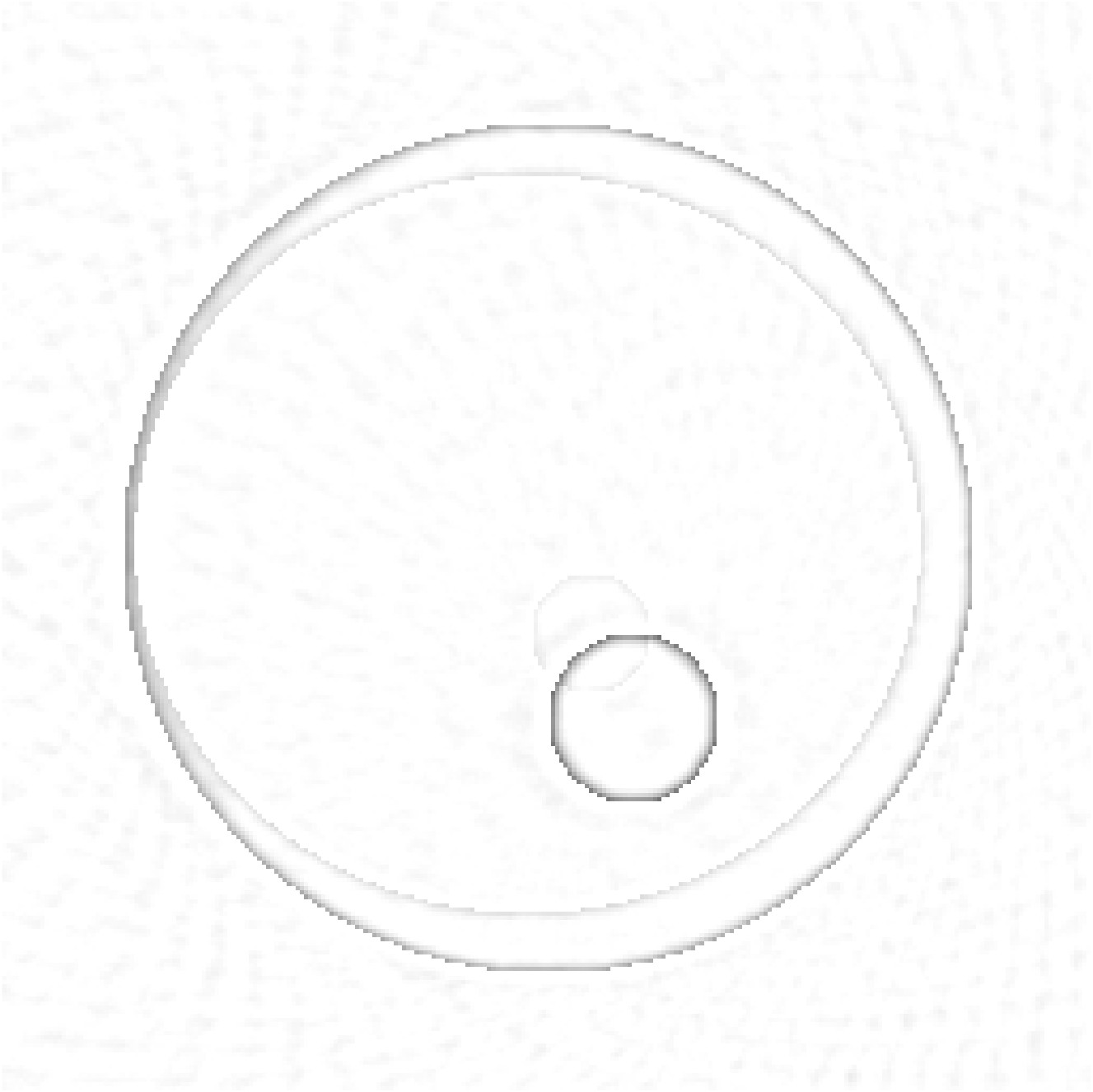}\\
	\end{minipage}

 	\begin{minipage}{0.14\textwidth}
		\centering
		\includegraphics[width=\textwidth]{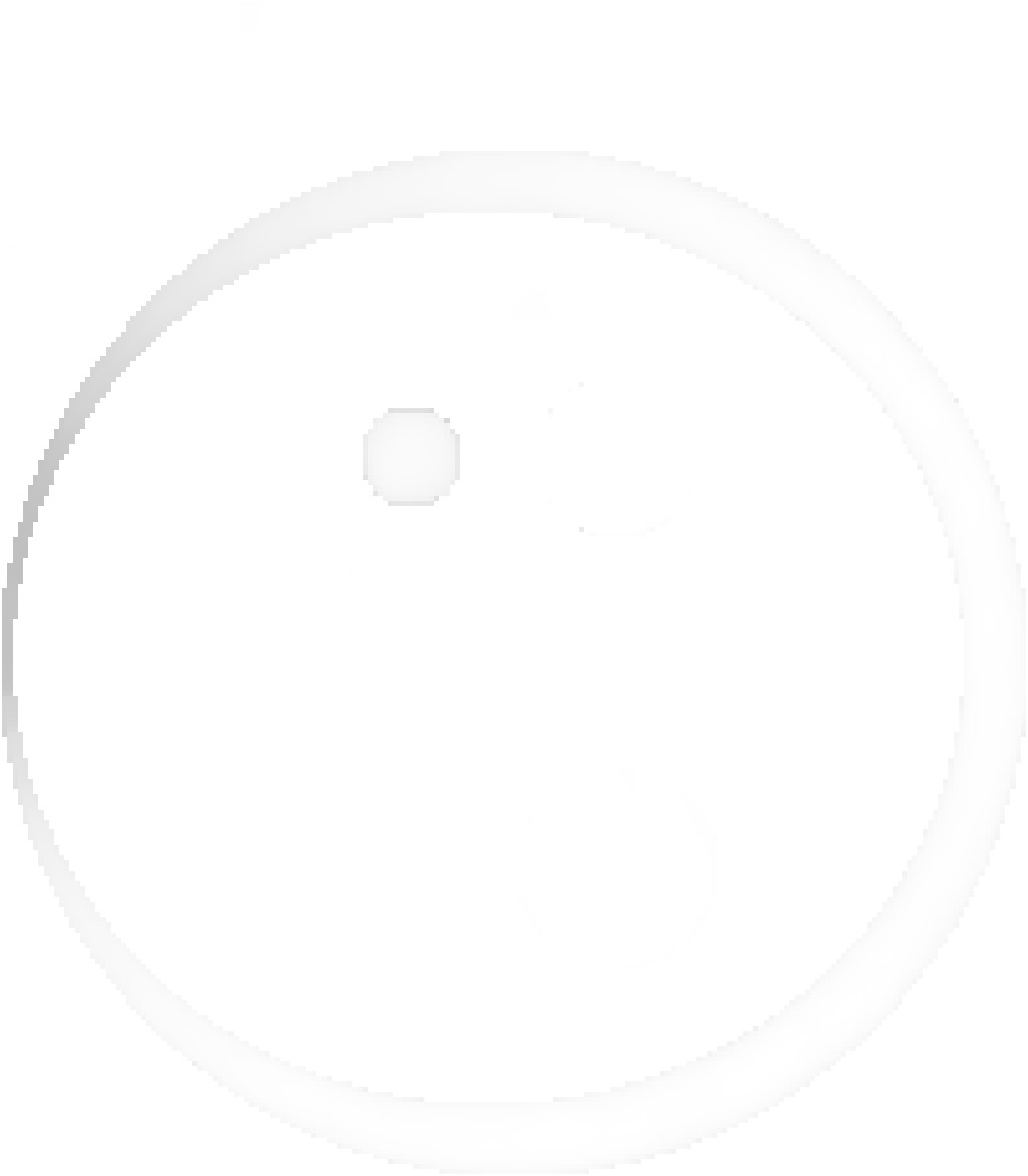}\\
	\end{minipage}
	\begin{minipage}{0.14\textwidth}
		\centering
		\includegraphics[width=\textwidth]{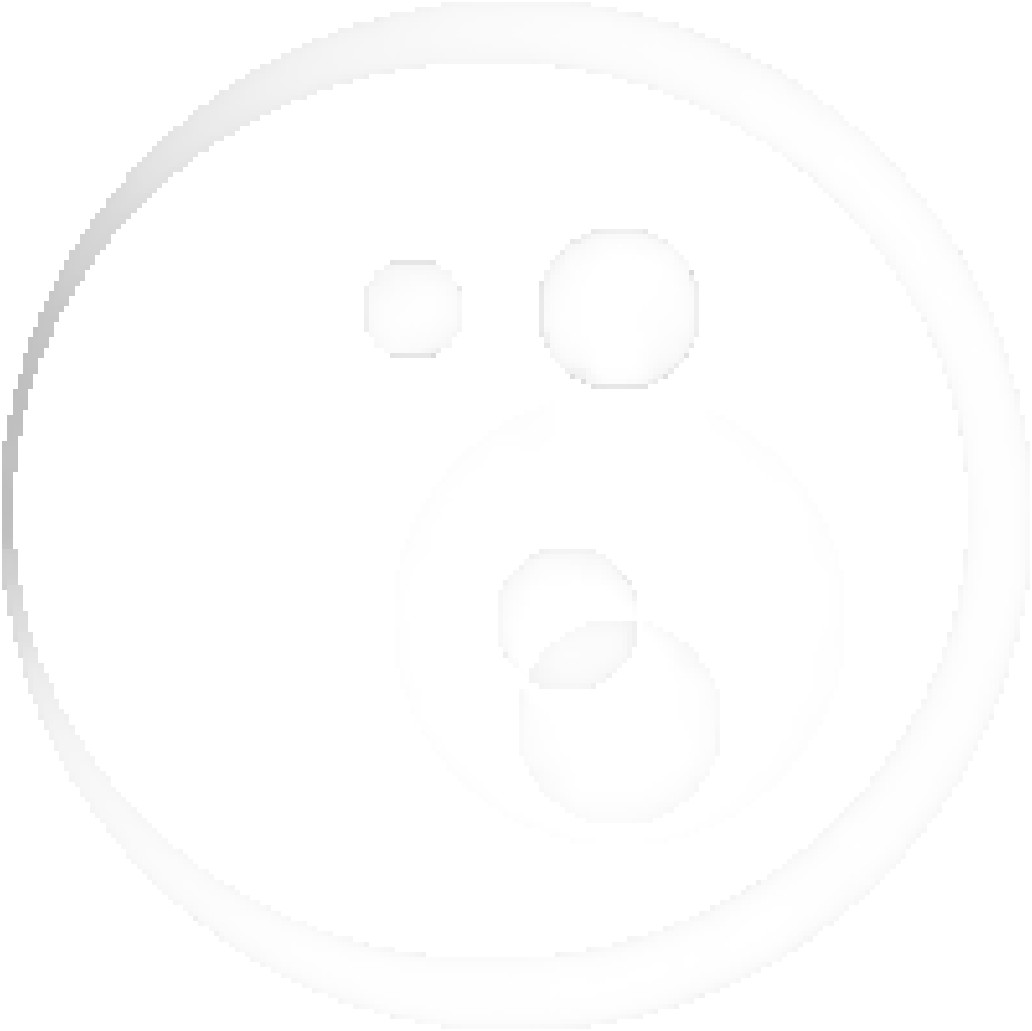}\\
	\end{minipage}
	\begin{minipage}{0.14\textwidth}
		\centering
		\includegraphics[width=\textwidth]{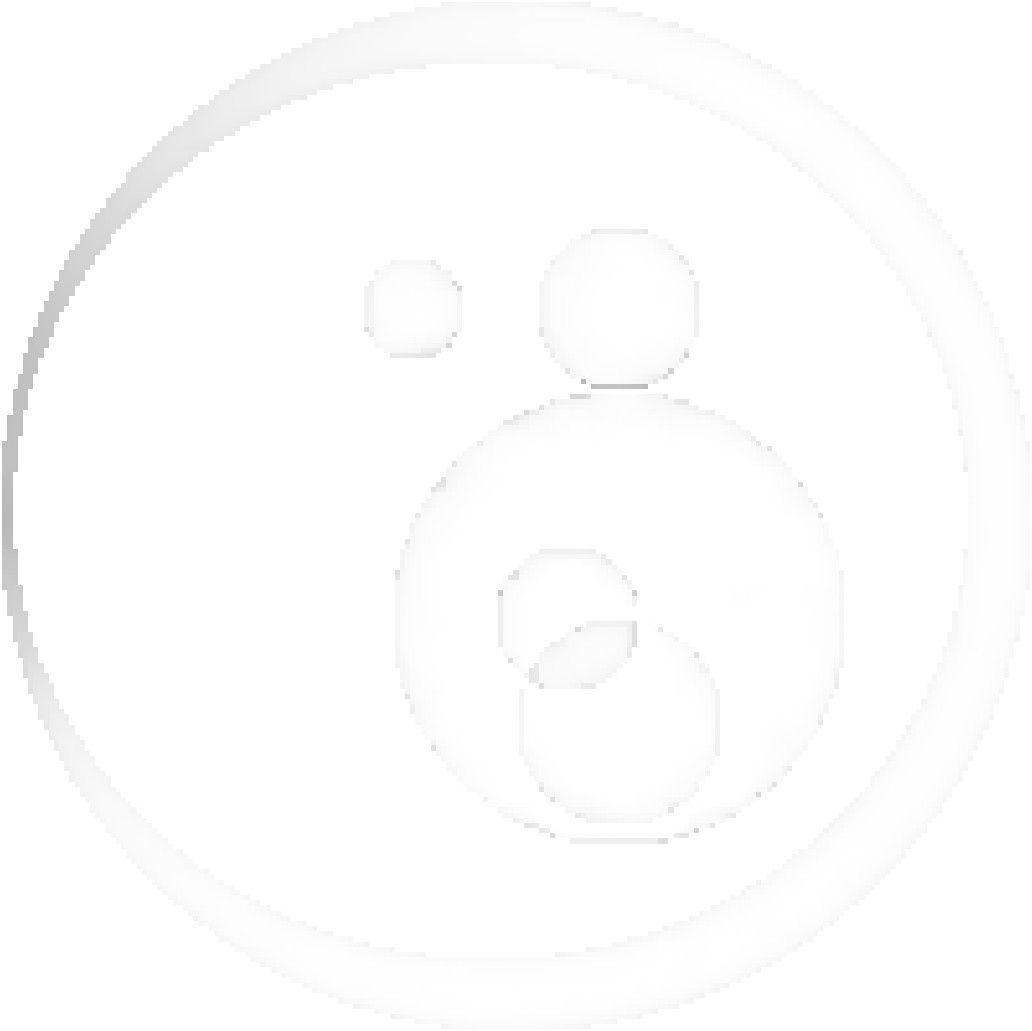}\\
	\end{minipage}
		\begin{minipage}{0.14\textwidth}
		\centering
		\includegraphics[width=\textwidth]{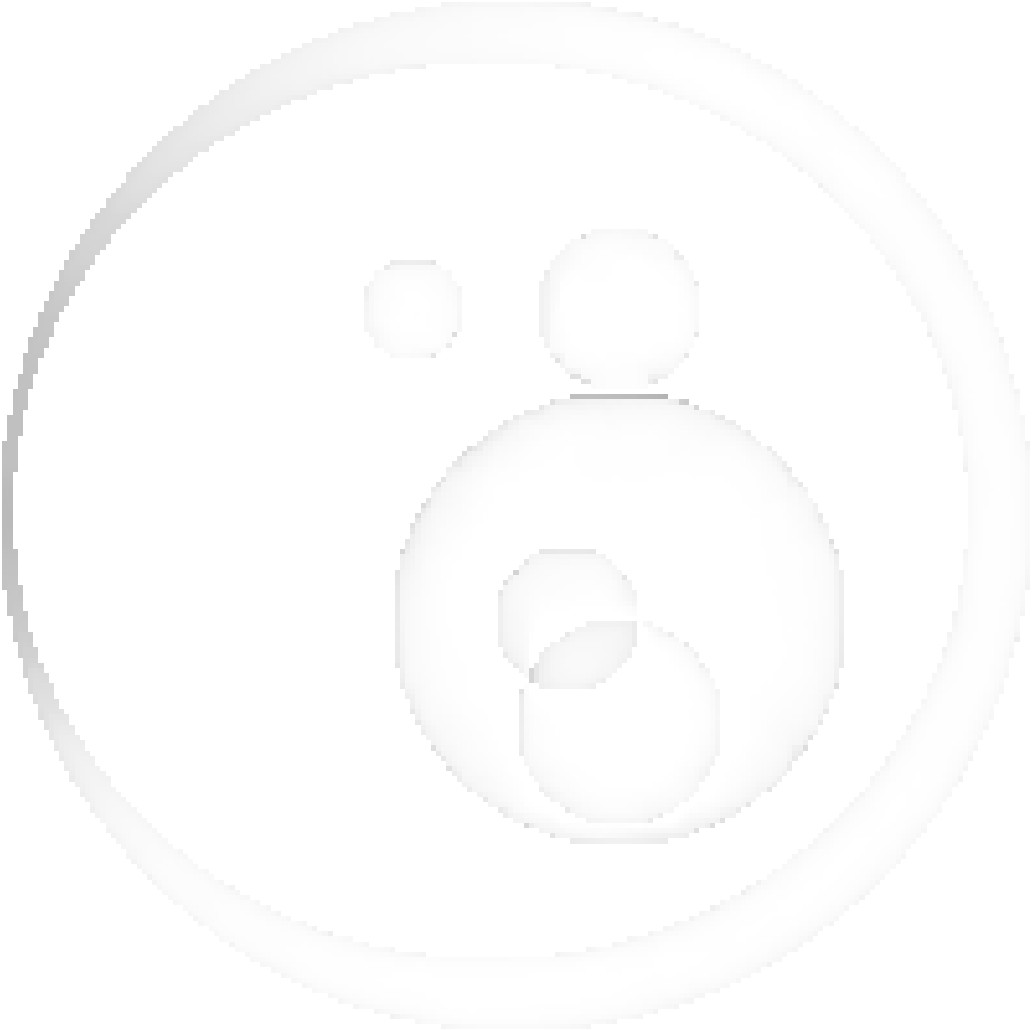}\\
	\end{minipage}
		\begin{minipage}{0.14\textwidth}
		\centering
		\includegraphics[width=\textwidth]{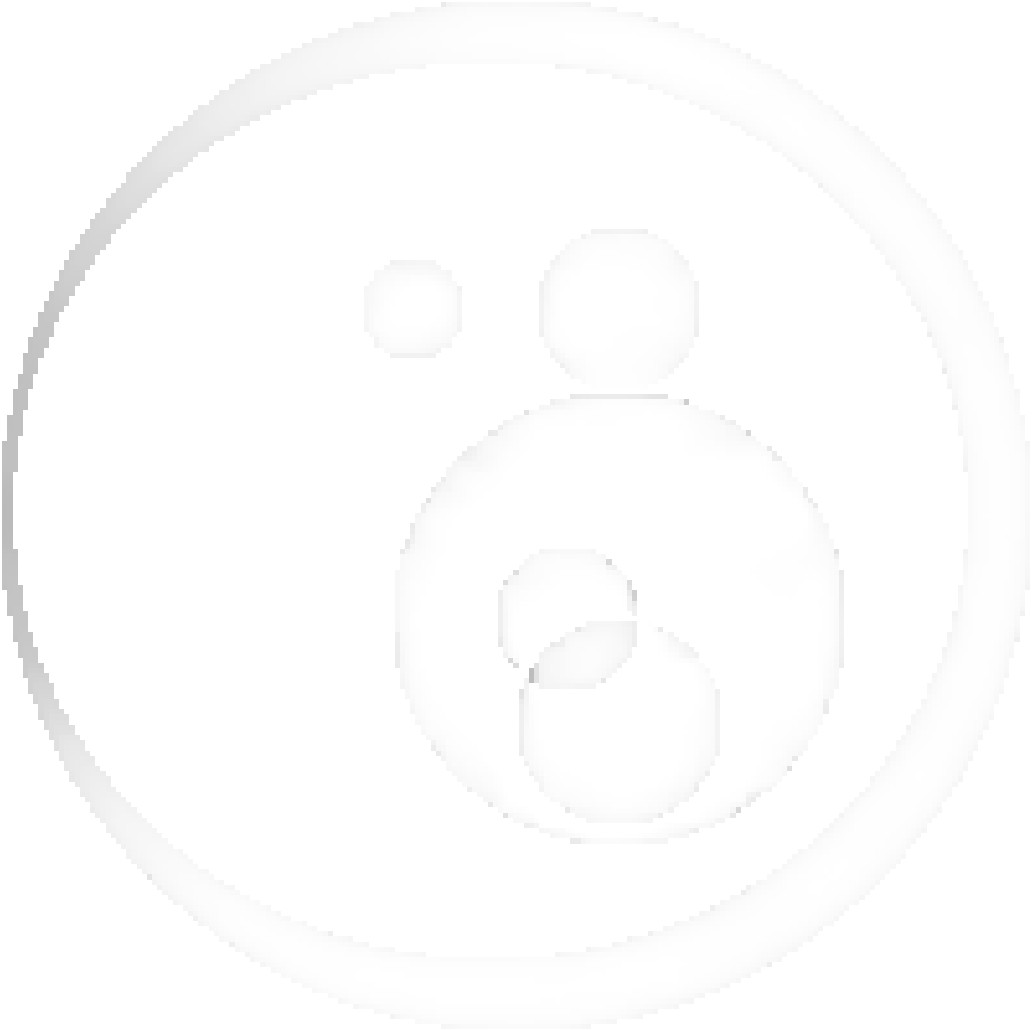}\\
	\end{minipage}
	\begin{minipage}{0.14\textwidth}
		\centering
		\includegraphics[width=\textwidth]{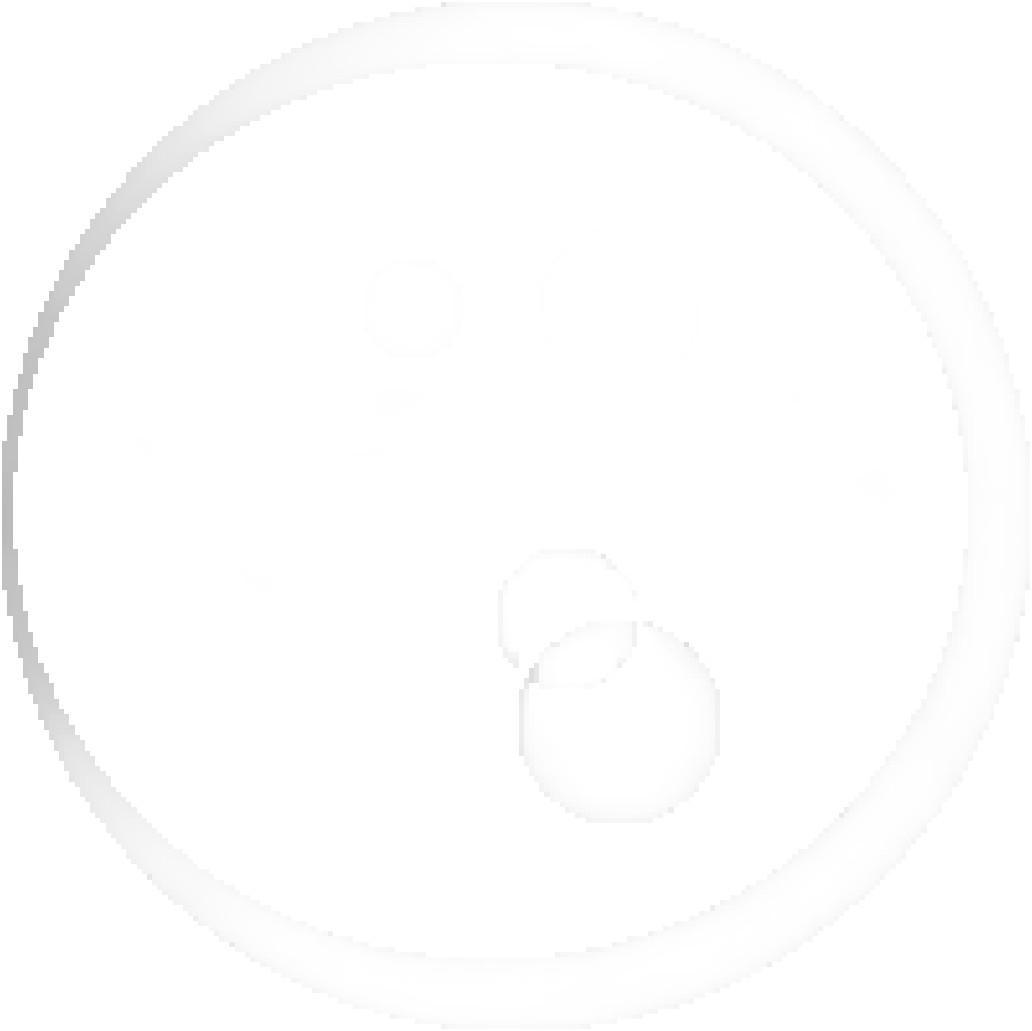}\\
	\end{minipage}
\caption{Dynamic PAT test problem ($\sigma = 1\%$).
Error images in the inverted color map at time steps $t=1, 10, 20, 30, 40, 50$ from left to right for MM-GKS after 15 iterations {(RRE $= 0.1322$, row 1)},
MM-GKS$_{\rm res}$ {(RRE $= 0.1264$, row 2)}, and
LM-MM-GKS {(RRE $= 0.0491$, row 3)}. }
\label{fig: PATRec}
\end{figure}

\begin{table}[ht!]
\centering
\small
\setlength{\tabcolsep}{5pt}
\begin{tabular}{@{} c ccc ccc ccc @{}}
\toprule
 & \multicolumn{3}{c}{MM-GKS$_{15}$}
 & \multicolumn{3}{c}{MM-GKS$_{\rm res}$}
 & \multicolumn{3}{c}{LM-MM-GKS} \\
\cmidrule(lr){2-4} \cmidrule(lr){5-7} \cmidrule(lr){8-10}
$\sigma$ & RRE & SSIM & PSNR
         & RRE & SSIM & PSNR
         & RRE & SSIM & PSNR \\
\midrule
0.1\% & 0.1264 & 0.9285 & 33.86
      & 0.0973 & 0.8849 & 36.13
      & 0.0476 & 0.9841 & 42.33 \\[2pt]
0.5\% & 0.1279 & 0.9254 & 33.75
      & 0.1099 & 0.8992 & 35.07
      & 0.0486 & 0.9836 & 42.16 \\[2pt]
1\%  & 0.1322 & 0.9162 & 33.46
      & 0.1264 & 0.9079 & 33.86
      & 0.0491 & 0.9833 & 42.07 \\[2pt]
5\%  & 0.1654 & 0.8453 & 31.52
      & 0.2381 & 0.8316 & 28.35
      & 0.0748 & 0.9647 & 38.41 \\
\bottomrule
\end{tabular}
\caption{Dynamic PAT test problem.
RRE, SSIM, and PSNR for several noise levels ($\sigma$) for MM-GKS after 15 iterations (memory-limited), MM-GKS$_{\rm res}$, and LM-MM-GKS, both after 500 iterations.}
\label{Table: table3}
\end{table}
In this example, we compare our proposed method LM-MM-GKS with the original MM-GKS method \cite{huang2017majorization}  and the recently proposed 
restarted MM-GKS method \cite{buccini2023limited}, 
denoted by MM-GKS$_{\rm res}$
for the full problem. 
For all methods, we set the maximum number of iterations to $500$, where for each method, one iteration corresponds to the addition of a new basis vector to the solution subspace.
Notice that $500$ iterations of MM-GKS requires storing and computing the QR factorizations for $\bA\bV_{500}$ and $\Psi\bV_{500}$ with vectors of size $886,900$
for $\bA\bV_{500}$ and size
$9,764,864$ for $\Psi\bV_{500}$. Even for this relatively small image size of $256 \times 256$ pixels, storing these matrices requires approximately $55$--$60$ GB of memory.
For more realistic image sizes of $1024 \times 1024$ pixels per time step, the storage for $\bPsi\bV_{500}$ alone would exceed $600$ GB, making MM-GKS prohibitively memory-intensive without compression or recycling.
To compare the behavior of the three  methods under severe memory constraints, which would be the case for (very) large problems, we assume that at most $15$ solution space basis vectors can be stored.
Hence, we show the reconstructions
from MM-GKS at $15$ iterations which we assume to be the hypothetical memory limit for this example.
For the compression in LM-MM-GKS we use the tSVD and we compress to $k_{\min}=5$ vectors, and we set $k_{\max} = 15$.
We also set the restart parameter for MM-GKS$_{\rm res}$ to $15$. LM-MM-GKS and MM-GKS$_{\rm res}$ can carry out 500 iterations, as they either restart or compress the solution space when the maximum of 15 solution space basis vectors is reached.
For MM-GKS$_{\rm res}$, we use the publicly available software package \cite{buccini2024software} without modifications. We note that the GCV implementation in this software did not perform well on this problem, so we use the DP for the regularization parameter selection in MM-GKS$_{\rm res}$, with $\epsilon = 1$ as recommended in the software documentation. LM-MM-GKS and MM-GKS use GCV with $\epsilon = 10^{-3}$. Python codes for MM-GKS can be found in 
\cite{pasha2025trips}.
We show the error images in the inverted colormap at time steps $t = 1$, $10$, $20$, $30$, $40$, $50$ in Figure \ref{fig: PATRec} for MM-GKS at 15 iterations (first row),  MM-GKS$_{\rm res}$(second row), and LM-MM-GKS (third row). In Table
\ref{Table: table3} we present
the RRE, SSIM, and PSNR for all three methods for noise levels varying from $0.1-5\%$. 

Table~\ref{Table: table3} shows that LM-MM-GKS consistently achieves RREs that are roughly $2.5$ to $3$ times smaller than those of MM-GKS$_{15}$ and MM-GKS$_{\rm res}$ across all noise levels, with correspondingly higher SSIM and PSNR values.

\section{Conclusions and future work}\label{sec: conclusion}
In this paper, we propose a limited memory technique for MM-GKS that 
periodically compresses the solution subspace while retaining information 
that is important for convergence,
using a fixed amount of memory, independent of the number of iterations, effectively eliminating otherwise serious memory constraints.  
Furthermore, making a crucial change to the search space update strategy, we are now able to prove convergence even under limited memory constraints. This provides a substantial improvement on the convergence theory for Majorization-Minimization methods using a subspace
approximation. 
We also propose a variant of LM-MM-GKS that is effective for streaming problems, that is, for scenarios where 
large-scale data exceeds memory limitations, or where all the data to be processed is not available at once.
We demonstrate that our proposed approach is able to effectively integrate new (incoming) data with selected (older) data to 
improve the approximate solution. 
In particular, we show that these techniques are effective in the context of nonlinear edge-preserving image reconstruction methods
that may require many iterations, updating the regularization operator
and dynamically selecting the regularization parameter at each iteration at a low computational cost.
Numerical examples from a wide range of applications, such as CT, image deblurring, and time-dependent inverse problems, illustrate the effectiveness of the proposed methods.
Important future work involves (1) more efficient implementations of these methods, in terms of memory, computational cost, and data movement, and (2) an extension of the method and the theory that will  allow us to prove convergence when determining the regularization parameter dynamically. 

\section*{Acknowledgments}
MP acknowledges support from the NSF under awards No.\ 2202846 and DMS 2410699 and from 
the Isaac Newton Institute (INI) for Mathematical Sciences, Cambridge, for hospitality during the programme ``Rich and Nonlinear Tomography'' 
where partial work on this manuscript was undertaken.
EdS acknowledges support by the NSF under Award No. 2208470.
MK acknowledges support by the NSF under CCF-1934553 and DMS-2410698 and also thanks the Turner-Kirk Charitable Trust for the support provided by a Kirk Distinguished Visiting Fellowship to attend the INI programme where partial work on this manuscript was undertaken.

\begin{remark}
An earlier version of this manuscript appeared as a preprint \cite{ouroldpaper}.  This is a substantially revised version, which includes changes to the original algorithm, new numerical results, new theoretical results, and a 
new title for the paper.    
\end{remark}

\paragraph{Reproducibility} The MATLAB codes for reproducing the numerical experiments in this paper are available on GitHub \url{https://github.com/mpasha3/LM-MM-GKS}. 

\bibliographystyle{siamplain}
\bibliography{references}
\end{document}